\documentclass[a4paper,reqno]{amsart}

\usepackage[utf8]{inputenc}
\usepackage{microtype}

\usepackage{amssymb}
\usepackage{mathrsfs}
\usepackage{mathtools}
\usepackage{enumitem}

\usepackage{color}
\definecolor{gray}{rgb}{.75,.75,.75}
\usepackage{tikz}
\usetikzlibrary{topaths}
\usetikzlibrary{tikzmark}
\usetikzlibrary{intersections}
\usetikzlibrary{decorations.pathreplacing}

\usepackage{savesym}
\usepackage[all]{xy}
\savesymbol{cir}

\usepackage{hyperref}

\usepackage{etex}
\usepackage{diagrams}
\newarrow{Iso}{C}{-}{-}{-}{>>}

\newtheorem{theorem}{Theorem}[section]
\newtheorem*{theorem*}{Theorem}

\newtheorem{lemma}[theorem]{Lemma}

\newtheorem{proposition}[theorem]{Proposition}
\newtheorem*{proposition*}{Proposition}

\newtheorem{corollary}[theorem]{Corollary}
\newtheorem*{corollary*}{Corollary}
\newtheorem{conjecture}[theorem]{Conjecture}
\newtheorem*{conjecture*}{Conjecture}
\newtheorem{observation}[theorem]{Observation}

\newtheorem{question}[theorem]{Question}
\newtheorem*{question*}{Question}

\newtheorem*{main:finiteness_conjecture}{Question~\ref{ques:finiteness_conjecture}}
\newtheorem*{main:borel_thomp_pos_fin_props}{Theorem~\ref{thrm:borel_thomp_pos_fin_props}}
\newtheorem*{main:thomp_abels_Finfty}{Theorem~\ref{thrm:thomp_abels_Finfty}}
\newtheorem*{main:Vmock_Vloop_Floop_conj}{Conjecture~\ref{conj:Vmock_conj},~\ref{conj:Vloop_Floop_conj}}
\newtheorem*{main:borel_thomp_full_fin_props}{Theorem~\ref{thrm:borel_thomp_full_fin_props}}
\newtheorem*{main:thompson_construction}{Proposition~\ref{prop:product_of_simple}}
\newtheorem*{main:existence_thompson_mock_loop}{Theorem~\ref{thm:existence_thompson_mock},~\ref{thm:existence_thompson_loop}}

\theoremstyle{definition}
\newtheorem{definition}[theorem]{Definition}
\newtheorem{remark}[theorem]{Remark}
\newtheorem{example}[theorem]{Example}

\newcommand{\Z}{\mathbb{Z}}
\newcommand{\N}{\mathbb{N}}
\newcommand{\Nz}{\mathbb{N}_0}
\newcommand{\Nnz}{\mathbb{N}}
\newcommand{\Q}{\mathbb{Q}}
\newcommand{\R}{\mathbb{R}}
\newcommand{\calO}{\mathcal{O}}
\newcommand{\calP}{\mathcal{P}}
\newcommand{\calQ}{\mathcal{Q}}
\newcommand{\calG}{\mathcal{G}}

\newcommand{\Preposet}{\widetilde{\mathcal{P}}}
\newcommand{\Poset}{\mathcal{P}}
\newcommand{\symm}{S}
\newcommand{\mock}[1]{{#1}^{\text{mock}}}
\newcommand{\match}{\mathcal{M}}
\newcommand{\independence}{\operatorname{Ind}}

\DeclareMathOperator{\Stab}{Stab}

\DeclareMathOperator{\Hom}{Hom}
\DeclareMathOperator{\Symm}{Symm}
\DeclareMathOperator{\Diff}{Diff}
\DeclareMathOperator{\lk}{lk}
\newcommand{\st}{\overline{\operatorname{st}}}
\DeclareMathOperator{\dlk}{{\lk}{\downarrow}}

\DeclareMathOperator{\Tor}{Tor}

\DeclareMathOperator{\id}{id}

\DeclareMathOperator{\PGL}{PGL}
\newcommand{\abels}{\mathit{Ab}}
\newcommand{\LB}{\mathit{LB}}
\newcommand{\PLB}{\mathit{PLB}}
\DeclareMathOperator{\Aut}{Aut}
\DeclareMathOperator{\SAut}{\Sigma\!\Aut}
\DeclareMathOperator{\PSAut}{P\!\Sigma\!\Aut}

\newcommand{\F}{\text{F}}
\newcommand{\FP}{\text{FP}}

\newcommand{\PB}{\mathit{PB}}
\newcommand{\barB}{\bar{B}}

\newcommand{\Fmonoid}{\mathscr{F}}
\newcommand{\Hmonoid}{\mathscr{H}}

\newcommand{\forest}{E}
\newcommand{\altforest}{F}
\newcommand{\tree}{T}
\newcommand{\alttree}{U}
\newcommand{\clone}{\kappa}
\newcommand{\symmclone}{\varsigma}

\newcommand{\gen}[1]{\langle #1 \rangle}
\newcommand{\ceil}[1]{\lceil #1 \rceil}
\newcommand{\floor}[1]{\lfloor #1 \rfloor}
\newcommand{\abs}[1]{\lvert #1 \rvert}
\newcommand{\into}{\hookrightarrow}
\newcommand{\len}{\operatorname{len}}
\newcommand{\image}{\operatorname{im}}
\newcommand{\rk}{\operatorname{rk}}

\newcommand{\realize}[1]{\lvert{#1}\rvert}

\newcommand{\defeq}{\mathbin{\vcentcolon =}}
\newcommand{\eqdef}{\mathbin{=\vcentcolon}}

\newcommand{\OneEdge}{J}

\usepackage{ifthen}
\usepackage{xargs}
\newcommand{\optionalarg}[2]{
\ifthenelse{\equal{#2}{}}{%
#1}{%
#1(#2)}
}

\newcommand{\Thomp}[1]%
   {\optionalarg{\mathscr{T}}{#1}}                 

\newcommand{\Thomphat}[1]%
   {\optionalarg{\widehat{\mathscr{T}}}{#1}}                 
   
\newcommand{\Thkern}[1]%
   {\optionalarg{\mathscr{K}}{#1}}                 

\newcommand{\Stein}[1]%
   {\optionalarg{\mathscr{X}}{#1}}                 

\newcommand{\dlkmodel}[2]                          
   {\optionalarg{\mathscr{L}_{#2}}{#1}}

\newcommand{\Fbr}%
   {F_{\operatorname{br}}}                 

\newcommand{\Tbr}%
   {T_{\operatorname{br}}}                 

\newcommand{\Vbr}%
   {V_{\operatorname{br}}}                 
   
\newcommand{\Vmock}%
   {V_{\operatorname{mock}}}                 
   
\newcommand{\Vloop}%
   {V_{\operatorname{loop}}}                 
   
\newcommand{\Floop}%
   {F_{\operatorname{loop}}}                 
   
\newcommand{\Fhat}%
   {\widehat{F}}       

\newcommand{\GlobalField}{k}
\newcommand{\CoefficientField}{\ell}

\begin{document}\parindent=0pt

\title[Thompson groups for systems and finiteness properties]{Thompson groups for systems of groups,\\and their finiteness properties}
\date{\today}


\author[S.~Witzel]{Stefan Witzel}
 \address{Faculty of Mathematics, Bielefeld University, Postfach 100131, 33501 Bielefeld, Germany}
 \email{switzel@math.uni-bielefeld.de}

\author[M.~C.~B.~Zaremsky]{Matthew C.~B.~Zaremsky}
\address{Department of Mathematics and Statistics, University at Albany (SUNY), Albany, NY 12222, USA}
\email{mzaremsky@albany.edu}

\begin{abstract}
 We describe a procedure for constructing a generalized Thompson group out of a family of groups that is equipped with what we call a cloning system. The previously known Thompson groups $F$, $V$, $\Vbr$ and $\Fbr$ arise from this procedure using, respectively, the systems of trivial groups, symmetric groups, braid groups and pure braid groups.
 
 We give new examples of families of groups that admit a cloning system and study how the finiteness properties of the resulting generalized Thompson group depend on those of the original groups. The main new examples here include upper triangular matrix groups, mock reflection groups, and loop braid groups. For generalized Thompson groups of upper triangular matrix groups over rings of $S$-integers of global function fields, we develop new methods for (dis-)proving finiteness properties, and show that the finiteness length of the generalized Thompson group is exactly the limit inferior of the finiteness lengths of the groups in the family.
\end{abstract}

\maketitle
\thispagestyle{empty}

\section*{Introduction}\label{sec:intro}

In 1965 Richard Thompson introduced three groups that today are usually denoted $F$, $T$, and $V$. These have received a lot of recent attention for their interesting and often surprising properties. Most prominently, $T$ and $V$ are finitely presented, infinite, simple groups, and $F$ is torsion-free with infinite cohomological dimension and of type~$\F_\infty$.

Numerous generalizations of Thompson's groups have been introduced in the literature; see for example \cite{higman74, stein92, guba97, roever99, brin04, hughes09, martinez-perez13, belk15}. Most of these constructions either generalize the way in which branching can occur, or mimic the self-similarity in some way. Here we describe a more algebraic construction of Thompson-like groups, which combines the usual branching of the group $F$ with a chosen family of groups. The construction is based on Brin's description of the braided Thompson group $\Vbr$ \cite{brin07}, which utilizes the family of braid groups. Another example is the pure braided Thompson group $\Fbr$ introduced by Brady, Burillo, Cleary and Stein in \cite{brady08}, using the pure braid groups. Classical examples include $F$, using the trivial group, and $V$, using the symmetric groups.

The input to our construction is a directed system of groups $(G_n)_{n \in \N}$ together with a \emph{cloning system}, which essentially determines how a group element is moved past a split. A cloning system consists of morphisms $G_n \to \symm_n$ (where $\symm_n$ is the symmetric group on $n$ symbols), and \emph{cloning maps} $\clone^n_{k} \colon G_n \to G_{n+1}$, $1 \le k \le n$, subject to certain conditions (see Definition~\ref{def:filtered_cloning_system}). The output is a group $\Thomp{G_*}$:

\begin{main:thompson_construction}
 Let $(G_n)_{n \in \N}$ be an injective directed system of groups equipped with a cloning system. Then there is a generalized Thompson group $\Thomp{G_*}$ that contains all of the $G_n$.
\end{main:thompson_construction}

The groups $F$, $V$, $\Fbr$ and $\Vbr$ are all examples of groups of the form $\Thomp{G_*}$.

One of our main motivations for constructing these new Thompson-like groups is the analysis of their finiteness properties. Recall that a group $G$ is of \emph{type~$\F_n$} if there is a $K(G,1)$ with finite $n$-skeleton. For example, $\F_1$ means finitely generated and $\F_2$ means finitely presented. We are in particular interested in understanding how the finiteness properties of $\Thomp{G_*}$ depend on the finiteness properties of the groups $G_n$. Our main results are:

\begin{main:borel_thomp_full_fin_props}
 Let $\GlobalField$ be a global function field, let $S$ be a set of places of $\GlobalField$, and let $\calO_S$ be the ring of $S$-integers in $\GlobalField$. Let $B_n$ denote the algebraic group of invertible upper triangular $n$-by-$n$ matrices. There is a generalized Thompson group $\Thomp{B_*(\calO_S)}$ and it is of type~$\F_{\abs{S}-1}$ but not of type~$\F_{\abs{S}}$.
\end{main:borel_thomp_full_fin_props}

To put this into context it is important to know that the groups $B_n(\calO_S)$ are themselves of type~$\F_{\abs{S}-1}$ but not of type~$\F_{\abs{S}}$ by \cite{bux04}. In particular, for every $n \in \N$, we get an example of a generalized Thompson group of type~$\F_{n-1}$ but not of type~$\F_n$.

\begin{main:thomp_abels_Finfty}
 Let $\abels_n(\Z[1/p])$ be the $n$th Abels group (see Section~\ref{sec:matrix_groups}). There is a generalized Thompson group $\Thomp{\abels_*(\Z[1/p])}$ and it is of type~$\F_\infty$.
\end{main:thomp_abels_Finfty}

The groups $\abels_n(\Z[1/p])$ are known to be of type~$\F_{n-1}$ but not of type~$\F_n$ by \cite{abels87,brown87}. To be of type~$\F_\infty$ for a generalization of Thompson's groups is a relatively common phenomenon, but what is interesting about this example is that it organizes the groups $\abels_n(\Z[1/p])$, none of which is individually of type~$\F_\infty$, into a group of type~$\F_\infty$.

To formulate the above statements in a unified way, it is helpful to introduce the \emph{finiteness length} $\phi(G)$ of a group $G$, which is just the supremum over all $n$ for which $G$ is of type~$\F_n$. Now Theorems~\ref{thrm:borel_thomp_full_fin_props} and \ref{thrm:thomp_abels_Finfty} can be formulated to say that
\begin{equation}\label{eq:limit}
\phi(\Thomp{G_*}) = \liminf_n \phi(G_n)
\end{equation}
for the respective groups.

This relation is not coincidental but is suggested by the structure of the groups. In fact, we give a general construction which reduces proving the inequality $\ge$ for~\eqref{eq:limit} to showing that certain complexes $\dlkmodel{G_*}{n}$ are asymptotically highly connected. This construction is an abstraction of the well developed methods from \cite{brown92,stein92,brown06,farley03,fluch13,bux14} (which were all used to prove that the respective groups are of type~$\F_\infty$). For this reason, the proof of the inequality $\ge$ in Theorem~\ref{thrm:borel_thomp_full_fin_props} works without change for the groups $B_n(R)$ where $R$ is an arbitrary ring. This evidence leads us to ask:

\begin{main:finiteness_conjecture}
 For which generalized Thompson groups $\Thomp{G_*}$ does~\eqref{eq:limit} hold?
\end{main:finiteness_conjecture}

The group $\Thomp{G_*}$ may be thought of as a limit of the groups $G_n$, for example since it contains all of them. From this point of view, it is rather remarkable that~\eqref{eq:limit} holds in such generality. For example compare this to an ascending direct limit of groups with good finiteness properties, which will not even be finitely generated.

Another reason why~\eqref{eq:limit} is interesting is that it describes how finiteness properties of groups change when they are subject to a certain operation (here Thompsonifying). A different such operation is braiding: when $V$ is ``braided,'' we get $\Vbr$, and similarly $F$ yields $\Fbr$. The question of the finiteness properties of $\Fbr$ and $\Vbr$ was answered in~\cite{bux14}; they are still of type~$\F_\infty$, just like $F$ and $V$. When reinterpreting $F$, $V$, $\Fbr$ and $\Vbr$ as Thompsonifications (of the trivial group, the symmetric groups, the pure braid groups, and the braid groups, respectively), they provide more examples where~\eqref{eq:limit} holds: in all of these cases all the groups $G_n$ are of type~$\F_\infty$ and so are the corresponding Thompson groups. This is in some cases related to a similar program carried out in \cite{bartholdi15} for wreath products, see Remark~\ref{rmk:wreath}.

\medskip

In addition to the groups discussed so far, we also construct generalized Thompson groups for more families of groups. All of them are relatives of the family of symmetric groups in some way and it is very natural to put them into a generalized Thompson group. The first is a family of mock reflection groups that were studied by Davis, Januszkiewicz and Scott \cite{davis03}. The groups naturally arise as blowups of symmetric groups and we call them mock symmetric groups. Constructing a generalized Thompson group for the mock symmetric groups was suggested to us by Januszkiewicz. The second family consists of loop braid groups, which are a melding of symmetric groups and braid groups.

\begin{main:existence_thompson_mock_loop}
There exist generalized Thompson groups $\Vmock$, $\Vloop$ and $\Floop$ built from (and thus containing) all mock symmetric groups, all loop braid groups, and all pure loop braid groups. The groups $\Vmock$ and $\Vloop$ surject onto $V$ and $\Floop$ surjects onto $F$.
\end{main:existence_thompson_mock_loop}

We expect that all of these groups belong to the list of groups that answer Question~\ref{ques:finiteness_conjecture} positively, and thus:

\begin{main:Vmock_Vloop_Floop_conj}
 $\Vmock$, $\Vloop$ and $\Floop$ are of type~$\F_\infty$.
\end{main:Vmock_Vloop_Floop_conj}

\medskip

To investigate the finiteness properties of a generalized Thompson group $\Thomp{G_*}$ we let it act on a contractible cube complex $\Stein{G_*}$ which we call the \emph{Stein--Farley complex}. This space exists for arbitrary cloning systems and in many cases has been used previously. When the cloning system is \emph{properly graded} (Definition~\ref{def:properly_graded}), the action has certain desirable properties: the cell stabilizers are subgroups of the groups $G_n$ and there is a natural cocompact filtration. To show that the generalized Thompson group is of type~$\F_n$, assuming that all the $G_n$ are, (which gives one half of~\eqref{eq:limit}) thus amounts to showing that the descending links $\dlkmodel{G_*}{n}$ in this filtration are eventually $(n-1)$-connected. This is the only part of the proof that needs to be done for every properly graded cloning system individually and depends on the nature of the concrete example. This treats the positive case, which so far has been sufficient for most existing Thompson groups since they have been of type~$\F_\infty$.

For the negative finiteness properties we have to develop new methods. For example we give a condition on a group homomorphism $G \to H$ that ensures that if the morphism factors through a group $K$ then $K$ cannot be of type~$\FP_n$, see Theorem~\ref{thrm:relative_brown} (type~$\FP_n$ is a homological, and slightly weaker, version of type~$\F_n$). This is a similar idea to that of~\cite{krstic97} and may be of independent use. Unlike the proof that $\Thomp{B_*(\calO_S)}$ is of type~$\F_{\abs{S}-1}$, the proof that it is not of type~$\FP_{\abs{S}}$ borrows large parts from the proof in~\cite{bux04} of the same fact for $B_n(\calO_S)$. For example, the space for $\Thomp{B_*(\calO_S)}$ is built out of the space for $B_2(\calO_S)$ (which is a Bruhat--Tits tree).

\bigskip

The paper is organized as follows. In Section~\ref{sec:preliminaries} we recall some background on monoids and the Zappa--Sz\'ep product. In Section~\ref{sec:defining_data} we introduce cloning systems (Definition~\ref{def:filtered_cloning_system}) and explain how they give rise to generalized Thompson groups. Section~\ref{sec:basic_properties} collects some group theoretic consequences that follow directly from the construction. To study finiteness properties, the Stein--Farley complex is introduced in Section~\ref{sec:spaces}. The filtration and its descending links are described in Section~\ref{sec:finiteness_props}, and we discuss some background on Morse theory and other related techniques for proving high connectivity, including a new method in Section~\ref{sec:relative_brown}.

Up to this point everything is mostly generic. The following sections discuss examples. Section~\ref{sec:direct_prods} gives an elementary example where $G_n = H^n$ for some group $H$. Section~\ref{sec:matrix_groups} discusses cloning systems for groups of upper triangular matrices. In Section~\ref{sec:mtx_fin_props} we study their finiteness properties. The last two sections \ref{sec:mock} and \ref{sec:loop} introduce the groups $\Vmock$ and $\Vloop$ and $\Floop$.

\subsection*{Acknowledgments} We are grateful to Matt Brin and Kai-Uwe Bux for helpful discussions, to Tadeusz Januszkiewicz for proposing to us the group $\Vmock$, and to Werner Thumann and an anonymous referee for many helpful comments. Both authors were supported by the SFB~878 in M\"unster. The first author was also supported directly by the DFG through project WI~4079/2 and by the SFB~701 in Bielefeld. All of this support is gratefully acknowledged.

\numberwithin{equation}{section}

\setcounter{section}{-1}
\section{Motivation}\label{sec:motivation}

Starting with the first section we will spend some ten pages introducing notions and technical results from the theory of monoids. Before we dive into these preparations, we want to explain why they are precisely the ones needed to describe generalized Thompson groups. We illustrate this on the example of $\Vbr$.

\begin{figure}[t]
\centering
\begin{tikzpicture}[line width=1pt]
  \draw
   (0,-2) -- (1,-1) -- (2,-2)   (1,-2) -- (.5,-1.5);
  \draw
   (1,-2) to [out=-90, in=90] (0,-4);
  \draw[white, line width=4pt]
   (0,-2) to [out=-90, in=90] (2,-4);
  \draw
   (0,-2) to [out=-90, in=90] (2,-4);
  \draw[white, line width=4pt]
   (2,-2) to [out=-90, in=90] (1,-4);
  \draw
   (2,-2) to [out=-90, in=90] (1,-4);
  \draw
   (0,-4) -- (1,-5) -- (2,-4)  (1,-4) -- (1.5,-4.5);
   \draw[decoration={brace,amplitude = 5},decorate,line width = 1] (-.5,-1.9) -- (-.5,-1.1);
   \node[anchor=east] at (-.7,-1.5) {splits};
   \draw[decoration={brace,amplitude = 5},decorate,line width = 1] (-.5,-3.9) -- (-.5,-2.1);
   \node[anchor=east] at (-.7,-3) {group element};
   \draw[decoration={brace,amplitude = 5},decorate,line width = 1] (-.5,-4.9) -- (-.5,-4.1);
   \node[anchor=east] at (-.7,-4.5) {merges};
\end{tikzpicture}
\hspace{1cm}
\begin{tikzpicture}[line width=1pt]
  \begin{scope}[yshift=-2.5cm,scale=.75]
  \node at (-.25,.75) {b};
  \begin{scope}
  \draw
   (0,1) -- (.5,.5) -- (1,1)
   (.5,-.5) -- (.5,.5)
   (0,-1) -- (.5,-.5) -- (1,-1);
  \end{scope}
  \node at (1.75,0) {$=$};   
 \begin{scope}[xshift=2.75cm]
  \draw
   (0,1) -- (0,-1)
   (1,1) -- (1,-1);
 \end{scope}
  \end{scope}
 \begin{scope}[scale=.75]
  \node at (-.25,.75) {a};
  \draw
   (.5,-1) -- (.5,-.75)
   (.5,1) -- (.5,.75)
   (0,-.25) -- (0,.25)
   (1,-.25) -- (1,.25)
   (0,-.25) -- (.5,-.75) -- (1,-.25)
   (0,.25) -- (.5,.75) -- (1,.25);
  \node at (1.5,0) {$=$};
  \draw
   (2.25,1) -- (2.25,-1);
 \end{scope}
  \begin{scope}[xshift=3cm,scale=.75]
  \node at (-.5,.75) {c};
  \begin{scope}
  \draw
   (0,1) -- (0,.75)
   (1,1) -- (1,.75)
   (0,.75) -- (.5,.25) -- (1,.75)
   (.5,-1) -- (.5,.25)
   (1.5,-1) -- (1.5,-.75)
   (2.5,-1) -- (2.5,-.75)
   (1.5,-.75) -- (2,-.25) -- (2.5,-.75)
   (2,1) -- (2,-.25);
  \end{scope}
  \node at (3,0) {$=$};   
  \begin{scope}[xshift=3.5cm]
  \draw
   (0,1) -- (0,-.25)
   (1,1) -- (1,-.25)
   (0,-.25) -- (.5,-.75) -- (1,-.25)
   (.5,-1) -- (.5,-.75)
   (1.5,-1) -- (1.5,.25)
   (2.5,-1) -- (2.5,.25)
   (1.5,.25) -- (2,.75) -- (2.5,.25)
   (2,1) -- (2,.75);
  \end{scope}
 \end{scope}
 \begin{scope}[xshift=4cm, yshift=-2.5cm]
  \node at (-.25,.625) {d};
 \begin{scope}
  \draw
   (0.25,0.75) to [out=-90, in=90, looseness=1] (1.25,-0.5) -- (1.25,-0.75);
  \draw[white, line width=4pt]
   (1.25,0.75) to [out=-90, in=90, looseness=1] (0.25,-0.5);
  \draw
   (1.25,0.75) to [out=-90, in=90, looseness=1] (0.25,-0.5);
  \draw
   (0,-0.75) -- (0.25,-0.5) -- (0.5,-0.75);
 \end{scope}
 \node at (2,0) {$=$};
 \begin{scope}[xshift=2.25cm]
  \draw
   (0.25,0.75) -- (.25,.5) to [out=-90, in=90, looseness=1] (1.25,-0.75);
  \draw
   (0.75,0.25) -- (1,0.5) -- (1.25,0.25)   (1,0.5) -- (1,0.75);
  \draw[white, line width=4pt]
   (0.75,0.25) to [out=-90, in=90, looseness=1] (0.25,-0.75)
   (1.25,0.25) to [out=-90, in=90, looseness=1] (0.75,-0.75);
  \draw
   (0.75,0.25) to [out=-90, in=90, looseness=1] (0.25,-0.75)
   (1.25,0.25) to [out=-90, in=90, looseness=1] (0.75,-0.75);
 \end{scope}
 \end{scope}
\end{tikzpicture}

\caption{On the left, an element of $\Vbr$ in its standard form consisting of splitting, braiding and merging. On the right, some relations: (a) splitting and then merging is trivial; (b) merging and then splitting is trivial; (c) splits and merges on different strands commute. The main relation, (d), which is encoded by the Zappa--Sz\'ep product, is how splits and group elements interact.}
\label{fig:vbr_element}
\end{figure}
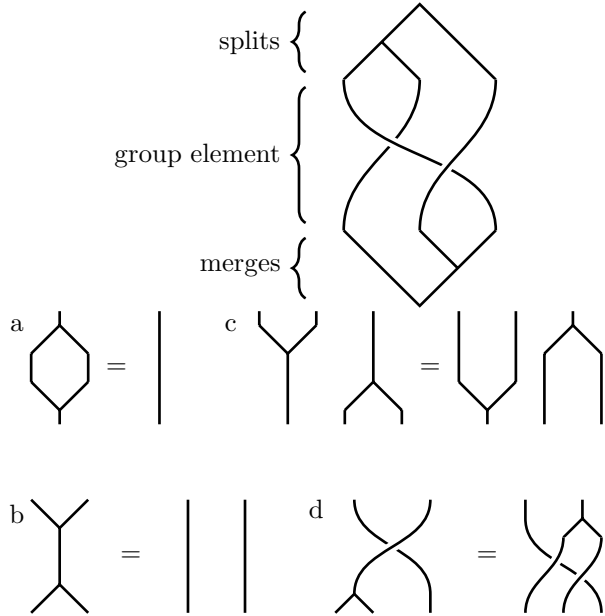

We want to think of an element of a Thompson group as consisting of a tree of splittings, followed by a group element from a chosen group (a braid in the example), and finally an inverse tree of merges. An element of $\Vbr$ is illustrated in Figure~\ref{fig:vbr_element}. Two elements are multiplied by stacking them on top of each other and reducing, as in Figure~\ref{fig:reducing}. Among the relations available to reduce an element are the fact that splitting and then merging again is a trivial operation, as well as merging and then splitting (Figure~\ref{fig:vbr_element}(a),(b)). Another relation that is implicit in the pictures is that a group element followed by another group element is the same as the product. However, these relations are not typically sufficient to bring a diagram into the form that we want: splits, group element, merges. To move all the splits to the top (and all the merges to the bottom), we eventually will have to move a split $\lambda$ past a group element $g$. In Figure~\ref{fig:reducing} this point is reached in the third step. Expressed algebraically, we need to rewrite $g \lambda  = \lambda' g'$ for some group element $g'$ and some split $\lambda'$ (Figure~\ref{fig:vbr_element}(d)). The algebraic operation that defines how a split is moved past a group element is the Zappa--Sz\'ep product.

The trees of splittings will be elements of the forest monoid $\Fmonoid$. We will then form the Zappa-Sz\'ep product $\Fmonoid \bowtie G$ with the chosen group $G$. To also obtain merges, we will pass to the group of fractions --- a merge is just the inverse of a split. For technical reasons, we will have started with infinitely many strands and in a final step have to reduce to elements that start and end with one strand. With this outline in mind, we hope the reader will find the following technical pages more illuminating.

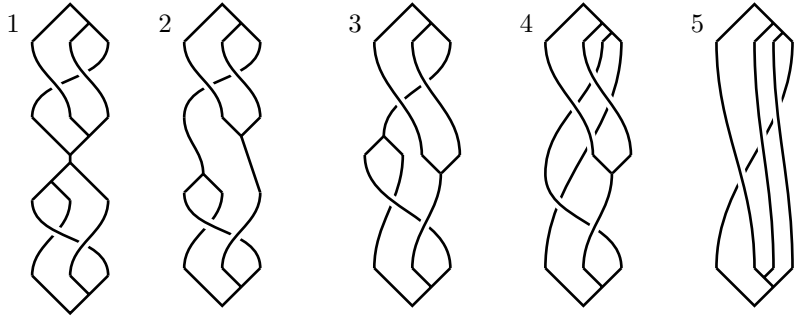
\begin{figure}[htb]
\centering
\begin{tikzpicture}[line width=1pt,scale=.5]
  \begin{scope}
  \node at (-.5,-1.5) {$1$};
  \draw
   (0,-2) -- (1,-1) -- (2,-2)   (1,-2) -- (1.5,-1.5);
  \draw
   (2,-2) to [out=-90, in=90] (0,-4);
  \draw[white, line width=4pt]
   (1,-2) to [out=-90, in=90] (2,-4);
  \draw
   (1,-2) to [out=-90, in=90] (2,-4);
  \draw[white, line width=4pt]
   (0,-2) to [out=-90, in=90] (1,-4);
  \draw
   (0,-2) to [out=-90, in=90] (1,-4);
  \draw
   (0,-4) -- (1,-5) -- (2,-4)  (1,-4) -- (1.5,-4.5);
  \end{scope}
  \draw (1,-5) -- (1,-5.2);
  \begin{scope}[yshift=-4.2cm]
  \draw
   (0,-2) -- (1,-1) -- (2,-2)   (1,-2) -- (.5,-1.5);
  \draw
   (1,-2) to [out=-90, in=90] (0,-4);
  \draw[white, line width=4pt]
   (0,-2) to [out=-90, in=90] (2,-4);
  \draw
   (0,-2) to [out=-90, in=90] (2,-4);
  \draw[white, line width=4pt]
   (2,-2) to [out=-90, in=90] (1,-4);
  \draw
   (2,-2) to [out=-90, in=90] (1,-4);
  \draw
   (0,-4) -- (1,-5) -- (2,-4)  (1,-4) -- (1.5,-4.5);
  \end{scope}
  \begin{scope}[xshift=4cm]
  \node at (-.5,-1.5) {$2$};
  \begin{scope}
  \draw
   (0,-2) -- (1,-1) -- (2,-2)   (1,-2) -- (1.5,-1.5);
  \draw
   (2,-2) to [out=-90, in=90] (0,-4);
  \draw[white, line width=4pt]
   (1,-2) to [out=-90, in=90] (2,-4);
  \draw
   (1,-2) to [out=-90, in=90] (2,-4);
  \draw[white, line width=4pt]
   (0,-2) to [out=-90, in=90] (1,-4);
  \draw
   (0,-2) to [out=-90, in=90] (1,-4);
  \draw
  (1,-4) --  (1.5,-4.5) -- (2,-4);
  \end{scope}
  \draw (0,-4) to [out=-90, in=90] (.5,-5.5);
  \draw (1.5,-4.5) -- (2,-6);
  \begin{scope}[yshift=-4cm]
  \draw
  (1,-2) -- (.5,-1.5) -- (0,-2);
  \draw
   (1,-2) to [out=-90, in=90] (0,-4);
  \draw[white, line width=4.5pt]
   (0,-2) to [out=-90, in=90] (2,-4);
  \draw
   (0,-2) to [out=-90, in=90] (2,-4);
  \draw[white, line width=4pt]
   (2,-2) to [out=-90, in=90] (1,-4);
  \draw
   (2,-2) to [out=-90, in=90] (1,-4);
  \draw
   (0,-4) -- (1,-5) -- (2,-4)  (1,-4) -- (1.5,-4.5);
  \end{scope}
  \end{scope}
  \begin{scope}[xshift=9cm]
  \node at (-.5,-1.5) {$3$};
  \begin{scope}
  \draw
   (0,-2) -- (1,-1) -- (2,-2)   (1,-2) -- (1.5,-1.5);
  \draw
   (2,-2) to [out=-90, in=90] (.25,-4.5);
  \draw[white, line width=4pt]
   (1,-2) to [out=-90, in=90] (2.25,-5);
  \draw
   (1,-2) to [out=-90, in=90] (2.25,-5);
  \draw[white, line width=4pt]
   (0,-2) to [out=-90, in=90] (1.25,-5);
  \draw
   (0,-2) to [out=-90, in=90] (1.25,-5);
  \draw
  (1.25,-5) --  (1.75,-5.5) -- (2.25,-5);
  \end{scope}
  \begin{scope}[yshift=-4cm]
  \draw
   (-.25,-1) -- (.25,-.5)   -- (.75, -1);
  \draw
   (.75,-1) to [out=-90, in=90] (0,-4);
  \draw[white, line width=4pt]
   (-.25,-1) to [out=-90, in=90] (2,-4);
  \draw
   (-.25,-1) to [out=-90, in=90] (2,-4);
  \draw[white, line width=4pt]
   (1.75,-1.5) to [out=-90, in=90] (1,-4);
  \draw
   (1.75,-1.5) to [out=-90, in=90] (1,-4);
  \draw
   (0,-4) -- (1,-5) -- (2,-4)  (1,-4) -- (1.5,-4.5);
  \end{scope}
  \end{scope}
  \begin{scope}[xshift=13.5cm]
  \node at (-.5,-1.5) {$4$};
  \draw
   (0,-2) -- (1,-1) -- (2,-2)   (1,-2) -- (1.5,-1.5)  (1.5,-2) -- (1.75,-1.75);
  \draw
   (2,-2) to [out=-90, in=90] (0,-8);
  \draw[white, line width=4pt]
   (1.5,-2) to [out=-90, in=90] (0,-5.5) to [out=-90, in=90] (2,-8);
  \draw
   (1.5,-2) to [out=-90, in=90] (0,-5.5) to [out=-90, in=90] (2,-8);
  \draw[white, line width=4pt]
   (1,-2) to [out=-90, in=90] (2.25,-5);
  \draw
   (1,-2) to [out=-90, in=90] (2.25,-5);
  \draw[white, line width=4pt]
   (0,-2) to [out=-90, in=90] (1.25,-5);
  \draw
   (0,-2) to [out=-90, in=90] (1.25,-5);
  \draw
  (1.25,-5) --  (1.75,-5.5) -- (2.25,-5);
  \draw[white, line width=4pt]
   (1.75,-5.5) to [out=-90, in=90] (1,-8);
  \draw
   (1.75,-5.5) to [out=-90, in=90] (1,-8);
  \draw
   (0,-8) -- (1,-9) -- (2,-8)  (1,-8) -- (1.5,-8.5);
  \end{scope}
  \begin{scope}[xshift=18cm]
  \node at (-.5,-1.5) {$5$};
  \draw
   (0,-2) -- (1,-1) -- (2,-2)   (1,-2) -- (1.5,-1.5)  (1.5,-2) -- (1.75,-1.75);
  \draw
   (2,-2) to [out=-90, in=90] (0,-8);
  \draw[white, line width=4pt]
   (1.5,-2) to [out=-90, in=90] (2,-8);
  \draw
   (1.5,-2) to [out=-90, in=90] (2,-8);
  \draw[white, line width=4pt]
   (1,-2) to [out=-90, in=90] (1.5,-8);
  \draw
   (1,-2) to [out=-90, in=90] (1.5,-8);
  \draw[white, line width=4pt]
   (0,-2) to [out=-90, in=90]  (1,-8);
  \draw
   (0,-2) to [out=-90, in=90] (1,-8);
  \draw
   (0,-8) -- (1,-9) -- (2,-8)  (1,-8) -- (1.5,-8.5)  (1.5,-8) -- (1.25,-8.25);
  \end{scope}
\end{tikzpicture}
\caption{Computing the product of two elements of $\Vbr$. First, both elements are stacked onto each other. Second, pairs of merges and splits are resolved. Third, merges and splits are moved past each other. In the fourth and fifth step a merge and a split are moved past a group element (here a braid).}
\label{fig:reducing}
\end{figure}

\section{Preliminaries}\label{sec:preliminaries}

Much of the material in this section is taken from~\cite{brin07}.

\subsection{Monoids}\label{sec:monoids}

A \emph{monoid} is an associative binary structure with a two-sided identity. A monoid $M$ is called \emph{left cancellative} if for all $x,y,z\in M$, we have that $xy=xz$ implies $y=z$. Elements $x,y\in M$ have a \emph{common left multiple} $m$ if there exist $z,w\in M$ such that $zx=wy=m$. This is the \emph{least common left multiple} if for all $p,q\in M$ such that $px=qy$, we have that $px$ is a left multiple of $m$. There are the obvious definitions of \emph{right cancellative}, \emph{common right multiples} and \emph{least common right multiples}.
We say that $M$ \emph{has common right/left multiples} if any two elements have a common right/left multiple. It is said to \emph{have least common right/left multiples} if any two elements that have some common right/left multiple have a least common right/left multiple. Finally, we say $M$ is \emph{cancellative} if it is both left and right cancellative. The importance of these notions lies in the following classical theorem (see \cite[Theorems~1.23, 1.25]{clifford61}):

\begin{theorem}[Ore]\label{thm:ore}
A cancellative monoid with common right multiples has a unique group of right fractions.
\end{theorem}

Recall that for every monoid $M$ there exists a group $G_M$ and a monoid morphism $\omega \colon M \to G_M$ such that every monoid morphism from $M$ to a group factors through $\omega$ (namely the group generated by all the elements of $M$ subject to all the relations that hold in $M$). This is \emph{the group of fractions of $M$}. The morphism $\omega$ will be injective if and only if $M$ embeds into a group. A group $G$ is called \emph{a group of right fractions of $M$} if it contains $M$ and every element of $G$ can be written as $m \cdot n^{-1}$ with $m,n \in M$. A group of right fractions exists precisely in the situation of Ore's theorem and is unique up to isomorphism; see \cite[Section~1.10]{clifford61} for details.  We call a monoid satisfying the hypotheses of Theorem~\ref{thm:ore} an \emph{Ore monoid}. The group of right fractions of an Ore monoid is its group of fractions (see for example \cite[Theorem~7.1.16]{kashiwara06}):

\begin{lemma}\label{lem:ore_fractions}
 Let $M$ be an Ore monoid, let $G$ be its group of right fractions and let $H$ be any group. Let $\varphi \colon M \to H$ be a monoid morphism. Then the map $\tilde{\varphi} \colon G \to H$ defined by $\tilde{\varphi}(mn^{-1}) = \varphi(m) \cdot \varphi(n)^{-1}$ is a group homomorphism and $\varphi = \tilde{\varphi}|_M$.
\end{lemma}

\begin{proof}
 That inverses map to inverses is clear. Let $m_1,m_2,n_1,n_2 \in M$ and let $n_1 \cdot x = m_2 \cdot y$ be a common right multiple so that $m_1 n_1^{-1} m_2 n_2^{-1} = m_1 x y^{-1}n_2^{-1}$. We have to check that
 \begin{equation}
 \label{eq:fraction_multiplication}
 \varphi(m_1) \varphi(n_1)^{-1} \varphi(m_2) \varphi(n_2)^{-1} = \varphi(m_1 x) \varphi(n_2 y)^{-1}\text{.}
 \end{equation}
 The fact that $\varphi$ is a monoid morphism means that $\varphi(n_1)\varphi(x) = \varphi(m_2) \varphi(y)$ which entails $\varphi(n_1)^{-1}\varphi(m_2) = \varphi(x)\varphi(y)^{-1}$. Extending by $\varphi(m_1)$ from the left and by $\varphi(n_2)^{-1}$ from the right gives \eqref{eq:fraction_multiplication}.
\end{proof}

\subsection{Posets from monoids}\label{sec:posets_from_monoids}

Throughout this section let $M$ be an Ore monoid and let $G$ be its group of right fractions. The notions of left/right multiple/factor are uninteresting for $G$ as a monoid because it is a group. Instead we introduce these notions relative to the monoid $M$. Concretely, assume that elements $a,b,c \in G$ satisfy
\[
ab=c \text{.}
\]
If $a \in M$ then we call $b$ a \emph{right factor} of $c$ and $c$ a \emph{left multiple} of $b$. If $b \in M$ then we call $a$ a \emph{left factor} of $c$ and $c$ a \emph{right multiple} of $a$. If $g$ is a left factor (respectively right multiple) of both $h$ and $h'$ then we say that it is a \emph{common left factor} (respectively \emph{common right multiple}). If $g$ is a common left factor of $h$ and $h'$ and any other left factor of $h$ and $h'$ is also a left factor of $g$ then $g$ is called a \emph{greatest common left factor}. If $g$ is a common right multiple of $h$ and $h'$ and every other right multiple is also a right multiple of $g$ then $g$ is called a \emph{least common right multiple} of $h$ and $h'$. Thus we obtain notions of when $G$ has (least) common right/left multiples and (greatest) common right/left factors. We say that two elements have \emph{no common right factor} if they have greatest common right factor $1$.

Under a moderate additional assumption, having least common right multiples is inherited by $G$ from $M$:

\begin{lemma}
 Let $M$ have least common right multiples. Let $n,n',m,m' \in M$ be such that $n$ and $m$ have no common right factor and neither do $n'$ and $m'$. Let $nv = n'u$ be a least common right multiple of $n$ and $n'$. Then $nv = n'u$ is a least common right multiple of $nm^{-1}$ and $n'{m'}^{-1}$.\qed
\end{lemma}

We call a monoid homomorphism $\len \colon M \to \Nz$ a \emph{length function} if every element of the kernel is a unit. It induces a length function $\len \colon G \to \Z$. Note that if $M$ admits a length function then every element of $G$ can be written as $mn^{-1}$ where $m$ and $n$ are elements of $M$ with no common right factor.

The following is an extension of \cite[Lemma~2.3]{brin07} to $G$.

\begin{lemma}
Assume that $M$ admits a length function. Then $G$ has least common right multiples if and only if it has greatest common left factors.\qed
\end{lemma}

One reason for our interest in least common right multiples and greatest common left factors is order theoretic. Define a relation on $G$ by declaring $g \le h$ if $g$ is a left factor of $h$. This relation is reflexive and transitive but fails to satisfy antisymmetry if $M$ has non-trivial units. We denote the relation induced on $G/M^\times$ also by $\le$. It is an order relation so $G/M^\times$ becomes a partially ordered set (poset). Spelled out, the relation is given by $gM^\times \le hM^\times$ if $g^{-1}h \in M$.

The algebraic properties discussed before immediately translate into order theoretic properties: recall that a poset $P$ is a \emph{join-semilattice} if any two elements of $P$ have a supremum (their join). We say that $P$ \emph{has conditional meets} if any two elements that have a lower bound have an infimum.

\begin{observation}
 If $M$ has common right multiples, least common right multiples, and greatest common left factors then $M/M^\times$ is a join-semilattice with conditional meets. Similarly, if $G$ has common right multiples, least common right multiples and greatest common left factors then $G/M^\times$ is a join-semilattice with conditional meets.\qed
\end{observation}

Putting everything together, we find:

\begin{corollary}\label{cor:fraction_semilattice}
 Let $M$ be a cancellative monoid with common right multiples, least common right multiples and length function. Let $G$ be its group of right fractions. Then $G/M^\times$ is a join-semilattice with conditional meets.\qed
\end{corollary}

\subsection{The monoid of forests}\label{sec:forest_monoid}

Since we are interested in Thompson's groups, an important monoid in all that follows will be the \emph{monoid of forests}, which we define in this section.

For us, a \emph{tree} is always a finite rooted full binary tree. In other words, every vertex has either no outgoing edges or a left and right outgoing edge, and every vertex other than the root has an incoming edge. The vertices without outgoing edges are called \emph{leaves}. The distinction between left and right induces a natural order on the leaves. If a tree has only one leaf, then the leaf is also its root and the tree is the \emph{trivial tree}.

By a \emph{forest} we mean a sequence of trees $\forest=(\tree_i)_{i\in\N}$ such that all but finitely many $\tree_i$ are trivial. The roots are numbered in the obvious way, i.e., the $i$th root of $\forest$ is the root of $\tree_i$. If all the $\tree_i$ are trivial we call $\forest$ \emph{trivial}. If the $\tree_i$ are trivial for $i > 1$ then the forest is called \emph{semisimple} (here we deviate from Brin's notation; what we call ``semisimple'' is called ``simple'' in \cite{brin07}, and what we will later call ``simple'', Brin calls ``simple and balanced''). The \emph{rank} of $\forest$ is the least index $i$ such that $\tree_j$ is trivial for $j > i$. So $\forest$ is semisimple if it has rank at most $1$. The leaves of all the $\tree_i$ are called the \emph{leaves} of $\forest$. The order on the leaves of the trees induces an order on the leaves of the forest by declaring that any leaf of $\tree_i$ comes before any leaf of $\tree_j$, whenever $i<j$. We may equivalently think of the leaves as numbered by natural numbers. The number of \emph{feet} of a semisimple forest $(\tree_i)_{i \in \N}$ is the number of leaves of $\tree_1$ .

Let $\Fmonoid$ be the set of forests. Define a multiplication on $\Fmonoid$ as follows. Let $\forest=(\tree_k)$ and $\forest'=(\tree'_k)$ be forests, and set $\forest \forest'$ to be the forest obtained by identifying the $i$th leaf of $\forest$ with the $i$th root of $\forest'$, for each $i$. This product is associative, and the trivial forest is a left and right identity, so $\Fmonoid$ is a monoid. Some more details on $\Fmonoid$ can be found in Section~3 of \cite{brin07}. Figure~\ref{fig:forest_mult} illustrates the multiplication of two elements.

\begin{figure}[t]
\centering
\begin{tikzpicture}[line width=0.8pt, scale=0.44]
  \draw
   (-3,-3) -- (0,0) -- (3,-3)   (1,-3) -- (-1,-1)   (-1,-3) -- (0,-2);
  \filldraw
   (-3,-3) circle (1.5pt)   (0,0) circle (1.5pt)   (3,-3) circle (1.5pt)   (-1,-1) circle (1.5pt)   (-1,-3) circle (1.5pt)   (0,-2) circle (1.5pt)   (1,-3) circle (1.5pt)   (2,0) circle (1.5pt)   (4,0) circle (1.5pt)   (6,0) circle (1.5pt);
  \node at (8,0) {$\dots$};

  \begin{scope}[yshift=-5cm]
   \draw
    (-4,-1) -- (-3,0) -- (-2,-1)   (1,-2) -- (3,0) -- (5,-2)   (3,-2) -- (2,-1);
   \filldraw
    (-4,-1) circle (1.5pt)   (-3,0) circle (1.5pt)   (-2,-1) circle (1.5pt)   (-1,0) circle (1.5pt)   (1,0) circle (1.5pt)   (1,-2) circle (1.5pt)   (3,0) circle (1.5pt)   (5,-2) circle (1.5pt)   (3,-2) circle (1.5pt)   (2,-1) circle (1.5pt)   (5,0) circle (1.5pt);
  \node at (7,0) {$\dots$};   \node at (10,2) {$=$};
  \end{scope}
   
  \begin{scope}[xshift=16cm]
   \draw
    (-3,-3) -- (0,0) -- (3,-3)   (1,-3) -- (-1,-1)   (-1,-3) -- (0,-2)
    (-4,-4) -- (-3,-3) -- (-2,-4)   (1,-5) -- (3,-3) -- (5,-5)   (3,-5) -- (2,-4);
   \filldraw
   (-3,-3) circle (1.5pt)   (0,0) circle (1.5pt)   (3,-3) circle (1.5pt)   (-1,-1) circle (1.5pt)   (-1,-3) circle (1.5pt)   (0,-2) circle (1.5pt)   (1,-3) circle (1.5pt)   (2,0) circle (1.5pt)   (4,0) circle (1.5pt)   (6,0) circle (1.5pt)   (-4,-4) circle (1.5pt)   (-2,-4) circle (1.5pt)   (1,-5) circle (1.5pt)   (5,-5) circle (1.5pt)   (3,-5) circle (1.5pt);
   \node at (8,0) {$\dots$};
  \end{scope}
\end{tikzpicture}
\caption{Multiplication of forests.}
\label{fig:forest_mult}
\end{figure}
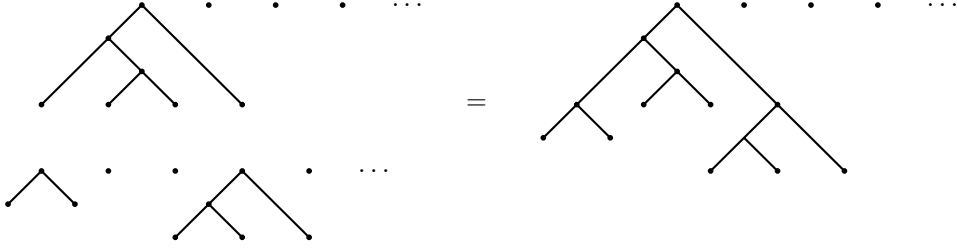

There is an obvious set of generators of $\Fmonoid$, namely the set of single-caret forests. Such a forest can be characterized by the property that there exists $k\in\N$ such that for $i<k$, the $i$th root is also the $i$th leaf, and for $i>k$, the $i$th root is also the $(i+1)$st leaf. Denote this forest by $\lambda_k$. Every tree in $\lambda_k$ is trivial except for the $k$th tree, which is a single caret.

\begin{proposition}[Presentation of the forest monoid]\label{prop:monoid_pres}\cite[Proposition~3.3]{brin07}
 $\Fmonoid$ is generated by the $\lambda_k$, and defining relations are given by
 \begin{equation}
 \label{eq:forest_relation}
 \lambda_j\lambda_i=\lambda_i\lambda_{j+1} \quad \text{for} \quad i<j \text{.}
 \end{equation}
 Every element of $\Fmonoid$ can be uniquely expressed as a word of the form $\lambda_{k_1}\lambda_{k_2}\cdots\lambda_{k_r}$ for some $k_1 \le \cdots \le k_r$.
\end{proposition}

A consequence is that the number of carets is an invariant of a forest, and is exactly the length of the word in the $\lambda_k$ representing the forest. The following is part of \cite[Lemma~3.4]{brin07}.

\begin{lemma}\label{lem:fmonoid_properties}
 The monoid $\Fmonoid$ has the following properties.
 \begin{enumerate}
  \item It is cancellative.\label{eq:fmonoid_cancellative}
  \item It has common right multiples.\label{eq:fmonoid_common_right_multiples}
  \item It has no non-trivial units.\label{eq:fmonoid_no_units}
  \item There is a monoid homomorphism $\len \colon \Fmonoid \to \Nz$ sending each generator to $1$. \label{eq:fmonoid_length}
  \item It has greatest common right factors and least common left multiples. \label{eq:fmonoid_greatest_common_right_factors}
  \item It has greatest common left factors and least common right multiples.\label{eq:fmonoid_greatest_common_left_factors}
 \end{enumerate}
\end{lemma}

In view of Theorem~\ref{thm:ore}, properties~\eqref{eq:fmonoid_cancellative} and~\eqref{eq:fmonoid_common_right_multiples} imply that $\Fmonoid$ has a unique group of right fractions, which we denote $\Fhat$.

\subsection{Zappa--Sz\'ep products}\label{sec:zappa}

In this section we recall the background on Zappa--Sz\'ep products of monoids. Our main reference is \cite[Section~2.4]{brin07}, and also see \cite{brin05}. When the monoids are groups, Zappa--Sz\'ep products generalize semidirect products by dropping the assumption that one of the groups be normal.

The internal Zappa--Sz\'ep product is straightforward to define. Let $M$ be a monoid with submonoids $U$ and $A$ such that every $m\in M$ can be written in a unique way as $m=u\alpha$ for $u\in U$ and $\alpha\in A$. In particular, for $\alpha\in A$ and $u\in U$ there exist $u'\in U$ and $\alpha'\in A$ such that $\alpha u=u'\alpha'$, and the $u'$ and $\alpha'$ are uniquely determined by $\alpha$ and $u$, so we denote them $u'=\alpha\cdot u$ and $\alpha'=\alpha^u$, following \cite{brin07}. The maps $(\alpha,u)\mapsto \alpha\cdot u$ and $(\alpha,u)\mapsto \alpha^u$ should be thought of as mutual actions of $U$ and $A$ on each other. Then we can define a multiplication on $U\times A$ via
\begin{equation}\label{eq:zappa--szep_multiplication}
(u,\alpha)(v,\beta) \defeq (u(\alpha\cdot v),\alpha^v \beta)\text{,}
\end{equation}
for $u,v\in U$ and $\alpha,\beta\in A$, and the map $(u,\alpha)\mapsto u\alpha$ is a monoid isomorphism from $U\times A$ (with this multiplication) to $M$; see \cite[Lemma~2.7]{brin07}. We say that $M$ is the (internal) Zappa--Sz\'ep product of $U$ and $A$, and write $M=U\bowtie A$.

\begin{example}[Semidirect product]\label{ex:semidirect}
 Suppose $G$ is a group that is a semidirect product $G=U \ltimes A$ for $U,A\le G$. Then for $u\in U$ and $\alpha\in A$ we have $\alpha u=u(u^{-1} \alpha u)$, and $u^{-1} \alpha u\in A$, so the actions defined above are just $\alpha\cdot u = u$ and $\alpha^u=u^{-1} \alpha u$.
\end{example}

We actually need to use the \emph{external} Zappa--Sz\'ep product. This is discussed in detail in \cite[Section~2.4]{brin07} (and in even more detail in \cite{brin05}).

\begin{definition}[External Zappa--Sz\'ep product]\label{def:ext_ZS_prod_monoids}
 Let $U$ and $A$ be monoids with maps $(\alpha,u)\mapsto \alpha\cdot u \in U$ and $(\alpha,u)\mapsto \alpha^u \in A$ satisfying the following eight properties for all $u,v\in U$ and $\alpha,\beta \in A$:

\smallskip

 \begin{enumerate}[label={\arabic*)}]
  \item $1_A \cdot u = u$ \hfill (Identity acting on $U$)
  \item $(\alpha\beta)\cdot u = \alpha\cdot (\beta\cdot u)$ \hfill (Product acting on $U$)
  \item $\alpha^{1_U} = \alpha$ \hfill (Identity acting on $A$)
  \item $\alpha^{(uv)} = (\alpha^u)^v$ \hfill (Product acting on $A$)
  \item $(1_A)^u = 1_A$ \hfill ($U$ acting on identity)
  \item $(\alpha\beta)^u = \alpha^{(\beta\cdot u)}\beta^u$ \hfill ($U$ acting on product)
  \item $\alpha \cdot 1_U = 1_U$ \hfill ($A$ acting on identity)
  \item $\alpha \cdot (uv) = (\alpha\cdot u)(\alpha^u\cdot v)$. \hfill ($A$ acting on product)
 \end{enumerate}

 \smallskip

 \noindent Then the maps are called a \emph{Zappa--Sz\'ep action}. The set $U \times A$ together with the multiplication defined by \eqref{eq:zappa--szep_multiplication} is called the \emph{(external) Zappa--Sz\'ep product} of $U$ and $A$, denoted $U \bowtie A$.
\end{definition}

It is shown in Lemma~2.9 in \cite{brin07} that the external Zappa--Sz\'ep product  turns $U \bowtie A$ into a monoid and coincides with the internal Zappa--Sz\'ep product of $U$ and $A$ with respect to the embeddings $u\mapsto (u,1_A)$ and $\alpha\mapsto (1_U,\alpha)$.

Some pedantry about the use of the word ``action'' might now be advisable. The action of $U$ on $A$ is a right action described by a homomorphism of monoids $U \to \Symm(A)$, where $\Symm(A)$ is the symmetric group on $A$ (and is \emph{not} the group of monoid automorphisms). The action of $A$ on $U$ is a left action described by a homomorphism of monoids $A \to \Symm(U)$, again \emph{not} to $\Aut(U)$. In a phrase, both actions are actions \emph{of} monoids as monoids, but \emph{on} monoids as sets.

Brin \cite{brin07} regards the action $(\alpha,u) \mapsto \alpha^u$ of $U$ on $A$ as a family of maps from $A$ to itself parametrized by $U$ and defines properties of this family. For brevity we apply the same adjectives to the action itself but one should think of the family of maps. The action is called \emph{injective} if for all $u\in U$, $\alpha^u = \beta^u$ implies $\alpha = \beta$. It is \emph{surjective} if for every $\alpha \in A$ and $u \in U$ there exists a $\beta \in A$ with $\beta^u = \alpha$. The action is \emph{strongly confluent} if the following holds: if $u, v \in U$ have a least common left multiple $ru=sv$ and $\alpha = \beta^u = \gamma^v$ for some $\beta,\gamma \in A$ then there is a $\theta \in A$ such that $\theta^r = \beta$ and $\theta^s = \gamma$. Note that if the action is injective then for this to happen it is sufficient that $\theta^{ru} = \alpha$. The notions for the action of $A$ on $U$ are defined by analogy.

The following lemma can be found as Lemma~2.12 in \cite{brin07}, or as Lemma~3.15 in \cite{brin05}.

\begin{lemma}\label{lem:ZS_least_common_left_multiples}
 Let $U$ be a cancellative monoid with least common left multiples and let $A$ be a group. Let $U$ and $A$ act on each other via Zappa--Sz\'ep actions. Assume that the action $(\alpha,u) \mapsto \alpha^u$ of $U$ on $A$ is strongly confluent. Then $M = U \bowtie A$ has least common left multiples.
 
 A least common left multiple $(r,\alpha)(u,\theta) = (s,\beta)(v,\phi)$ of $(u,\theta)$ and $(v,\phi)$ in $M$ can be constructed so that $r(\alpha \cdot u) = s(\beta \cdot v)$ is the least common left multiple of $(\alpha \cdot u)$ and $(\beta \cdot v)$ in $U$. If $M$ is cancellative, every least common left multiple will have that property.
\end{lemma}

Being actions of monoids, Zappa--Sz\'ep actions are already determined by the actions of generating sets. It is not obvious, but also true, that they are often also determined by the actions of generating sets \emph{on} generating sets. This means that, in order to define the actions, we need only define $\alpha \cdot u$ and $u^\alpha$ where both $\alpha$ and $u$ come from generating sets. Brin \cite[pp.~768--769]{brin07} gives a sufficient condition for such partial actions to extend to well defined Zappa--Sz\'ep actions, which we restate here. Given sets $X$ and $Y$, let $X^*$ and $Y^*$ denote the free monoids generated respectively by them. Suppose maps $Y \times X \to Y^*, (\alpha,u) \mapsto \alpha^u$ and $Y \times X \to X, (\alpha,u) \mapsto \alpha \cdot u$ are given (so $\alpha\cdot u$ should be a single generator, but $\alpha^u$ may be a string of generators). Let $W$ be the set of relations $(\alpha u, (\alpha \cdot u)(\alpha^u))$ with $\alpha \in Y, u \in X$. Then
\[
\gen{X \cup Y \mid W}
\]
is a Zappa--Sz\'ep product of $X^*$ and $Y^*$. In particular, the above maps extend to Zappa--Sz\'ep actions $Y^* \times X^* \to Y^*$ and $Y^* \times X^* \to X^*$.

\begin{lemma}[{\cite[Lemma~2.14]{brin07}}]\label{lem:extend_zs-products}
 Let $U = \gen{X \mid R}$ and $A = \gen{Y \mid T}$ be presentations of monoids (with $X \cap Y = \emptyset$). Assume that functions $Y \times X \to Y^*, (\alpha,u) \mapsto \alpha^u$ and $Y \times X \to X, (\alpha,u) \mapsto \alpha \cdot u$ are given. Let $\sim_R$ and $\sim_T$ denote the equivalence relations on $X^*$ and $Y^*$ imposed by the relation sets $R$ and $T$.

 Extend the above maps to $Y^* \times X^*$ as above. Assume that the following are satisfied. If $(u,v) \in R$ then for every $\alpha \in Y$ we have $(\alpha \cdot u, \alpha \cdot v) \in R$ or $(\alpha \cdot v, \alpha \cdot u) \in R$, and also $\alpha^u \sim_T \alpha^v$. If $(\alpha,\beta) \in T$ then for all $u \in X$ we have $\alpha \cdot u = \beta \cdot u$ and $\alpha^u \sim_T \beta^u$.

 Then the lifted maps induce well defined Zappa--Sz\'ep actions and the restriction of the map $A \times U \to U$ to $A \times X$ has its image in $X$. A presentation for $U \bowtie A$ is
 \[
 \gen{X \cup Y \mid R \cup T \cup W}
 \]
 where $W$ consists of all pairs $(\alpha u, (\alpha \cdot u)(\alpha^u))$ for $(\alpha,u) \in Y \times X$.
\end{lemma}

\section{Cloning systems and generalized Thompson groups}\label{sec:defining_data}

\subsection{Brin--Zappa--Sz\'ep products and cloning systems}\label{sec:BZS}

To construct Thomp\-son-like groups we now consider Zappa--Sz\'ep products $\Fmonoid \bowtie G$ of the forest monoid $\Fmonoid$ with a group $G$.

\begin{definition}[BZS products]\label{def:BZS_product}
 Suppose we have Zappa--Sz\'ep actions $(g,\forest)\mapsto g\cdot \forest$ and $(g,\forest)\mapsto g^\forest$ on $G \times \Fmonoid$, for $G$ a group. For each standard generator $\lambda_k$ of $\Fmonoid$ the map $\clone_k=\clone_{\lambda_k} \colon G\to G$ given by $g\mapsto g^{\lambda_k}$ is called the $k$th \emph{cloning map}. If every such cloning map is injective, we call the actions \emph{Brin--Zappa--Sz\'ep (BZS) actions} and call the monoid $\Fmonoid \bowtie G$ the \emph{Brin--Zappa--Sz\'ep (BZS) product}.
\end{definition}

Since the action of $\Fmonoid$ on $G$ is a right action we will also write the cloning maps $\clone_k$ on the right.

The monoid $\Fmonoid$ is cancellative and has common right multiples, and the same is true of $G$, being a group. Since $G$ is a group these properties are inherited by $\Fmonoid \bowtie G$:

\begin{observation}\label{obs:BZS_cancellative_lcrms}
 A BZS product $\Fmonoid \bowtie G$ is cancellative and has (least) common right multiples. In particular it has a group of right fractions.
\end{observation}

\begin{proof}
 This follows easily from the statements about $\Fmonoid$ using the unique factorization in Zappa--Sz\'ep products and that $\forest$ is a right multiple and left factor of $(\forest,g)$.
\end{proof}

In Definition~\ref{def:BZS_product} we have already simplified the data needed to describe BZS products by using the fact that $\Fmonoid$ is generated by the $\lambda_k$. In a similar fashion the following lemma reduces the data needed to describe the action of $G$ on $\Fmonoid$. We denote by $\symm_\omega$ the group $\Symm(\N)$ of permutations of $\N$ and by $\symm_\infty\le \symm_\omega$ the subgroup of permutations that fix almost all elements of $\N$.

\begin{lemma}[Carets to carets]\label{lem:shuffling_carets}
 Let $\Fmonoid \bowtie G$ be a BZS product. The action of $G$ on $\Fmonoid$ preserves the set $\Lambda=\{\lambda_k\}_{k\in\N}$ and so induces a homomorphism $\rho \colon G\to \symm_\omega$. Conversely, the action of $G$ on $\Fmonoid$ is completely determined by~$\rho$ and $(\clone_k)_{k\in\N}$.
\end{lemma}

\begin{proof}
 For $g\in G$ and $\forest,\altforest \in \Fmonoid$, we know that $g\cdot(\forest\altforest)=(g\cdot \forest)(g^{\forest} \cdot \altforest)$ by Definition~\ref{def:ext_ZS_prod_monoids}. We show that the action of $G$ preserves $\Lambda$. If $g\cdot \lambda_k=\forest\altforest$ then $g^{-1} \cdot (\forest\altforest)=\lambda_k$, so one of $g^{-1} \cdot \forest$ or $(g^{-1})^{\forest}\cdot \altforest$ equals $1_\Fmonoid$. Again by Definition~\ref{def:ext_ZS_prod_monoids}, we see that either $\forest=1_\Fmonoid$ or $\altforest=1_\Fmonoid$. We conclude that $g\cdot \lambda_k$ equals $\lambda_\ell$ for some $\ell$ depending on $k$ and $g$. The map $\rho$ then is defined via $\rho(g)k=\ell$.
 
 To see that the action of $G$ on $\Fmonoid$ is determined by $\rho$ and $(\clone_k)$, we use repeated applications of the equation $g\cdot (\lambda_k \forest)=\lambda_{\rho(g)k} ((g)\clone_k \cdot \forest)$.
\end{proof}

As a consequence we see that the action of $G$ on $\Fmonoid$ preserves the length of an element:

\begin{corollary}\label{cor:length}
 There is a monoid homomorphism $\len \colon \Fmonoid \bowtie G \to \Nz$ taking $(\forest,g)$ to the length of $\forest$ in the standard generators. The kernel of $\len$ is $G = (\Fmonoid \bowtie G)^\times$. \qed
\end{corollary}

In particular, $\len$ is a length function in the sense of Section~\ref{sec:posets_from_monoids}. The induced morphism from the group of right fractions to $\Z$ (Lemma~\ref{lem:ore_fractions}) is also denoted $\len$.

The next result is a technical lemma that tells us that $\rho$ and the cloning maps always behave well together, in any BZS product.

\begin{lemma}[Compatibility]\label{lem:compatibility}
 Let $\Fmonoid \bowtie G$ be a BZS product. The homomorphism $\rho \colon G \to \symm_\omega$ and the maps $(\clone_k)_{k \in \N}$ satisfy the following compatibility condition for $k < \ell$:

 If $\rho(g)k < \rho(g)\ell$ then $\rho((g)\clone_\ell)k = \rho(g)k$ and $\rho((g)\clone_k)(\ell+1) = \rho(g)\ell + 1$.

 If $\rho(g)k > \rho(g)\ell$ then $\rho((g)\clone_\ell)k = \rho(g)k + 1$ and $\rho((g)\clone_k)(\ell+1) = \rho(g)\ell$.
\end{lemma}

\begin{proof}
 For $k < \ell$ we know that
 \[
 g \cdot (\lambda_\ell \lambda_k) = g \cdot(\lambda_k \lambda_{\ell+1}) \text{.}
 \]
 Writing this out using the axioms for Zappa--Sz\'ep products we obtain that
 \[
 (g \cdot \lambda_\ell) (g^{\lambda_\ell} \cdot \lambda_k) = (g \cdot \lambda_k)(g^{\lambda_k} \cdot \lambda_{\ell + 1})
 \]
 which can be rewritten using the action morphism $\rho$ as
 \[
 \lambda_{\rho(g)\ell} \lambda_{\rho(g^{\lambda_\ell})k} = \lambda_{\rho(g)k}\lambda_{\rho(g^{\lambda_k})(\ell + 1)}\text{.}
 \]
 Using the normal form for $\Fmonoid$ (see Proposition~\ref{prop:monoid_pres}) we can distinguish cases for how this could occur. The first case is that both pairs of indices
 \[
 (\rho(g) \ell, \rho(g^{\lambda_\ell})k) \text{ and } (\rho(g)k,\rho(g^{\lambda_k})(\ell + 1))
 \]
 are ordered increasingly and coincide. But this is impossible because $\rho(g) \ell \ne \rho(g) k$. The second case is that both pairs are ordered strictly decreasingly and coincide, which is impossible for the same reason. The remaining two cases have that one pair is ordered increasingly and the other strictly decreasingly. In either case the monoid relation now yields a relationship among the indices, namely either
 \[
 \rho(g^{\lambda_k})(\ell+1)-1 = \rho(g)\ell > \rho(g^{\lambda_\ell})k = \rho(g)k
 \]
 or
 \[
 \rho(g)\ell = \rho(g^{\lambda_k})(\ell+1) < \rho(g)k = \rho(g^{\lambda_\ell})k - 1 \text{.}
 \]
 Finally, replacing the action of $\lambda_k$ by the map $\clone_k$ yields the result.
\end{proof}

The compatibility condition can also be rewritten as
\begin{equation}
\label{eq:compatibility_cases}
\rho((g)\clone_\ell)(k) =
\left\{
\begin{array}{ll}
\rho(g)(k) & k < \ell, \rho(g) k < \rho(g) \ell,\\
\rho(g)(k) + 1 & k < \ell, \rho(g) k > \rho(g) \ell,\\
\rho(g)(k-1) & k-1 > \ell, \rho(g)(k-1) < \rho(g) \ell,\\
\rho(g)(k-1)+1 & k-1 > \ell,\rho(g)(k-1) > \rho(g)\ell\text{.}
\end{array}
\right.
\end{equation}

\medskip

Lemma~\ref{lem:shuffling_carets} said that the action of $G$ on $\Fmonoid$ is uniquely determined by $\rho$ and the cloning maps. The action of $\Fmonoid$ on $G$ is also uniquely determined by the cloning maps, simply because $\Fmonoid$ is generated by the $\lambda_k$. Our findings can be summarized as:

\begin{proposition}[Uniqueness]\label{prop:BZS_uniqueness}
 A BZS product $\Fmonoid \bowtie G$ induces a homomorphism $\rho \colon G \to \symm_\omega$ and injective maps $\clone_k \colon G \to G, k \in \N$ satisfying the following conditions for $k, \ell \in \N$ with $k <\ell$ and $g,h \in G$:
 \smallskip
 \begin{enumerate}[label={(CS\arabic*)}, ref={CS\arabic*}, leftmargin=*]
  \item $(gh) \clone_k = (g)\clone_{\rho(h)k}(h)\clone_k$. \hfill (Cloning a product) \label{item:cs_cloning_a_product}
  \item $\clone_\ell \circ \clone_k = \clone_k \circ \clone_{\ell+1} $. \hfill (Product of clonings) \label{item:cs_product_of_clonings}
  \item If $\rho(g)k < \rho(g)\ell$ then $\rho((g)\clone_\ell)k = \rho(g)k$ and \\
   \indent $\rho((g)\clone_k)(\ell+1) = \rho(g)\ell + 1$.\\
   If $\rho(g)k > \rho(g)\ell$ then $\rho((g)\clone_\ell)k = \rho(g)k + 1$ and \\
   \indent $\rho((g)\clone_k)(\ell+1) = \rho(g)\ell$. \hfill(Compatibility) \label{item:cs_compatibility}
 \end{enumerate}
 
 \smallskip

 The BZS product is uniquely determined by these data. \qed
\end{proposition}

The converse is also true:

\begin{proposition}[Existence]\label{prop:BZS_existence}
 Let $G$ be a group, $\rho \colon G \to \symm_\omega$ a homomorphism and $(\clone_k)_{k \in \N}$ a family of injective maps from $G$ to itself. Assume that for $k < \ell$ and $g,h \in G$ the conditions \eqref{item:cs_cloning_a_product}, \eqref{item:cs_product_of_clonings} and \eqref{item:cs_compatibility} in Proposition~\ref{prop:BZS_uniqueness} are satisfied.

 Then there is a well defined BZS product $\Fmonoid \bowtie G$ corresponding to these data.
\end{proposition}

\begin{proof}
 We will verify the assumptions of Lemma~\ref{lem:extend_zs-products}. This will produce a Zappa--Sz\'ep action, which will be a Brin--Zappa--Sz\'ep action by construction. We take $U$ to be $\Fmonoid$ with the presentation
 \[
 \gen{\lambda_k  \text{ for }  k \in \N \mid (\lambda_\ell\lambda_k,\lambda_k\lambda_{\ell+1})  \text{ for } k < l}\text{.}
 \]
 Let $R$ denote the set of relations used here and let $R^{\text{sym}}$ be the symmetrization. We take $A$ to be $G$ with the trivial presentation
 \[
 \gen{g  \text{ for } g \in G \mid (gh,g') \text{ for } gh=g'}\text{.}
 \]
 The maps on generators are defined as $g^{\lambda_k} \defeq (g)\clone_k$ and $g \cdot \lambda_k \defeq \lambda_{\rho(g)k}$.
 
 First, for $k < \ell$ and $g \in G$ we need to verify that
 \[
 (g \cdot (\lambda_\ell \lambda_k), g \cdot (\lambda_k \lambda_{\ell+1})) \in R^{\text{sym}} \quad \text{and} \quad g^{\lambda_\ell \lambda_k} = g^{\lambda_k \lambda_{\ell+1}} \text{.}
 \]
 The latter of these is just condition~\eqref{item:cs_product_of_clonings}. The former condition means that
 \[
 (\lambda_{\rho(g)\ell} \lambda_{\rho((g)\clone_\ell)k}, \lambda_{\rho(g)k} \lambda_{\rho((g)\clone_k)(\ell+1)})
 \]
 should lie in $R^{\text{sym}}$. If $\rho(g)k > \rho(g)\ell$ we can use condition~\eqref{item:cs_compatibility} to rewrite this as
 \[
 (\lambda_{\rho(g)\ell}\lambda_{\rho(g)k+1},\lambda_{\rho(g)k}\lambda_{\rho(g)\ell})
 \]
 which is in $R^{\text{sym}}$. If $\rho(g)k < \rho(g)\ell$ then the tuple is
 \[
 (\lambda_{\rho(g)\ell}\lambda_{\rho(g)k},\lambda_{\rho(g)k}\lambda_{\rho(g)\ell+1})
 \]
 which already lies in $R$.
 
 Second, for every relation $(gh,g')$ of $G$ and every $k \in \N$ we have to verify that
 \[
 (gh) \cdot \lambda_k = g' \cdot \lambda_k \quad \text{and} \quad (gh)^{\lambda_k} = {(g')}^{\lambda_k}
 \]
 for $k \in \N$. The former is not really a condition because the partial action was already defined using $G$ (rather than the free monoid spanned by $G$). The latter means that we need
 \[
 {(g')}^{\lambda_k} = g^{\lambda_{\rho(h)k}}h^{\lambda_k}
 \]
 which is just condition~\eqref{item:cs_cloning_a_product}.
\end{proof}

\begin{definition}\label{def:cloning_system}
 Let $G$ be a group, $\rho \colon G \to S_\omega$ a homomorphism and $(\clone_k)_{k \in \N} \colon G \to G$ a family of maps, also denoted $\clone_*$ for brevity. The triple $(G,\rho,\clone_*)$ is called a \emph{cloning system} if the data satisfy conditions \eqref{item:cs_cloning_a_product}, \eqref{item:cs_product_of_clonings} and \eqref{item:cs_compatibility} above. We may also refer to $\rho$ and $\clone_*$ as a forming a \emph{cloning system on} $G$.
\end{definition}

We now discuss an extended example, of the infinite symmetric group, and show that we have a cloning system. It is exactly the cloning system that gives rise to Thompson's group $V$.

\begin{example}[Symmetric groups]\label{ex:symm_gps}
 Let $G=\symm_\infty$. Let $\rho \colon \symm_\infty \to \symm_\omega$ just be inclusion. The action of $G$ on $\Fmonoid$ is thus given by $g \cdot \lambda_k = \lambda_{\rho(g)k} = \lambda_{gk}$.
 
 Since we will use the specific cloning maps in this example even in the future general setting, we will give them their own name, $\symmclone_\ell$. They are defined by the formula
 \begin{equation}
\label{eq:symmclone}
 ((g)\symmclone_k)(m) =
 \left\{
 \begin{array}{ll}
 gm & m \le k, gm \le gk\text{,}\\
 gm+1 & m < k, gm > gk\text{,}\\
 g(m-1) & m > k, g(m-1) < gk\text{,}\\
 g(m-1) + 1 & m > k, g(m-1) \ge gk\text{.}
 \end{array}
 \right.
 \end{equation}
 
 If we draw permutations as strands crossing each other, the word ``cloning'' becomes more or less literal: applying the $k$th cloning map creates a parallel copy of the $k$th strand, where we count the strands at the bottom. See Figure~\ref{fig:symm_clone} for an example.
 
 \begin{figure}[htb]
 \centering
 \begin{tikzpicture}[line width=0.8pt]
  \draw (0,-2) -- (1,0); \draw (1,-2) -- (0,0); \draw (2,-2) -- (2,0);
  \node at (3,-1) {$\stackrel{\symmclone_2}{\longrightarrow}$};
  \begin{scope}[xshift=4cm]
   \draw (0,-2) -- (2,0); \draw (1,-2) -- (0,0); \draw (2,-2) -- (1,0); \draw (3,-2) -- (3,0);
  \end{scope}
 \end{tikzpicture}
 \caption{An example of cloning in symmetric groups. Here we see that $(1~2)\symmclone_2 = (1~3~2)$.}
 \label{fig:symm_clone}
 \end{figure}
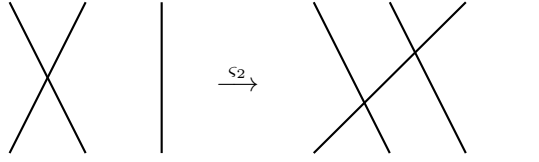

 We will prove that this defines a cloning system by verifying~\eqref{item:cs_cloning_a_product},~\eqref{item:cs_product_of_clonings} and~\eqref{item:cs_compatibility}. For this example we will just verify them directly, and not use any specific presentation for $\symm_\infty$. It is immediate from~\eqref{eq:symmclone} that the compatibility condition~\eqref{item:cs_compatibility} in the formulation~\eqref{eq:compatibility_cases} is satisfied.
 
 To aid in checking condition~\eqref{item:cs_cloning_a_product}, we define two families of maps, $\pi_k \colon \N \to \N$ and $\tau_k \colon \N \to \N$, for $k \in \N$:
 \begin{equation}
 \label{eq:pi_tau}
 \pi_k(m) = \left\{
 \begin{array}{ll}
 m & m \le k\text{,}\\
 m-1 & m > k
 \end{array}
 \right.
 \quad\text{and}\quad
 \tau_k(m) = \left\{
 \begin{array}{ll}
 m & m \le k\text{,}\\
 m + 1& m > k\text{.}
 \end{array}
 \right.
 \end{equation}
 Note that $\pi_k \circ \tau_k = \id$ and $\tau_k \circ \pi_k(m) = m$, unless $m = k+1$ in which case it equals $m-1$. In the $m=k+1$ case, we see that
 \[
 (gh)\symmclone_k(k+1) = gh(k) +1 = (g)\symmclone_{hk}(hk +1)= (g)\symmclone_{hk} (h)\symmclone_k(k+1) \text{,}
 \]
 by repeated use of the last case in the definition. It remains to check condition~\eqref{item:cs_cloning_a_product} in the $m \ne k+1$ case. According to the definitions, we have
 \[
 ((g)\symmclone_k)(m) = \tau_{gk}(g\pi_k(m))
 \]
 whenever $m \ne k+1$. Using this we see that
 \begin{align*}
 ((g)\symmclone_{hk}) \circ ((h)\symmclone_k)(m) &= \tau_{ghk} g \pi_{hk} \circ \tau_{hk} h \pi_k (m)\\
 & = \tau_{ghk} gh \pi_k(m)\\
 & = ((gh)\symmclone_k)(m)
 \end{align*}
 for $m \ne k+1$.
 
 To check condition~\eqref{item:cs_product_of_clonings}, we consider $k<\ell$. We first verify, from the definition, the special cases
 \begin{align*}
 ((g)\symmclone_\ell \circ \symmclone_k)(k+1) &= gk+1 = ((g)\symmclone_k \circ \symmclone_{\ell+1})(k+1)\quad  \text{and}\\
 ((g)\symmclone_\ell \circ \symmclone_k) (\ell+2) &= g\ell+2 = ((g)\symmclone_k \circ \symmclone_{\ell+1})(\ell+2)\text{.}
 \end{align*}
 For the remaining case, when $m \ne k+1,\ell+2$, we have
 \begin{align*}
 ((g)\symmclone_\ell \circ \symmclone_k) (m) &= \tau_k \tau_\ell g \pi_\ell \pi_k (m)\quad\text{and}\\
 ((g)\symmclone_k \circ \symmclone_{\ell+1}) (m) &= \tau_{\ell+1}\tau_k  g \pi_k \pi_{\ell+1} (m)
 \end{align*}
 and it is straightforward to check that
 \begin{equation}
 \label{eq:pi_and_tau_relations}
 \pi_\ell \pi_k = \pi_k \pi_{\ell+1} \text{ and }\tau_k \tau_\ell = \tau_{\ell+1}\tau_k\text{.}
 \end{equation}
 
 We conclude that $(\symm_\infty,\rho,(\symmclone_k)_k)$ is a cloning system.
\end{example}

\begin{remark}
Besides the example of symmetric groups there are two more examples of cloning systems previously existing in the literature (though of course not using this language): they are for the families of braid groups and pure braid groups and were used in \cite{brin07,brady08} to construct $\Vbr$ and $\Fbr$.
\end{remark}

\begin{observation}[Simplified compatibility]\label{obs:sync}
 Condition~\eqref{item:cs_compatibility} in Proposition~\ref{prop:BZS_uniqueness} can equivalently be rewritten as
 \[
 \rho((g)\clone_k)(i) = (\rho(g))\symmclone_k (i) \text{ for all } i\ne k,k+1 \text{.}
 \]
\end{observation}

All the examples in the later sections satisfy the condition in Observation~\ref{obs:sync} even when $i=k,k+1$.

\begin{remark}\label{rmk:axioms_via_pres}
Proposition~\ref{prop:BZS_existence} is an application of Lemma~\ref{lem:extend_zs-products} to the trivial presentation. As this example demonstrates, it can be rather involved to verify the conditions for a cloning system. If the group in question comes equipped with a presentation involving only short relations, it may be easier to re-run the proof of Proposition~\ref{prop:BZS_existence} with that presentation by applying Lemma~\ref{lem:extend_zs-products}. In this case one has to check \eqref{item:cs_product_of_clonings} and \eqref{item:cs_compatibility} only on generators, but also has to check a variant of \eqref{item:cs_cloning_a_product} for every relation.
\end{remark}

We finish by discussing the case when we have least common left multiples. Let $\clone_*$ be the cloning maps of a cloning system. For $\forest = \lambda_{k_1} \cdots \lambda_{k_r}$ define $\clone_\forest \defeq \clone_{k_1} \circ \cdots \circ \clone_{k_r}$. Note that this is well defined by condition~\eqref{item:cs_product_of_clonings} and is just the map $g \mapsto g^\forest$.

\begin{observation}\label{obs:BZS_lclm}
 Let $G$ be a group and let $(\rho,\clone_*)$ be a cloning system on $G$. The action of $\Fmonoid$ on $G$ defines a strongly confluent family if and only if $\image(\clone_{\forest_1}) \cap \image(\clone_{\forest_2}) = \image(\clone_{\altforest})$ whenever $\forest_1$ and $\forest_2$ have least common left multiple $\altforest$.
 
 In particular the BZS product $\Fmonoid \bowtie G$ has least common left multiples in that case.
\end{observation}

\begin{proof}
 The proof is obtained just by unraveling the definition and using the remark before Lemma~\ref{lem:ZS_least_common_left_multiples}. Assume that the above condition holds. Write $\altforest = \altforest_1 \forest_1 = \altforest_2 \forest_2$. Assume that $g = g_1^{\forest_1} = g_2^{\forest_2}$, that is, $g \in \image(\clone_{\forest_1}) \cap \image(\clone_{\forest_2})$. By assumption there is an $h \in G$ such that $g = (h)\clone_{\altforest}$. That is $g = h^{\altforest} = h^{\altforest_1 \forest_1} = g_1^{\forest_1}$. Injectivity of the action of $\Fmonoid$ on $G$ now implies $h^{\altforest_1} = g_1$. A similar argument shows $h^{\altforest_2} = g_2$.
 
 Conversely assume that the action of $\Fmonoid$ on $G$ is strongly confluent and write $\altforest$ as before. Let $g \in \image(\clone_{\forest_1}) \cap \image(\clone_{\forest_2})$. Write $g = (g_1)\clone_{\forest_1}$ and $g = (g_2)\clone_{\forest_2}$, that is $g = g_1^{\forest_1}$ and $g = g_2^{\forest_2}$. By strong confluence there is an $h \in G$ such that $h^{\altforest_1} = g_1$ and $h^{\altforest_2} = g_2$. Then $g = h^{\altforest} = (h)\clone_{\altforest}$ as desired.
\end{proof}

To check this global confluence condition one either needs a good understanding of the action of $\Fmonoid$ on $G$ (as was the case for $\Vbr$ \cite[Section~5.3]{brin07}) or one has to reduce it to local confluence statements.

\subsection{Interlude: hedges}\label{sec:hedges}
In the above example of the symmetric group, the action of $\Fmonoid$ on $\symm_\infty$ factors through an action of a proper quotient. This amounts to a further relation being satisfied in addition to the product of clonings relation \eqref{item:cs_product_of_clonings}. The quotient turns out to be what Brin~\cite{brin07} called the monoid of \emph{hedges}. Without going into much detail we want to explain the action of the hedge monoid on $\symm_\infty$.

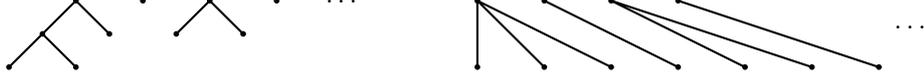
\begin{figure}[htb]
\centering
\begin{tikzpicture}[line width=0.8pt, scale=0.44]
  \draw
   (-2,-2) -- (0,0) -- (1,-1)   (-1,-1) -- (0,-2);
  \filldraw
   (-2,-2) circle (1.5pt)   (-1,-1) circle (1.5pt)   (0,-2) circle (1.5pt)   (0,0) circle (1.5pt)   (1,-1) circle (1.5pt);
  \draw
   (3,-1) -- (4,0) -- (5,-1);
  \filldraw
   (3,-1) circle (1.5pt)   (4,0) circle (1.5pt)   (5,-1) circle (1.5pt)   (2,0) circle (1.5pt)   (6,0) circle (1.5pt);
  \node at (8,0) {$\dots$};

  \begin{scope}[xshift=12cm]
   \draw
    (0,0) -- (0,-2)  (0,0) -- (2,-2)   (0,0) -- (4,-2)
    (2,0) -- (6,-2)
    (4,0) -- (8,-2)  (4,0) -- (10,-2)
    (6,0) -- (12,-2);
   \filldraw
    (0,0) circle (1.5pt)   (2,0) circle (1.5pt)   (4,0) circle (1.5pt)   (6,0) circle (1.5pt)
    (0,-2) circle (1.5pt)   (2,-2) circle (1.5pt)   (4,-2) circle (1.5pt)   (6,-2) circle (1.5pt)   (8,-2) circle (1.5pt)   (10,-2) circle (1.5pt) (12,-2) circle (1.5pt);
   \node at (13,-.8) {$\dots$};
  \end{scope}
\end{tikzpicture}
\caption{A forest and the corresponding hedge.}
\label{fig:forest_hedge}
\end{figure}

The \emph{hedge monoid} $\Hmonoid$ is the monoid of monotone surjective maps $\Nnz \to \Nnz$. Multiplication is given by composition: $f \cdot h = f \circ h$. There is an action of $\symm_\infty$ on $\Hmonoid$ given by the property that, for $g\in\symm_\infty$ and $f\in\Hmonoid$, the cardinality of $(g \cdot f)^{-1}(i)$ is that of $f^{-1}(g^{-1}i)$. There is an obvious equivariant morphism $c \colon \Fmonoid \to \Hmonoid$ (see Figure~\ref{fig:forest_hedge}) given by $c(\lambda_k) = \eta_k$ where
\[
\eta_k(m) = \left\{
\begin{array}{ll}
m & m \le k\text{,}\\
m-1 & m > k\text{.}
\end{array}
\right.
\]
This morphism is surjective but not injective, in fact (see \cite[Proposition~4.4]{brin07}):

\begin{lemma}
 The monoid $\Hmonoid$ has the presentation
 \[
 \gen{\eta_k, k \in \N\mid \eta_\ell \eta_k = \eta_k \eta_{\ell+1}, \ell \ge k}\text{.}
 \]
\end{lemma}

Observe that the only difference between this and the presentation of $\Fmonoid$ is that the relation also holds for $\ell = k$, rather than only for $\ell > k$. It turns out that the action of $\Fmonoid$ on $\symm_\infty$ defined in Example~\ref{ex:symm_gps} factors through $c$:

\begin{observation}
 The maps $\symmclone_k$ defined in \eqref{eq:symmclone} satisfy $\symmclone_k \symmclone_k = \symmclone_k \symmclone_{k+1}$. Thus they define an action of $\Hmonoid$ on $\symm_\infty$.
\end{observation}

\begin{proof}
The verification of \eqref{item:cs_product_of_clonings} above extends to the case $k = \ell$.
\end{proof}

\subsection{Filtered cloning systems}\label{sec:filtered_cloning_systems}

Typically one will want to think of Thompson's group $V$ not as built from $\symm_\infty$ but rather from the family $(\symm_n)_{n \in \N}$. We will now describe this approach. We regard $\symm_\infty$ as the direct limit $\varinjlim \symm_n$ where the maps $\iota_{m,n} \colon \symm_m \to \symm_n$ are induced by the inclusions $\{1,\ldots,m\} \into \{1,\ldots,n\}$.

Let $(G_n)_{n\in\N}$ be a family of groups with monomorphisms $\iota_{m,n} \colon G_m \to G_n$ for each $m\le n$. For convenience we will sometimes write $G_*$ for $(G_n)_{n \in \N}$; note that in this case the index set is always $\N$. The maps $\iota_{m,n}$ will be written on the right, e.g., $(g)\iota_{m,n}$ for $g\in G_m$. Suppose that $\iota_{m,m}=\id$ and $\iota_{m,n} \circ \iota_{n,\ell} = \iota_{m,\ell}$ for all $m\le n\le \ell$. Then $((G_n)_{n\in\N},(\iota_{m,n})_{m\le n})$ is a directed system of groups with a direct limit $G \defeq \varinjlim G_n$. Since all the $\iota_{m,n}$ are injective, we may equivalently think of a group $G$ filtered by subgroups $G_n$.

Consider injective maps $\clone^n_k \colon G_n \to G_{n+1}$ for $k,n \in \N, k \le n$. We call such maps a \emph{family of cloning maps} for the directed system $(G_n)_{n \in \N}$ if for $m,k \le n$ they satisfy
\begin{equation}
\label{eq:cloning_graded}
\iota_{m,n} \circ \clone^n_k = 
\left\{
\begin{array}{ll}
\clone^m_k \circ \iota_{m+1,n+1} &\text{ if } k\le m \\
\iota_{m,n+1} &\text{ if } m<k\text{.}
\end{array}
\right.
\end{equation}

This amounts to setting $\clone^n_k \defeq \iota_{n,n+1}$ for $k > n$ and requiring that
\[
\iota_{m,n} \circ \clone^n_k = \clone^m_k \circ \iota_{m+1,n+1}\text{,}
\]
i.e., that the family $(\clone_k^n)_{n \in \N}$ defines a morphism of directed systems of sets. From that it is clear that a family of cloning maps induces a family of injective maps $\clone_k \colon G \to G$ by setting
\[
(g)\iota_{n} \circ \clone_k = (g)\clone_k^n \circ \iota_{n+1}
\]
for $g \in G_n$. Here $\iota_n \colon G_n \to G$ denotes the map given by the universal property of $G$.

\begin{definition}[Properly graded]\label{def:properly_graded}
 We say that the cloning maps are \emph{properly graded} if the following strong confluence condition holds: if $g \in G_{n+1}$ can be written as $(h)\clone_k^n = g = (\bar{g})\iota_{n,n+1}$ then there is an $\bar{h} \in G_{n-1}$ with $(\bar{h})\clone_k^{n-1} = \bar{g}$ and $(\bar{h})\iota_{n-1,n} = h$.
\end{definition}

In view of the injectivity of all maps involved this is equivalent to saying that
\begin{equation}\label{eq:properly_graded}
 \image \clone_k^n \cap \image \iota_{n,n+1} \subseteq \image (\iota_{n-1,n} \circ \clone_k^n)
\end{equation}
(where the converse inclusion is automatic) or to saying that the diagram
\begin{diagram}
 G_{n-1} & \rTo^{\iota_{n-1,n}} & G_n\\
 \dTo^{\clone_k^{n-1}} && \dTo^{\clone_k^n} \\
 G_n & \rTo^{\iota_{n,n+1}} & G_{n+1}
\end{diagram}
is a pullback diagram of sets. A formulation in terms of the direct limit $G$ is that if $(h)\clone_k \in G_n$ for $k\le n$ then $h \in G_{n-1}$. Note that a filtered cloning system satisfying the confluence condition of Observation~\ref{obs:BZS_lclm} is automatically properly graded.

\begin{example}
 Take $G_n = \symm_n$ as in Example~\ref{ex:symm_gps}. A family of cloning maps $\symmclone^n_k$ is obtained by restriction of the maps from Example~\ref{ex:symm_gps}:
 \begin{equation}
  \label{eq:symmetric_cloning_maps}
  \symmclone^n_k \defeq \symmclone_k|_{\symm_n}^{\symm_{n+1}}\text{.}
 \end{equation}
 
 This family of cloning maps is properly graded: if $g \in \image \iota_{n,n+1}$ then $g$ fixes $n+1$; if moreover $g = (h)\symmclone_k$ then it follows from~\eqref{eq:symmclone} that $h$ fixes $n$ so $h \in \image \iota_{n-1,n}$.
\end{example}

Now suppose further that we have a family of homomorphisms $\rho_n \colon G_n \to \symm_n$ for each $n\in\N$ that are compatible with the directed systems, i.e., $\rho_n((g)\iota_{m,n})= (\rho_m(g))\iota_{m,n}$ for $m < n$ and $g \in G_m$. Let $\rho \colon G \to \symm_\infty$ be the induced homomorphism. We are of course interested in the case when $\rho$ and the family $(\clone_k)_{k \in \N}$ define a cloning system on $G$. The corresponding defining formulas are obtained by adding decorations to the formulas from Section~\ref{sec:BZS}:

\begin{definition}[Cloning system]\label{def:filtered_cloning_system}
 Let $((G_n)_{n \in \N}, (\iota_{m,n})_{m \le n})$ be an injective directed system of groups. Let $(\rho_n)_{n \in \N} \colon G_n \to \symm_n$ be a homomorphism of directed systems of groups and let $(\clone^n_k)_{k \le n} \colon G_n \to G_{n+1}$ be a family of cloning maps. The quadruple
 \[
 ((G_n)_{n \in \N},(\iota_{m,n})_{m \le n},(\rho_n)_{n \in \N},(\clone^n_k)_{k \le n})
 \]
 is called a \emph{cloning system} if the following hold for all $k \le n$, $k<\ell$, and $g,h\in G_n$:
 \begin{enumerate}[label={(FCS\arabic*)}, ref={FCS\arabic*}, leftmargin=*]
  \item $(gh) \clone_k^n = (g)\clone_{\rho(h)k}^n(h)\clone_k^n$. \label{item:fcs_cloning_a_product}\hfill (Cloning a product)
  \item $\clone_\ell^n \circ \clone_k^{n+1} = \clone_k^n \circ \clone_{\ell+1}^{n+1}$.\label{item:fcs_product_of_clonings}\hfill (Product of clonings)
  \item $\rho_{n+1}((g)\clone^n_k)(i) = (\rho_n(g))\symmclone^n_k (i)$ for all $i\ne k,k+1$ \label{item:fcs_compatibility}\hfill(Compatibility)
 \end{enumerate}
 We may also refer to $\rho_*$ and $(\clone_k^n)_{k \le n}$ as forming a \emph{cloning system on} the directed system $G_*$. The cloning system is \emph{properly graded} if the cloning maps are properly graded.
\end{definition}

Note that condition~\eqref{item:fcs_compatibility} is phrased more concisely than \eqref{item:cs_compatibility}, but this is just in light of Observation~\ref{obs:sync}. Again, condition~\eqref{item:fcs_compatibility} will in practice often be satisfied even when $i=k,k+1$.

\begin{observation}\label{obs:filtered_and_filtration_preserving_cloning_systems}
 Let $(G_n)_{n \in \N}$ be an injective directed system of groups. A cloning system on $(G_n)_{n \in \N}$ gives rise to a cloning system on $G \defeq \varinjlim G_n$. Conversely a cloning system on $G$ gives rise to a cloning system on $(G_n)_{n \in \N}$ provided $(G_n)\clone^n_k \subseteq G_{n+1}$ and $\rho_n(G_n) \subseteq \symm_n$.
\end{observation}

We will usually not distinguish explicitly between a cloning system on $G_*$ and a cloning system on $\varinjlim G_*$ that preserves the filtration. In particular, given a cloning system on a directed system of groups we will implicitly define $\rho \defeq \varinjlim \rho_n$ and $\clone_k \defeq \varinjlim \clone_k^n$.

\subsection{Thompson groups from cloning systems}\label{sec:BZS_to_thomp}

Let $(G,\rho,(\clone_k)_{k\in\N})$ be a cloning system and let $\Fmonoid \bowtie G$ be the associated BZS product. We now define a group $\Thomphat{G}$ for the cloning system. This is a supergroup of the actual group $\Thomp{G_*}$ that we construct later in the case when $G$ arises as a limit of a family $(G_n)_n$ (Definition~\ref{def:thompson}).

\begin{definition}[Thompson group of a cloning system]\label{def:big_thompson}
The group of right fractions of $\Fmonoid \bowtie G$ is denoted by $\Thomphat{G}$ and is called the \emph{large generalized Thompson group} of $G$. If more context is required we denote it $\Thomphat{G,\rho,(\clone_k)_k}$ and call it the large generalized Thompson group of the cloning system $(G,\rho,(\clone_k)_k)$.
\end{definition}

By Observation~\ref{obs:BZS_cancellative_lcrms} and Theorem~\ref{thm:ore} every element $t$ of $\Thomphat{G}$ can be written as $t = (\forest_-,g)(\forest_+,h)^{-1}$ for some $\forest_-,\forest_+ \in \Fmonoid$ and $g,h \in G$. If it can also be written $t = (\forest_-,g')(\forest_+,h')^{-1}$ then $gh^{-1} = g'{h'}^{-1}$. It therefore makes sense to represent it by just the triple $(\forest_-,gh^{-1},\forest_+)$. Of course, this representation is still not unique, for example $(\forest,1_G,\forest)$ represents the identity element for every $\forest \in \Fmonoid$. We will denote the element represented by $(\forest_-,g,\forest_+)$ by $[\forest_-,g,\forest_+]$. Note that $[\forest_-,g,\forest_+]^{-1} = [\forest_+,g^{-1},\forest_-]$. We will call $(\forest_-(g\cdot \altforest),g^\altforest,\forest_+ \altforest)$ an \emph{expansion} of $(\forest_-,g,\forest_+)$, and the latter a \emph{reduction} of the former, so any reduction or expansion of a triple $(\forest_-,g,\forest_+)$ represents the same element of $\Thomphat{G}$ as $(\forest_-,g,\forest_+)$.

\medskip

Now assume that $G = \varinjlim G_n$ is an injective direct limit of groups $(G_n)_{n \in \N}$ and that the cloning system is a cloning system on $(G_n)_{n \in \N}$. Recall from Section~\ref{sec:forest_monoid} that a forest $\forest$ is called semisimple if all but its first tree are trivial and in that case its number of feet is the number of leaves of the first tree.

We collect some facts about semisimple elements of $\Fmonoid$.

\begin{observation}\label{obs:forest_semisimple_elements}
 Let $\forest,\forest_1,\forest_2,\altforest \in \Fmonoid$.
 \begin{enumerate}
  \item The number of feet of a non-trivial semisimple element of $\Fmonoid$ is its length plus one.
  \item Any two semisimple elements of $\Fmonoid$ have a semisimple common right multiple.\\
   More generally, any two elements of rank at most $m$ have a common right multiple of rank at most $m$.\label{item:forest_semisimple_right_common_multiple}
  \item If $\forest$ is semisimple with $n$ feet then $\forest\altforest$ is semisimple if and only if $\altforest$ has rank at most $n$.\\
   More generally, if $\forest$ is non-trivial of rank $m$ and length $n-m$ then $\forest\altforest$ has rank $m$ if and only if $\altforest$ has rank at most $n$.\label{item:forest_right_multiple_rank}
  \item If $\forest_1, \forest_2$ are semisimple with $n$ feet then $\forest_1\forest$ is semisimple if and only if $\forest_2\forest$ is.
 \end{enumerate}
\end{observation}

Now we upgrade these facts to $\Fmonoid \bowtie G$. We say that an element $(\forest,g) \in \Fmonoid \bowtie G$ is \emph{semisimple} if $\forest$ is semisimple with $n$ feet (for some $n$) and $g\in G_n$. In this case we also say $(\forest,g)$ has \emph{$n$ feet}.

\begin{lemma}\label{lem:semisimple_elements}
 Let $\forest,\forest_1,\forest_2,\altforest \in \Fmonoid$ and $g,h \in G$.
 \begin{enumerate}
  \item The number of feet of a semisimple element of $\Fmonoid \bowtie G$ is its length plus one.\label{item:type_length}
  \item Any two semisimple elements of $\Fmonoid \bowtie G$ have a semisimple common right multiple.\label{item:semisimple_right_common_multiple}
  \item If $(\forest,g)$ is semisimple then $(\forest,g) \altforest = (\forest (g \cdot \altforest),g^{\altforest})$ is semisimple if and only if $\forest (g \cdot \altforest)$ is semisimple.\label{item:monoid_semisimplicity_cancel_g}
  \item If $(\forest,g)$ is semisimple with $n$ feet then $(\forest,g) \altforest$ is semisimple if and only if $\altforest$ has rank at most $n$.\label{item:monoid_semisimplicity_multiple}
  \item If $(\forest_1,g)$ and $(\forest_2,h)$ are semisimple with same number of feet then $(\forest_1,g) \forest$ is semisimple if and only if $(\forest_2,g)\forest$ is semisimple.\label{item:monoid_semisimplicity_common_multiple}
 \end{enumerate}
\end{lemma}

\begin{proof}
 The first statement is clear by definition. The second statement can be reduced to the corresponding statement in $\Fmonoid$ because $\forest$ is a right multiple of $(\forest,g)$.
 
 In the third statement only the implication from right to left needs justification, namely that $g^\altforest \in G_n$ where $n$ is the number of feet of $\forest (g \cdot \altforest)$. This is because if $g \in G_m$ and $\len E = k$ then $g^E \in G_{m+k}$ as can be seen by induction on $\len E$ using $\clone_k(G_n) \subseteq G_{n+1}$.
 
 For \eqref{item:monoid_semisimplicity_multiple} note that $g \in G_n$. But $\rho(G_n) \subseteq S_n$ so having rank at most $n$ is preserved under the action of $G_n$, i.e., $\rk (g \cdot F) \le n \Leftrightarrow  \rk F \le n$. Thus the statement follows from the one for $\Fmonoid$. The last statement is immediate from \eqref{item:monoid_semisimplicity_multiple}.
\end{proof}

\begin{definition}[Simple]\label{def:simple}
 A triple $(\forest_-,g,\forest_+)$ (and the element $[\forest_-,g,\forest_+]$ represented by it) is said to be \emph{simple} if $\forest_-$ and $\forest_+$ are semisimple, both of them with $n$ feet and $g \in G_n$. This is the case if it can be written as $(\forest_-,g) (\forest_+,h)^{-1}$ with both factors semisimple with the same number of feet.
\end{definition}

\begin{proposition}\label{prop:product_of_simple}
 The set of simple elements in $\Thomphat{G}$ is a subgroup.
\end{proposition}

\begin{proof}
The proof closely follows \cite[Section~7]{brin07}.

Consider two simple elements $s = [\forest_-,g,\forest_+], t=[\altforest_-,h,\altforest_+]$. Let
 \begin{equation}
 \label{eq:common_simple_right_multiple}
 \forest_+ \forest = \altforest_- \altforest
 \end{equation}
 be a semisimple common right multiple of $\forest_+$ and $\altforest_-$ (Observation~\ref{obs:forest_semisimple_elements} \eqref{item:forest_semisimple_right_common_multiple}). Then
 \begin{align}
 st & = \forest_- g \forest \altforest^{-1} h \altforest_+^{-1}\nonumber\\
 & = (\forest_-(g \cdot \forest),g^\forest)(\altforest_+(h^{-1} \cdot \altforest), (h^{-1})^\altforest)^{-1}\label{eq:simple_representative}\\
 & = [\forest_-(g \cdot \forest),g^\forest h^{h^{-1} \cdot \altforest},\altforest_+(h^{-1} \cdot \altforest)] \text{.}\nonumber
 \end{align}
 In the last line we used that $(h^\altforest)^{-1} = (h^{-1})^{h \cdot \altforest}$ so that $((h^{-1})^{\altforest})^{-1} = h^{h^{-1} \cdot \altforest}$.

 We claim that the last expression of~\eqref{eq:simple_representative} is simple. Indeed, $(\forest_-,g)$ and $\forest_+$ are semisimple with the same number of feet and $\forest_+ \forest$ is semisimple so $(\forest_-,g) \forest = (\forest_- (g \cdot \forest),g^\forest)$ is semisimple by Lemma~\ref{lem:semisimple_elements} \eqref{item:monoid_semisimplicity_common_multiple}. Similar reasoning applies to $(\altforest_+(h^{-1} \cdot \altforest), {h^{-1}}^\altforest)$. Moreover, we can use Corollary~\ref{cor:length} to compute
 \[
 \len (\forest_-,g) + \len \forest \stackrel{\mathclap{s \text{ simple}}}{\strut=} \len \forest_+ + \len \forest \stackrel{\eqref{eq:common_simple_right_multiple}}{=} \len \altforest_- + \len \altforest \stackrel{\mathclap{t \text{ simple}}}{\strut=} \len (\altforest_+,(h^{-1})^F) + \len \altforest\text{.}
 \]
 By Lemma~\ref{lem:semisimple_elements}~\eqref{item:type_length} this shows that the last expression of~\eqref{eq:simple_representative} is simple.
\end{proof}

\begin{definition}[Thompson group of a filtered cloning system]\label{def:thompson}
 The group of simple elements in $\Thomphat{G}$ is denoted $\Thomp{G_*}$ and called the \emph{generalized Thompson group} of $G_*$. If we need to be more precise, as with $\Thomphat{G}$, we can include other data from the cloning system in the notation as in $\Thomp{G_*,\rho_*,(\clone^*_k)_k}$.
\end{definition}

Notationally, when we talk about a generalized Thompson group, the asterisk will always take the position of the index of the family. For instance, the generalized Thompson group for the family $(G^n)_{n \in \N}$ of direct powers in Section~\ref{sec:direct_prods} will be denoted $\Thomp{G^*}$; and the generalized Thompson group for the family of matrix groups $(B_n(R))_{n \in \N}$ in Section~\ref{sec:matrix_groups} will be denoted $\Thomp{B_*(R)}$.

Recall from the discussion after Corollary~\ref{cor:length} that there is a length morphism $\len \colon \Thomphat{G} \to \Z$ which takes an element $[\forest,g,\altforest]$ to $\len(\forest) - \len(\altforest)$. The group $\Thomp{G_*}$ lies in the kernel of that morphism, that is, simple elements have length $0$.

Given a simple element $[\forest,g,\altforest]$ with $\forest=(\tree_i)_{i\in\N}$ and $\altforest=(\alttree_i)_{i\in\N}$, since all the $\tree_i$ and $\alttree_i$ are trivial for $i>1$, we will often write our element as $[\tree_1,g,\alttree_1]$ instead. In other words, we view an element of $\Thomp{G_*}$ as being a tree with $n$ leaves, followed by an element of $G_n$, followed by another tree with $n$ leaves.

\begin{remark}
 Constructing $\Thomp{G_*}$ as the subgroup of simple elements of $\Thomphat{G}$ is somewhat artificial as can be seen in some of the proofs above. The more natural approach would be to have each element of $\Fmonoid$ ``know'' on which level it can be applied. This amounts to considering the category of forests $\calP$ that has objects the natural numbers and morphisms $\lambda_k^n \colon n \to n+1, 1 \le k \le n$ subject to the forest relations \eqref{eq:forest_relation}, cf.\ \cite[Section~7]{belk04}. Let $\calG$ be another category that also has objects the natural numbers and morphisms from $n$ to $n$ that form a group $G_n$. So while $\calP$ has only ``vertical'' arrows, $\calG$ has only ``horizontal'' arrows. One would then want to form the Zappa--S\'zep product $\calP \bowtie \calG$ which would be specified by commutative squares of the form $\gamma \lambda_{k}^n  = \lambda_{\rho(\gamma)k}^{n} \gamma^{\lambda_k}$ with $\gamma \in G_n$ and $\gamma^{\lambda_k} \in G_{n+1}$. Localizing everywhere one would obtain a groupoid of fractions $\calQ$ and $\Thomp{G_*}$ should be just $\Hom_\calQ(1,1)$.

 The reason that we have not chosen that description is simply that Zappa--S\'zep products for categories are not well-developed to our knowledge, while for monoids all the needed statements were already available thanks to Brin's work \cite{brin05, brin07}.

 Artifacts of this approach, which should be overcome by the general approach above, include the maps $\iota_{n,n+1}$ and the property of being properly graded. Not having to collect all the groups $G_n$ in a common group $G$ would also make it possible to construct, for example, the Thompson groups $T$ and $\Tbr$.
\end{remark}

\subsection{Morphisms}\label{sec:morphisms}

Let $(G,\rho^G,(\clone^G_k)_{k \in \N})$ and $(H,\rho^H,(\clone^H_k)_{k \in \N})$ be cloning systems. A homomorphism $\varphi \colon G \to H$ is a \emph{morphism of cloning systems} if
\begin{enumerate}
 \item $(\varphi(g)) \clone^H_k = \varphi((g) \clone^G_k)$ for all $k\in\N$ and $g\in G$, and\label{item:morphism_group}
 \item $\rho^H \circ \varphi = \rho^G$.\label{item:morphism_monoid}
\end{enumerate}

\begin{observation}\label{obs:thomphat_morphism}
 Let $\varphi \colon G \to H$ be a morphism of cloning systems. There is an induced homomorphism $\Thomphat{\varphi} \colon \Thomphat{G} \to \Thomphat{H}$. If $\varphi$ is injective or surjective then so is $\Thomphat{\varphi}$. In particular, if Observation~\ref{obs:sync} holds even for $i=k,k+1$, there is always a homomorphism $\Thomphat{G} \to \Thomphat{\symm_\omega}$.
\end{observation}

\begin{proof}
 We show that a morphism of cloning systems induces a homomorphism $\Fmonoid \bowtie G \to \Fmonoid \bowtie H$. The statement then follows from Lemma~\ref{lem:ore_fractions}. Naturally, $\Thomphat{\varphi}$ is defined by $\Thomphat{\varphi}(\forest g) = \forest \varphi(g)$. Well definedness amounts to $\Thomphat{\varphi}((g \cdot \forest)g^E) = (\varphi(g) \cdot \forest)(\varphi(g)^\forest)$ which follows from \eqref{item:morphism_group} and \eqref{item:morphism_monoid} above by writing $\forest$ as a product of $\lambda_k$s and inducting on the length.
 
 The injectivity and surjectivity statements are clear.
\end{proof}

Similarly let $(G_n)_{n \in \N}$ and $(H_n)_{n \in \N}$ be injective direct systems equipped with cloning systems. A morphism of directed systems of groups $\varphi_* \colon G_* \to H_*$ is a \emph{morphism of cloning systems} if
\begin{enumerate}
 \item $(\varphi_n(g)) \clone^{H,n}_k = \varphi_{n+1}((g) \clone^{G,n}_k)$ for all $1\le k\le n$ and $g\in G_n$, and
 \item $\rho^H_n \circ \varphi_n = \rho^G_n$ for all $n\in\N$.
\end{enumerate}

\begin{observation}\label{obs:thomp_morphism}
 Let $\varphi_* \colon G_* \to H_*$ be a morphism of cloning systems. There is an induced homomorphism $\Thomp{\varphi} \colon \Thomp{G_*} \to \Thomp{H_*}$. If $\varphi$ is injective or surjective then so is $\Thomp{\varphi}$. In particular, if \eqref{item:fcs_compatibility} holds even for $i=k,k+1$, there is always a homomorphism $\Thomp{G_*} \to \Thomp{\symm_*}$, the latter being Thompson's group $V$.
\end{observation}

\begin{proof}
 We have to show that if $\forest g \in \Fmonoid \bowtie G$ is semisimple with $n$ feet then so is $\Thomphat{\varphi}(\forest g) = \forest \varphi(g)$. But this follows since $\forest$ is semisimple with $n$ feet and $g \in G_n$, so $\varphi(g) \in H_n$.
\end{proof}

Functoriality is straightforward:

\begin{observation}\label{obs:functoriality}
 If $\varphi \colon G_* \to H_*$ and $\psi \colon H_* \to K_*$ are morphisms of cloning systems then $\Thomphat{\psi \varphi} = \Thomphat{\psi} \Thomphat{\varphi} \colon \Thomphat{G} \to \Thomphat{K}$. If $\varphi$ and $\psi$ are morphisms of filtered cloning systems then $\Thomp{\psi \varphi} = \Thomp{\psi} \Thomp{\varphi} \colon \Thomp{G_*} \to \Thomp{K_*}$.\qed
\end{observation}

\section{Basic properties}\label{sec:basic_properties}

Throughout this section let $\Thomp{G_*}$ be the generalized Thompson group of a cloning system on an injective directed system of groups $(G_n)_{n \in \N}$ and let $G = \varinjlim G_n$. We collect some properties of $\Thomp{G_*}$ that follow directly from the construction.

\subsection{A short exact sequence}

\begin{observation}
 Let $\tree \in \Fmonoid$ be semisimple with $n$ feet. The map $g \mapsto [\tree,g,\tree]$ is an injective homomorphism $G_n \to \Thomp{G_*}$.
\end{observation}

\begin{proof}
 The maps $G_n \to G \to \Fmonoid \bowtie G \to \Thomphat{G}$ are all injective. The element $[\tree,g,\tree]$ is simple, so the image lies in $\Thomp{G_*}$. The map is visibly a homomorphism.
\end{proof}

In fact, this can be explained more globally. For a semisimple forest $\tree$ with $n$ feet let $G_\tree$ denote the subgroup (isomorphic to $G_n$) of $\Thomp{G_*}$ that consists of elements $[\tree,g,\tree]$. The cloning map $\clone_k$ induces an embedding $G_\tree \into G_\alttree$ where $\alttree$ is obtained from $\tree$ by adding a split to the $k$th foot (so $\alttree=\tree\lambda_k$). Finite binary trees form a directed set and the condition \eqref{item:fcs_product_of_clonings} (product of clonings) ensures that that the groups $(G_\tree)_\tree$ form a directed system of groups.

\begin{lemma}
Consider a cloning system that satisfies condition~\eqref{item:fcs_compatibility} even for $i =k, k+1$ (this is the case in particular if $\rho = 0$). There is a directed subsystem $(K_T)_T$ of $(G_T)_T$ and a short exact sequence
\[
1 \to \varinjlim_\tree K_\tree \to \Thomp{G_*}\to W \to 1
\]
where the quotient morphism is the morphism $\Thomp{\rho_*}$ from Observation~\ref{obs:thomp_morphism} and $W$ is its image.
\end{lemma}

Note that $W$ contains Thompson's group $F$.

\begin{proof}
For each $\tree$, say with $n$ feet, let $K_\tree$ be the kernel of $\rho_n \colon G_\tree \to \symm_n$. The assumption on the cloning system implies that if $\rho(g) = 1$ then $\rho((g)\clone_k) = 1$, showing that $(K_T)_T$ is indeed a subsystem of $(G_T)_T$. It remains to see that the direct limit is isomorphic to the kernel of $\Thomp{\rho_*}$. This is clear once one realizes that it consists of all elements that can be written in the form $[\tree,g,\tree]$, for some $\tree$ and $g\in K_\tree$.
\end{proof}

In what follows we will concentrate on the case where $\rho=0$ is the trivial morphism $\rho(g) = 1$, so $K_\tree=G_\tree$ for all $\tree$. Examples are $F$ and $\Fbr$ but not $V$ and $\Vbr$.

\begin{observation}\label{obs:thomp_kern}
Suppose $\rho = 0$. Then $\Thomp{G_*}=\Thkern{G_*}\rtimes F$.
\end{observation}

\begin{proof}
 Since each $\rho_n=0$, we have $W=F$, which is $\Thomp{\{1\}}$. Then the splitting map $F\to \Thomp{G_*}$ is $\Thomp{\iota_*}$ where $\iota_* \colon \{1\} \to G_*$ is the trivial homomorphism.
\end{proof}

\begin{remark}
\label{rmk:wreath}
Bartholdi, Cornulier, and Kochloukova \cite{bartholdi15} studied finiteness properties of wreath products. Observation~\ref{obs:thomp_kern} shows how this relates to our groups. A wreath product is built by taking a direct product of copies of a group $H$, indexed by a set $X$, and combining this with another group $G$ acting on $X$. The generalized Thompson groups in Observation~\ref{obs:thomp_kern} can be viewed as the result of taking a direct limit (instead of product) of groups from a family $(G_\tree)_\tree$, indexed by a set of trees $\tree$ on which there is a \emph{partial} (instead of full) action of $F$, and combining these data into a group $\Thomp{G_*}$.
\end{remark}

The question of whether $F$ is amenable or not is probably the most famous question about Thompson's groups. The following observation does not purport to be deep, but it seems worth recording nonetheless.

\begin{observation}[Amenability]\label{obs:amenability}
 Suppose $\rho = 0$. Then $\Thomp{G_*}$ is amenable if and only if $F$ and every $G_n$ is amenable.
\end{observation}

\begin{proof}
We have seen that $\Thkern{G_*}$ is a direct limit of copies of $G_n$. Since amenability is preserved under taking subgroups and direct limits, this tells us that $\Thkern{G_*}$ is amenable if and only if every $G_n$ is. Then since $\Thomp{G_*}=\Thkern{G_*}\rtimes F$, the conclusion follows since amenability is also closed under group extensions.
\end{proof}

\begin{observation}[Free group-free]\label{obs:free_gps}
 Suppose $\rho = 0$. If none of the $G_n$ contains a non-abelian free group then neither does $\Thomp{G_*}$.
\end{observation}

\begin{proof}
 Suppose $H\le \Thomp{G_*}$ is free. If $H\cap \Thkern{G_*}=\{1\}$ then $H$ embeds into $F$, and so $H$ must be cyclic, since $F$ does not contain a non-abelian free group. Now suppose there is some $1\neq x\in H\cap \Thkern{G_*}$. For any $y\in H$, the conjugate $x^y$ is in $H\cap \Thkern{G_*}$. Since $\Thkern{G_*}$ is a direct limit of copies of the $G_n$, it does not contain a non-abelian free group by assumption, and so $\langle x,x^y\rangle$ is abelian. But $y\in H$ was arbitrary, so $H$ must already be abelian.
\end{proof}

The next result does not require $\rho=0$. It fits into the context of this section but to prove it we need some of the tools of Section~\ref{sec:stein_space}.

\begin{lemma}[Torsion-free]\label{lem:torfree}
 Assume that the cloning system is properly graded. If all the $G_n$ are torsion-free then so is $\Thomp{G_*}$.
\end{lemma}

\subsection{Truncation}\label{sec:truncation}
For $g \in G_n$ and $k \le n$ we have the equation $g \lambda_k = (g \cdot \lambda_k) g^{\lambda_k}$ in $\Fmonoid \bowtie G$ where $g^{\lambda_k} \in G_{n+1}$. In $\Thomphat{G}$ this implies
\begin{equation}
\label{eq:lift_group_element}
g = (g\cdot \lambda_k) g^{\lambda_k} \lambda_k^{-1}\text{.}
\end{equation}
This elementary observation has an interesting consequence. Let $N \in \N$ be arbitrary and define a directed system of groups $(G'_n)_{n \in \N}$ by $G'_n \defeq \{1\}$ for $n \le N$ and $G'_n \defeq G_n$ for $n > N$. Define a cloning system on $G'_*$ by letting $(\clone')^n_k \colon G'_n \to G'_{n+1}$ be the trivial homomorphism when $n\le N$, and $(\clone')^n_k = \clone^n_k$ and $\rho'_n = \rho_n$ when $n > N$. We call $G'_*$ the \emph{truncation} of $G_*$ at $N$ and $((\rho'_n)_n,((\clone')^n_k)_{k \le n})$ the truncation of $((\rho_n)_n, ({\clone}^n_k)_{k \le n})$ at $N$.

\begin{proposition}[Truncation isomorphism]\label{prop:truncation_isomorphism}
 Let $G'_*$ be the truncation of $G_*$ at $N$. The morphism $\Thomp{G'_*} \to \Thomp{G_*}$ induced by the obvious homomorphism $G'_* \to G_*$ is an isomorphism.
\end{proposition}

\begin{proof}
 The morphism $G'_* \to G_*$ is injective hence so is $\Thomp{G'_*} \to \Thomp{G_*}$. To show that it is surjective let $[\tree,g,\alttree] \in \Thomp{G_*}$ be such that $\tree$ and $\alttree$ have $n$ leaves. If $n > N$ there is nothing to show. Otherwise use \eqref{eq:lift_group_element} to write
 \[
 [\tree,g,\alttree] = [\tree (g \cdot \lambda_k) , g^{\lambda_k}, \alttree \lambda_k]
 \]
 for some $k \le n$. The trees in the right hand side expression have $n+1$ leaves. Proceeding inductively, we obtain an element whose trees have $N+1$ leaves and therefore the element is in $\Thomp{G'_*}$.
\end{proof}

This proposition is in line with treating $\Thomp{G_*}$ as a sort of limit of $G_*$ since it does not depend on an initial segment of data.

\section{Spaces for generalized Thompson groups}\label{sec:spaces}

The goal of this section is to produce for each generalized Thompson group $\Thomp{G_*}$ a space on which it acts. The space will be contractible and have stabilizers isomorphic to the groups $G_n$, assuming the cloning system on $G_*$ is properly graded. The ideas used in the construction were used before in \cite{stein92, brown92, farley03, brown06, fluch13, bux14}. Throughout let $G_*$ be an injective directed system of groups equipped with a cloning system and let $G = \varinjlim G_*$.

As a starting point we note that Corollary~\ref{cor:fraction_semilattice}, Observation~\ref{obs:BZS_cancellative_lcrms} and Corollary~\ref{cor:length} imply that $\Thomphat{G}/G$ is a join-semilattice with conditional meets, under the relation $xG \le yG$ if $x^{-1}y \in \Fmonoid \bowtie G$. Later on it will be convenient to have a symbol for the quotient relation so we let $x \sim_G y$ if $x^{-1}y \in G$.

\subsection{Semisimple group elements}\label{sec:semisimple}
We generalize some of the notions that were introduced in Sections~\ref{sec:forest_monoid} and \ref{sec:BZS_to_thomp}. We say that an arbitrary (not necessarily semisimple) element $\forest$ of $\Fmonoid$ has \emph{$n$ feet} if it has rank $m$ and length $n-m$. Visually this means that the last leaf that is not a root is numbered $n$. An element $(\forest,g)$ of $\Fmonoid \bowtie G$ has \emph{$n$ feet} if $\forest$ has at most $n$ feet and $g \in G_n$. Finally, we call an element $[\forest,g,\altforest]$ of $\Thomphat{G}$ \emph{semisimple} if $(\forest,g)$ is semisimple with $n$ feet and $\altforest$ has at most $n$ feet (note $\altforest$ need not be semisimple). This is consistent with the previous definition of ``semisimple'': If an element of the group $\Thomphat{G}$ is semisimple in this sense, and is an element of the monoid $\Fmonoid\bowtie G$, then it must be semisimple in the monoid. We let $\Preposet_1$ denote the set of all semisimple elements of $\Thomphat{G}$.

\begin{lemma}\label{lem:simple_times_semisimple}
 If $[\forest_1,g_1,\altforest_1]$ is simple and $[\forest_2,g_2,\altforest_2]$ is semisimple then their product $[\forest_1,g,\altforest_1][\forest_2,g,\altforest_2]$ is semisimple. As a consequence, $\Thomp{G_*}$ acts on $\Preposet_1$.
\end{lemma}

\begin{proof}
 This is shown analogously to Proposition~\ref{prop:product_of_simple}.
\end{proof}

If $[\forest,g,\altforest]$ is semisimple we say that it has $\len([\forest,g,\altforest]) + 1 = \len(\forest) - \len(\altforest) + 1$ \emph{feet}, which is well defined by Corollary~\ref{cor:length}. This can be visualized as the number of roots of $\altforest$ that can be ``reached'' from the first root of $\forest$. We let $\Preposet_{1,n}$ denote the set of all semisimple elements with at most $n$ feet. We define $\Poset_{1,n}$ to be the quotient $\Preposet_{1,n}/{\sim}_G$ and call the passage from $\Preposet_{1,n}$ to $\Poset_{1,n}$ \emph{dangling}. Note that $\Poset_{1,n}$ is a subposet of $\Thomphat{G}/G$. We also denote $\Preposet_1/{\sim_G}$ by $\Poset_1$.

For context, the term ``dangling'' comes from the case when $G_*$ is the system of braid groups $B_*$, and the elements of $\Poset_{1,n}$ can be pictured as ``dangling braided strand diagrams'' \cite{bux14}, originating on one strand and ending on $n$ strands.

The next lemma is the reason for having introduced the notion of a cloning system being properly graded.

\begin{lemma}\label{lem:dangling}
Assume that the cloning system is properly graded. If $x,y \in \Preposet_{1,n}$ are semisimple then $x \sim_G y$ if and only if $x^{-1}y \in G_n$.
\end{lemma}

\begin{proof}
 What needs to be shown is that if $x^{-1} y \in G$ then $x^{-1} y\in G_n$. Write $x = [\forest_1,g^{-1},\altforest_1]$ and $y = [\forest_2,h^{-1},\altforest_2]$. Let $\forest = \forest_1 \forest_1' = \forest_2 \forest_2'$ be a common right multiple so that $x^{-1}y = [\altforest_1 (g \cdot \forest_1'),g^{\forest_1'} (h^{\forest_2'})^{-1},\altforest_2 (h \cdot \forest_2')] \eqdef [A,b,C]$. For this to equal some $d \in G$ it is necessary that $Ab = dC$ in $\Fmonoid \bowtie G$, that is, $A = d \cdot C$ and $b = d^C$.
 
 Say that $\forest$ has length $m$. Then we compute that $\len(A) = \len(C) \ge m-n+1$. Since the cloning system is properly graded, the fact that $b=d^C$ implies that $d$ has to lie in $G_{m+1-\len(C)} \subseteq G_n$.
\end{proof}

\subsection{Poset structure}\label{sec:poset}

Consider the geometric realization $\realize{\Poset_1}$. This is the simplicial complex with a $k$-simplex for each chain $x_0\le\cdots\le x_k$ of elements of $\Poset_1$, and face relation given by subchains.

\begin{lemma}\label{lem:poset_semilattice}
 The poset $\Poset_1$ is a join-semilattice with conditional meets, in particular $\realize{\Poset_1}$ is contractible.
\end{lemma}

\begin{proof}
 We already know that $\Thomphat{G}/G$ is a join-semilattice with conditional meets so it suffices to show that $\Poset_1$ is closed under taking suprema and infima. In other words, it suffices to show that least common right multiples of semisimple elements are semisimple and that left factors of semisimple elements are semisimple. The first is similar to the proof of Proposition~\ref{prop:product_of_simple} and the second is easy.
\end{proof}

In $\realize{\Poset_1}$ every vertex is contained in a simplex of arbitrarily large dimension, which makes it too big for practical purposes. It has proven helpful to consider a subspace called the \emph{Stein--Farley complex}, which we introduce next.

\subsection{The Stein--Farley complex}\label{sec:stein_space}

The preorder on $\Preposet_1$ was defined by declaring that $x \le y$ if $y = x(\forest,g)$ for some $(\forest,g) \in \Fmonoid \bowtie G$. The basic idea in constructing the Stein--Farley complex is to regard this relation as a transitive hull of a finer relation $\preceq$ and to use this finer relation in constructing the space. It is defined by declaring $x \preceq y$ if $y = x(\forest,g)$ for some $(\forest,g) \in \Fmonoid \bowtie G$ with the additional assumption that $\forest$ is elementary. An \emph{elementary} forest is one in which every tree has at most two leaves. That is, a forest is elementary if it can be written as $\lambda_{k_1} \cdots \lambda_{k_r}$ with $k_{i+1} > k_i + 1$ for $i < r$. Note that if $x \in \Preposet_{1,n}$, in order for $x(E,g)$ to be in $\Preposet_1$ as well, it is necessary that $E$ has rank at most $n$ and that $g \in G_{n + \len(E)}$. Note also that if $\forest$ is elementary then so is $g\cdot \forest$ for any $g\in G$ because the action of $G$ (via $\rho \colon G \to \symm_\omega$) just permutes the trees of $\forest$.

As a consequence $\preceq$ is invariant under dangling and we also write $\preceq$ for the relation induced on $\Poset_1$. Note that $\preceq$ is not transitive, but it is true that if $x\preceq z$ and $x\le y\le z$ then $x\preceq y\preceq z$. Given a simplex $x_0\le\cdots\le x_k$ in $\realize{\Poset_1}$, call the simplex \emph{elementary} if $x_0\preceq x_k$. The property of being elementary is preserved under passing to subchains, so the elementary simplices form a subcomplex.

\begin{definition}[Stein--Farley complex]\label{def:stein_space}
 The subcomplex of elementary simplices of $\realize{\Poset_1}$ is denoted by $\Stein{G_*}$ and called the \emph{Stein--Farley complex} of $\Thomp{G_*}$.
\end{definition}

The Stein--Farley complex has the structure of a cubical complex, which we now describe. The key point is:

\begin{observation}\label{obs:fmonoid_elementary_boolean}
 If $\forest$ is elementary then the set of right factors of $\forest$ forms a boolean lattice under $\preceq$. \qed
\end{observation}

For $x \preceq y$ in $\Poset_1$ we consider the closed interval $[x,y]\defeq\{z\in\Poset_1\mid x\le z\le y\}$ as well as the open and half open intervals $(x,y)$, $[x,y)$ and $(x,y]$ that are defined analogously. As a consequence of Observation~\ref{obs:fmonoid_elementary_boolean} we obtain that the interval $[x,y]\defeq\{z\in\Poset_1\mid x\le z\le y\}$ is a boolean lattice and so $\realize{[x,y]}$ has the structure of a cube. The intersection of two such cubes $\realize{[x,y]}$ and $\realize{[z,w]}$ is empty if $y$ and $w$ do not have a common lower bound and is $\realize{[\sup(x,z),\inf(y,w)]}$ (which may be empty if the supremum is larger than the infimum) otherwise. In particular the intersection of cubes is either empty or is again a cube. Hence $\Stein{G_*}$ is a cubical complex in the sense of Definition~7.32 of~\cite{bridson99}.

\begin{observation}\label{obs:stein_space_locally_finite}
 For any vertex $x$ in $\Stein{G_*}$, there are only finitely many vertices $y$ in $\Stein{G_*}$ with $x\preceq y$.
\end{observation}

\begin{proof}
 If $\tilde{x} \in \Preposet_1$ is a vertex representative (modulo dangling) for $x$, it is clear using dangling that every vertex $y$ with $x \preceq y$ has a representative $\tilde{y}$ with $\tilde{y} = \tilde{x}(\forest,1)$ for some some elementary forest $\forest$. In order for $\tilde{y}$ to be semisimple, $\forest$ can have rank at most $\len(\tilde{x}) - 1$, and there are only finitely many elementary forests of a given rank, so the result follows.
\end{proof}

The next step is to show that $\Stein{G_*}$ is itself contractible. The argument is similar to that given in Section~4 of \cite{brown92}. We follow the exposition in \cite{bux14}.

\begin{lemma}\label{lem:cube_lemma}
 For $x<y$ with $x\not\prec y$, $|(x,y)|$ is contractible.
\end{lemma}

\begin{proof}
 For any $z\in(x,y]$ let $z_0$ be the unique largest element of $[x,z]$ such that $x\preceq z_0$.  By hypothesis $z_0\in[x,y)$, and by the definition of $\preceq$ it is clear that $z_0\in(x,y]$, so in fact $z_0\in(x,y)$.  Also, $z_0\le y_0$ for any $z\in(x,y)$.  The inequalities $z\ge z_0\le y_0$ then imply that $|(x,y)|$ is contractible, by Section~1.5 of~\cite{quillen78}.
\end{proof}

\begin{proposition}\label{prop:stein_space_contractible}
 $\Stein{G_*}$ is contractible.
\end{proposition}

\begin{proof}
 We know that $|\Poset_1|$ is contractible by Lemma~\ref{lem:poset_semilattice}. We can build up from $\Stein{G_*}$ to $|\Poset_1|$ by attaching new subcomplexes, and we claim that this never changes the homotopy type, so $\Stein{G_*}$ is contractible. Given a closed interval $[x,y]$, define $r([x,y])\defeq \len(y)-\len(x)$. As a remark, if $x\preceq y$ then $r([x,y])$ is the dimension of the cube given by $[x,y]$.  We attach the contractible subcomplexes $|[x,y]|$ for $x\not\preceq y$ to $\Stein{G_*}$ in increasing order of $r$-value. When we attach $|[x,y]|$ then, we attach it along $|[x,y)|\cup|(x,y]|$.  But this is the suspension of $|(x,y)|$, and so is contractible by the previous lemma.  We conclude that attaching $|[x,y]|$ does not change the homotopy type, and since $|\Poset_1|$ is contractible, so is $\Stein{G_*}$.
\end{proof}

\begin{lemma}[Stabilizers]\label{lem:stabs}
Assume that the cloning system is properly graded.  The stabilizer in $\Thomp{G_*}$ of a vertex in $\Stein{G_*}$ with $n$ feet is isomorphic to $G_n$. The stabilizer in $\Thomp{G_*}$ of an arbitrary cell is isomorphic to a finite index subgroup of some $G_n$.
\end{lemma}

\begin{proof}
 First consider the stabilizer of a vertex $x$ with $n$ feet. We claim that $\Stab_{\Thomp{G_*}}(x)\cong G_n$. Choose $\tilde{x}\in\Preposet_1$ representing $x$ and let $g\in \Stab_{\Thomp{G_*}}(x)$. By the definition of dangling, and by Lemma~\ref{lem:dangling}, there is a (unique) $h \in G_n$ such that $g\tilde{x} = \tilde{x}h$. Then the map $g\mapsto h = \tilde{x}^{-1} g \tilde{x}$ is a group isomorphism.
 
 Now let $\sigma = \realize{[x,y]}$, $x \preceq y$ be a an arbitrary cube. Since the action of $\Thomp{G_*}$ preserves the number of feet, the stabilizer $G_\sigma$ of $\sigma$ fixes $x$ and $y$. Hence $G_\sigma$ is contained in $G_x$ and contains the kernel of the map $G_x \to \Symm(\{w\mid x\preceq w \preceq y\})$, the image of which is finite by Observation~\ref{obs:stein_space_locally_finite}.
\end{proof}

We close this section by providing the proof of Lemma~\ref{lem:torfree}, left out in the last section, which says that $\Thomp{G_*}$ is torsion-free as soon as all the $G_n$ are.

\begin{proof}[Proof of Lemma~\ref{lem:torfree}]
 The vertices in $\Stein{G_*}$ coincide with the vertices of $\realize{\Poset_1}$, and, as we just proved, any vertex has some $G_n$ as a stabilizer. Hence it suffices to prove that if $g\in\Thomp{G_*}$ has finite order then it fixes an element of the directed poset $\Poset_1$. By Lemma~\ref{lem:poset_semilattice}, $\Poset_1$ is a join-semilattice, so any finite collection of elements has a unique least upper bound. But then if $g$ has finite order, for any $x\in\Poset_1$ the unique least upper bound of the finite set $\gen{g}.x$ is necessarily fixed by $g$.
\end{proof}

\section{Finiteness properties}\label{sec:finiteness_props}

One of our main motivations for defining the functor $\Thomp{-}$ is to study how it behaves with respect to finiteness properties. Recall that a group $G$ if said to be \emph{of type~$\F_n$} if there is a $K(G,1)$ whose $n$-skeleton is compact. Most of the known Thompson's groups are of type~$\F_\infty$, that is, of type~$\F_n$ for all $n$. To efficiently speak about groups that are not of type~$\F_\infty$ recall that the \emph{finiteness length} of $G$, denoted $\phi(G)$, is the supremum over all $n \in \N$ such that $G$ is of type~$\F_n$.

We will see below that proofs of the finiteness properties of $\Thomp{G_*}$ depend on the finiteness properties of the individual groups $G_n$ as well as on the asymptotic connectivity of certain descending links, which is infinite in many cases. Since finite initial intervals of $G_*$ can always be ignored by Proposition~\ref{prop:truncation_isomorphism} we ask:

\begin{question}\label{ques:finiteness_conjecture}
 For which directed systems of groups $G_*$ equipped with properly graded cloning systems do we have
 \[
 \phi(\Thomp{G_*}) = \liminf \phi(G_*)\text{?}
 \]
\end{question}

Note that for any directed system of groups $G_*$ one can take all $\rho_k$ to be trivial and all $\clone_k^n$ to be $\iota_{n,n+1}$. In this case $\Thomp{G_*} = (\lim_n G_n) \times F$, which would seem to give a negative answer to Question~\ref{ques:finiteness_conjecture}. However, in order to be properly graded in this example we would need $\image \iota_{n,n+1} \subseteq \image \iota_{n-1,n+1}$, and this implies that the $\iota_{n,n+1}$ are all isomorphisms. Thus, in fact this does provide a positive answer to the question.

\subsection{Morse theory}\label{sec:morse}

One of the main tools to study connectivity properties of spaces, and thus to study finiteness properties of groups, is combinatorial Morse theory. We collect here the main ingredients that will be needed later on.

Let $X$ be a Euclidean cell complex. A map $h \colon X^{(0)} \to \Nz$ is called a \emph{Morse function} if the maximum of $h$ over the vertices of a cell of $X$ is attained in a unique vertex. We typically think of $h$ as assigning a \emph{height} to each vertex. If $h$ is a Morse function and $r \in \R$, the sublevel set $X_r = X^{\le r}$ consists of all cells of $X$ whose vertices have height at most $r$. For a vertex $x \in X^{(0)}$ of height $r$, the \emph{descending link} $\dlk(x)$ of $x$ is the subcomplex of $\lk (x)$ spanned by all vertices of strictly lower height. The main observation that makes Morse theory work is that keeping track of the connectivity of descending links allows one to deduce global (relative) connectivity statements:

\begin{lemma}[Morse Lemma]\label{lem:morse}
 Let $X$ be a Euclidean cell complex and let $h \colon X^{(0)} \to \Nz$ be a Morse function on $X$. Let $s,t \in \R \cup \{\infty\}$ with $s < t$. If $\dlk(x)$ is $(k-1)$-connected for every vertex in $X_t \setminus X_s$ then the pair $(X_t,X_s)$ is $k$-connected.
\end{lemma}

The connection between connectivity of spaces and finiteness properties of groups is most directly made using Brown's criterion. A Morse function on $X$ gives rise to a filtration $(X_r)_{r \in \Nz}$ by subcomplexes. We say that the filtration is \emph{essentially $k$-connected} if for every $i \in \Nz$ there exists a $j \ge i$ such that $\pi_\ell(X_i \to X_j)$ is trivial for all $\ell \le k$.

Now assume that a group $G$ acts on $X$. If $h$ is $G$-invariant then so is the filtration $(X_r)_r$. We say that the filtration is \emph{cocompact} if the quotient $G \backslash X_r$ is compact for all $r$. This is the setup for Brown's criterion, see \cite[Theorems~2.2,~3.2]{brown87}.

\begin{theorem}[Brown's criterion]
 Let $n \in \N$ and assume a group $G$ acts on an $(n-1)$-connected CW complex $X$. Assume that the stabilizer of every $p$-cell of $X$ is of type~$\F_{n-p}$. Let $(X_r)_{r \in \Nz}$ be a $G$-cocompact filtration of $X$. Then $G$ is of type~$\F_n$ if and only if $(X_r)_r$ is essentially $(n-1)$-connected.
\end{theorem}

Putting both statements together we obtain the version that we will mostly use.

\begin{corollary}\label{cor:brown_crit_use}
 Let $G$ act on a contractible Euclidean cell complex $X$ and let $h \colon X^{(0)} \to \Nz$ be a $G$-invariant Morse function. Assume that the stabilizer of every $p$-cell of $X$ is of type~$\F_{n-p}$ and that the sublevel sets $X_r$ are cocompact. If there is an $s \in \R$ such that $\dlk(x)$ is $(n-1)$-connected for all vertices $x \in X^{(0)} \setminus X_s$ then $G$ is of type~$\F_n$.
\end{corollary}

If $G_*$ is a system of groups equipped with a properly graded cloning system then $\Thomp{G_*}$ acts on the Stein--Farley complex $\Stein{G_*}$, which is contractible (Proposition~\ref{prop:stein_space_contractible}) with stabilizers from $G_*$ (Lemma~\ref{lem:stabs}). Our next goal is to define an invariant, cocompact Morse function and to describe the descending links.

\subsection{The Morse function}\label{sec:morse_function}
Recall that the vertices of $\Stein{G_*}$ are classes $[\forest,g,\altforest]$ of semisimple elements modulo dangling. The height function we will be using assigns to such a vertex its number of feet (see~Section~\ref{sec:semisimple}). That is, $\Stein{G_*}_n = \realize{\Poset_{1,n}}\cap \Stein{G_*}$. This height function is $\Thomp{G_*}$-invariant because it is induced by the morphism $\len \colon \Thomphat{G} \to \Z$ and every element of $\Thomp{G_*}$ has length $0$.

\begin{lemma}[Cocompactness]\label{lem:cocompact}
 The action of $\Thomp{G_*}$ is transitive on vertices of $\Stein{G_*}$ with a fixed number of feet. Consequently the action of $\Thomp{G_*}$ on $\Stein{G_*}_n$ is cocompact for every $n$.
\end{lemma}

\begin{proof}
Let $\tilde{x} = [\forest_-,g,\forest_+]$ and $\tilde{y} = [\altforest_-,h,\altforest_+]$ be semisimple with $n$ feet. We know that $\tilde{x}\tilde{y}^{-1}$ takes $\tilde{y}$ to $\tilde{x}$, so it suffices to show that $\tilde{x}\tilde{y}^{-1}$ is simple. Note that $\forest_+$ and $\altforest_+$ have rank at most $n$. By Observation~\ref{obs:forest_semisimple_elements}~\eqref{item:forest_semisimple_right_common_multiple} they admit a common right multiple $\forest_+\forest = \altforest_+\altforest$ of rank at most $n$. Let the length of this multiple be $m$, so it has at most $m+n$ feet. Then
 \[
 \tilde{x}\tilde{y}^{-1} = [\forest_-(g \cdot \forest),g^\forest(h^{\altforest})^{-1},\altforest_-(h \cdot \altforest)]
 \]
 and both $\forest_-(g \cdot \forest)$ and $\altforest_-(h \cdot \altforest)$ are semisimple by Observation~\ref{obs:forest_semisimple_elements}~\eqref{item:forest_right_multiple_rank}. They have $m+n$ feet and both $g^\forest$ and $h^\altforest$ lie in $G_{n+m}$. Thus $\tilde{x}\tilde{y}^{-1}$ is simple.

The second statement now follows from Observation~\ref{obs:stein_space_locally_finite}.
\end{proof}

\subsection{Descending links}\label{sec:dlks}

Let $x$ be a vertex in $\Stein{G_*}$, with $n$ feet. We want to describe the descending link of $x$. A vertex $y$ is in the link of $x$ if either $x \preceq y$ or $y \preceq x$. In the first case $y$ is ascending so the descending link is spanned by vertices $y$ with $y \preceq x$. These are by definition of the form $x(\forest,g)^{-1}$ for $\forest$ an elementary forest and $g \in G_n$. In particular, for a fixed $n$, the descending links of any vertices of height $n$ look the same, and are all isomorphic to the simplicial complex of products $g \forest^{-1}$ where $g \in G_n$ and $\forest$ is an elementary forest with at most $n$ feet, modulo the relation $\sim_G$.

\medskip

It is helpful to describe this complex somewhat more explicitly. In doing so we slightly shift notation by making use of the fact that elementary forests can be parametrized by subgraphs of linear graphs.

Let $L_n$ be the graph with $n$ vertices, labeled $1$ through $n$, and a single edge connecting $i$ to $i+1$, for each $1\le i\le n-1$. This is the \emph{linear graph} with $n$ vertices. Denote the edge from $i$ to $i+1$ by $e_i$. We will exclusively consider \emph{spanning} subgraphs of $L_n$, that is, subgraphs whose vertex set is $\{1,\ldots,n\}$. We call the spanning subgraph without edges \emph{trivial}. A \emph{matching} on a graph is a spanning subgraph in which no two edges share a vertex. For an elementary forest $\forest$ with at most $n$ feet, define $\Gamma(\forest)$ to be the spanning subgraph of $L_n$ that has an edge from $i$ to $i+1$ if and only if the $i$th and $(i+1)$st leaves of $\forest$ are leaves of a common caret. Note that this is a matching. Conversely, given a matching $\Gamma$ of $L_n$, there is an elementary forest $E(\Gamma) = \lambda_{i_k} \cdots \lambda_{i_1}$ where $\Gamma$ has edges $e_{i_1}, \ldots,e_{i_k}$. Both operations are inverse to each other so we conclude:

\begin{observation}\label{obs:matchings_to_forests}
 There is a one-to-one correspondence between matchings of $L_n$ and elementary forests with at most $n$ feet. \qed
\end{observation}

In particular, if $\Gamma$ is a matching with $m$ edges and $n$ vertices we obtain a cloning map $\clone_\Gamma \colon G_{n-m} \to G_n$ which is just the cloning map of $E(\Gamma)$ as defined before Observation~\ref{obs:BZS_lclm}. We also get an action of $G_{n-m}$ on $\Gamma$ which is given by the action of $\rho(G_{n-m})$ on connected components. For future reference we also note:

\begin{observation}\label{obs:graphs_to_hedges}
 There is a one-to-one correspondence between spanning subgraphs of $L_n$ and hedges with at most $n$ feet. \qed
\end{observation}

Now define a simplicial complex $\dlkmodel{G_*}{n}$ as follows. A simplex in $\dlkmodel{G_*}{n}$ is represented by a pair $(g,\Gamma)$, where $g\in G_n$ and $\Gamma$ is a non-trivial matching of $L_n$. Two such pairs $(g_1,\Gamma_1)$, $(g_2,\Gamma_2)$ are \emph{equivalent (under dangling)} if the following conditions hold:
\begin{enumerate}
 \item $\Gamma_1$ and $\Gamma_2$ both have $m$ edges for some $1\le m\le n/2$,
 \item $g_2^{-1} g_1$ lies in the image of $\clone_{\Gamma_1}$, and
 \item $\Gamma_2=(g_2^{-1} g_1)\clone_{\Gamma_1}^{-1} \cdot \Gamma_1$.
\end{enumerate}

We make $\dlkmodel{G_*}{n}$ into a simplicial complex with face relation given by passing to subgraphs of the second term in the pair. Denote the equivalence class of $(g,\Gamma)$ under dangling by $[g,\Gamma]$. In summary,
\[
\dlkmodel{G_*}{n} \text{ has simplex set } \{[g,\Gamma]\mid \Gamma \text{ is a matching of } L_n \text{ and } g\in G_n\} \text{.}
\]

\begin{observation}\label{obs:descending_link_model}
 If $x$ has $n$ feet, the correspondence $(g,\Gamma) \mapsto x g E(\Gamma)^{-1}$ induces an isomorphism $\dlkmodel{G_*}{n} \to \dlk (x)$. \qed
\end{observation}

In particular, the $\dlkmodel{G_*}{n}$ are indeed simplicial complexes as claimed, since $\Stein{G_*}$ is a cubical complex.

We now have all the pieces together to apply Brown's criterion to our setting.

\begin{proposition}\label{prop:generic_finiteness}
 Let $G_*$ be equipped with a properly graded cloning system. If $G_k$ is eventually of type~$\F_n$ and $\dlkmodel{G_*}{k}$ is eventually $(n-1)$-connected then $\Thomp{G_*}$ is of type~$\F_n$.
\end{proposition}

\begin{proof}
 Suppose first that all $G_k$ are of type~$\F_n$. Let $X = \Stein{G_*}$, which is contractible by Proposition~\ref{prop:stein_space_contractible}. Our Morse function ``number of feet'' has cocompact sublevel sets by Lemma~\ref{lem:cocompact}. The stabilizer of any cell is a finite index subgroup of some $G_k$ by Lemma~\ref{lem:stabs}. Since finiteness properties are inherited by finite index subgroups, our assumption implies that all stabilizers are of type~$\F_n$. By the second assumption there is an $s$ such that $\dlkmodel{G_*}{k}$ is $(n-1)$-connected for $k > s$, which by Observation~\ref{obs:descending_link_model} means that descending links are $(n-1)$-connected from $s$ on. Applying Corollary~\ref{cor:brown_crit_use} we conclude that $\Thomp{G_*}$ is of type~$\F_n$.
 
 If the $G_k$ are of type~$\F_n$ only from $t$ on, we use Proposition~\ref{prop:truncation_isomorphism} to replace $\Thomp{G_*}$ by the isomorphic group $\Thomp{G_*'}$ where $G_k' = G_k$ for $k \ge t$ and $G_k = \{1\}$ for $k < t$. In particular, all of the $G_k'$ are of type~$\F_n$.
 
 Of course $\Stein{G_*'}$ is not isomorphic to $\Stein{G_*}$ and neither are the $\dlkmodel{G_*'}{m}$ isomorphic to the $\dlkmodel{G_*}{m}$. However, the $k$-skeleton of $\dlkmodel{G_*'}{m}$ is isomorphic to the $k$-skeleton of $\dlkmodel{G_*}{m}$ once $m > k + t$. Since $(n-1)$-connectivity only depends on the $n$-skeleton, if the $\dlkmodel{G_*}{m}$ are eventually $(n-1)$-connected then so are the $\dlkmodel{G_*'}{m}$.
\end{proof}

For a negative counterpart to this statement, that is, to show that $\Thomp{G_*}$ is not of type~$\F_n$, we would need stabilizers with good finiteness properties and a filtration that is not essentially $(n-1)$-connected -- at least as long as we are trying to apply Brown's criterion. Hence if we have groups $G_n$ whose finiteness lengths do not have a limit inferior of $\infty$, we would need an action on a different space to show that $\Thomp{G_*}$ answers Question~\ref{ques:finiteness_conjecture} affirmatively.

Returning to the positive statement, we remark that inspecting the homotopy type of $\dlkmodel{G_*}{n}$ does not seem possible uniformly. Instead, in what follows we will focus on examples and in particular find some instances of $\dlkmodel{G_*}{n}$ being highly connected. In the case where the $G_n$ are braid groups, these complexes were modeled by arc complexes in \cite{bux14}. In Section~\ref{sec:matrix_groups} below, where the $G_n$ are matrix groups, we will directly work with the combinatorial description. General tools that have turned out to be helpful will be collected in Sections~\ref{sec:high_connectivity} and~\ref{sec:relative_brown}.

We can make one positive statement about finiteness properties without knowing much at all about $G_*$. Before stating this as a lemma, we need to define the \emph{matching complex} of $L_n$. This is a simplicial complex, denoted $\match(L_n)$, whose simplices are matchings on $L_n$ and with face relation given by passing to subgraphs. It is well-known and not hard to see that $\match(L_n)$ is $(\floor{\frac{n-2}{3}}-1)$-connected. A precise description of the homotopy type is given in \cite[Proposition~11.16]{kozlov08} where $\match(L_n)$ arises as the independence complex $\independence(L_{n-1})$.

\begin{lemma}[Finite generation]\label{lem:fin_gen_case}
 Let $G_*$ be a family of groups equipped with a properly graded cloning system, with cloning maps $\clone_k^n$. Suppose that for $n$ sufficiently large, all $G_n$ are finitely generated and also are generated by the images of the cloning maps with codomain $G_n$. Then $\Thomp{G_*}$ is finitely generated.
\end{lemma}

\begin{proof}
 By the above discussion, we need only show that the $\dlkmodel{G_*}{n}$ are connected, for large enough $n$. Suppose $n$ is large enough that: (a) $G_n$ is generated by images of cloning maps, and (b) $n\ge 5$ so $\match(L_n)$ is connected. We will show that every vertex can be connected by an edge path to the vertex $[1,\OneEdge_1]$, where $\OneEdge_i$ denotes the spanning graph whose only edge connects the $i$th vertex to the $(i+1)$st. So let $[g,\Gamma]$ be a vertex of $\dlkmodel{G_*}{n}$ and write $g=s_1\cdots s_r$, where the $s_i$ are generators coming from images of cloning maps $s_i\in \image(\clone_{k_i})$ for some $k_i$. Since $\match(L_n)$ is connected, there is a path in $\dlkmodel{G_*}{n}$ from $[s_1\cdots s_r,\Gamma]$ to $[s_1\cdots s_r,\OneEdge_{k_r}]=[s_1\cdots s_{r-1},((s_r)\clone_{k_r}^{-1})\cdot\OneEdge_{k_r}]$. Repeating this $r$ times, we connect to $[1,\OneEdge_k]$ for some $k$, and then to $[1,\OneEdge_1]$.
\end{proof}

\subsection{Proving high connectivity}\label{sec:high_connectivity}

As we have seen, Morse theory is a tool that allows one to show that a pair $(X,X_0)$ is highly connected. We will eventually want to inductively apply this to the situation where $X = \dlkmodel{G_*}{n}$ and $X_0 = \dlkmodel{G_*}{n-k}$ for some $k \in \N$. This is insufficient to conclude that the connectivity tends to infinity though, because we would be trying to get $X$ to be more highly connected than $X_0$. The following lemma expresses the degree of insufficiency. The lemma is straightforward to prove but can be seen as a roadmap for the argument that follows.

\begin{lemma}\label{lem:relative_connectivity}
 Let $(X,X_0)$ be a $k$-connected CW-pair. Assume that $X_0$ is $(k-1)$-connected. Then $X$ is $k$-connected if and only if $\pi_k(X_0 \to X)$ is trivial.
\end{lemma}

\begin{proof}
 Consider the part of the homotopy long exact sequence associated to $(X,X_0)$:
 \[
 \pi_{j+1}(X,X_0) \to \pi_j(X_0) \stackrel{\iota_j}{\to} \pi_j(X) \to \pi_j(X,X_0) \text{.}
 \]
 For $j < k$ the map $\iota_j$ is an isomorphism and $\pi_j(X_0)$ trivial. For $j = k$ it is an epimorphism, so indeed $\pi_k(X)$ is trivial if and only if $\iota_k$ is.
\end{proof}

In our applications we will know $X_0$ to be $(k-1)$-connected by induction and $(X,X_0)$ will be seen to be $k$-connected using Morse theory. To show that $\pi_k(X_0 \to X)$ is trivial we will use a relative variant of the Hatcher flow for arc complexes that was shown to us by Andrew Putman (Proposition~\ref{prop:putman_flow} below). Before we can prove it we need some technical preliminaries.

A \emph{combinatorial~$k$-sphere (respectively~$k$-disk)} is a simplicial complex that can be subdivided to be isomorphic to a subdivision of the boundary of a $(k+1)$-simplex (respectively to a subdivision of a $k$-simplex).  An $m$-dimensional \emph{combinatorial manifold} is an $m$-dimensional simplicial complex in which the link of every simplex~$\sigma$ of dimension $k$ is a combinatorial $(m-k-1)$-sphere.  In an $m$-dimensional \emph{combinatorial manifold with boundary} the link of a $k$-simplex $\sigma$ is allowed to be homeomorphic to a combinatorial~$(m-k-1)$-disk; its \emph{boundary} consists of all the simplices whose link is indeed a disk.

A simplicial map is called \emph{simplexwise injective} if its restriction to any simplex is injective. The following is Lemma~3.8 of \cite{bux14}, cf.\ also the proof of Proposition~5.2 in \cite{putman12}.

\begin{lemma}\label{lem:injectifying}
 Let $Y$ be a $k$-dimensional combinatorial manifold.  Let $X$ be a simplicial complex and assume that the link of every $d$-simplex in $X$ is $(k-2d-2)$-connected for $d \ge 0$.  Let $\psi \colon Y \to X$ be a simplicial map whose restriction to $\partial Y$ is simplexwise injective.  Upon changing the simplicial structure of $Y$, $\psi$ is homotopic relative $\partial Y$ to a simplexwise injective map.
\end{lemma}

In practice $Y$ will be a sphere, so the lemma allows us to restrict attention to simplexwise injective combinatorial maps when collapsing spheres.

For the proposition, we need one more technical definition. Let $X$ be a simplicial complex and $w$ a vertex. We say that $X$ is \emph{conical at} $w$ if for any simplex $\sigma$, as soon as every vertex of $\sigma$ lies in the closed star $\st(w)$ then so does $\sigma$ (that is, the star of $w$ is the cone over the link of $w$). In particular, if $X$ is a flag complex then it is conical at every vertex.

\begin{proposition}\label{prop:putman_flow}
 Let $X_0 \subseteq X_1 \subseteq X$ be simplicial complexes. Assume that $(X,X_0)$ is $k$-connected, that $X_0$ is $(k-1)$-connected and that the link of every $d$-simplex is $(k-2d-2)$-connected for $d \ge 0$. Further assume the following ``exchange condition'':
 \begin{enumerate}[label={(EXC)}, ref={EXC}, leftmargin=*]
  \item There is a vertex $w \in X$ at which $X$ is conical, such that for every vertex $v \in X_0$ that is not in $\st w$ there is a vertex $v' \in \st_{X_1} w$ such that $\lk_{X_1} v \subseteq \lk_{X_1} v'$ and $\lk_{X_1} v$ is $(k-1)$-connected.\label{item:exchange1}
 \end{enumerate}
 Then $X$ is $k$-connected.
\end{proposition}
%

\begin{proof}
 Let $\iota \colon X_0 \to X$ denote the inclusion. In view of Lemma~\ref{lem:relative_connectivity}, all that needs to be shown is that if $\varphi \colon S^k \to X_0$ is a map from a $k$-sphere then $\bar{\varphi} \defeq \iota \circ \varphi$ is homotopically trivial.

 By simplicial approximation \cite[Theorem~3.4.8]{spanier66} we may assume $\varphi$ (and thus $\bar{\varphi}$) to be a simplicial map $Y \to X_0$ and by our assumptions and Lemma~\ref{lem:injectifying} we may assume it to be simplexwise injective. Our goal is to homotope $\bar{\varphi}$ to a map to $\st w$. Once we have achieved that, we are done since $\st w$ is contractible.

 The simplicial sphere $Y$ contains finitely many vertices $x$ whose image $v = \bar{\varphi}(x)$ does not lie in $\st w$. Pick one and define $\bar{\varphi}' \colon Y \to X$ to be the map that coincides with $\bar{\varphi}$ outside the open star of $x$ and takes $x$ to the vertex $v'$ from the statement. We claim that $\bar{\varphi}$ is homotopic to $\bar{\varphi}'$. Inductively replacing vertices then finishes the proof, since $X$ is conical at $w$.

 It remains to show that $\bar{\varphi} |_{\st x}$ and $\bar{\varphi}'|_{\st x}$ are homotopic relative to $\lk x$. Note that $\bar{\varphi}(\lk x) \subseteq \lk v$ by simplexwise injectivity. Furthermore the complex spanned by $v$, $v'$ and $\lk v$ is the suspension $\Sigma(\lk v)$ of $\lk v$ (unless $v$ and $v'$ are adjacent in which case there is nothing to show). So both $\bar{\varphi}|_{\st x}$ and $\bar{\varphi}'|_{\st x}$ are maps $(D^k, S^{k-1}) \cong (\st x, \lk x) \to (\Sigma(\lk v),\lk v)$. But $\lk v$ is $(k-1)$-connected by assumption so $(\Sigma(\lk v), \lk v)$ is $k$-connected and both maps are homotopic.
\end{proof}

\subsection{Proving negative finiteness properties}\label{sec:relative_brown}

We have already seen that if the $G_*$ are not eventually of type~$\F_n$, then Brown's criterion applied to the Stein--Farley complex cannot be used to show that $\Thomp{G_*}$ is not of type~$\F_n$. In Section~\ref{sec:mtx_neg_fin_props}, when the $G_n$ are matrix groups, we will instead use a different action, together with the following result. It is formulated in terms of the homological finiteness properties $\FP_n$. The relationship is explained for example in \cite[Chapter~8]{geoghegan08}, but we mostly just need to know the fact that a group of type~$\F_n$ is also of type~$\FP_n$. Note that for $\Lambda = \Gamma$ the following theorem is essentially one half of Brown's criterion.

\begin{theorem}\label{thrm:relative_brown}
 Let $\Lambda$ be a group and let $\Gamma$ be a subgroup. Let $Y$ be a CW complex on which $\Lambda$ acts. Assume that $Y$ is $(n-1)$-acyclic and that the stabilizer of every $p$-cell in $Y$ (in $\Lambda$ as well as in $\Gamma$) is of type~$\FP_{n-p}$. Let $Z$ be a $\Gamma$-cocompact subspace of $Y$ . Let $(Y_\alpha)_{\alpha \in I}$ be a $\Lambda$-cocompact filtration of $Y$. Assume that there is no $\alpha$ with $Z\subseteq Y_\alpha$ such that the map $\tilde{H}_{n-1}(Z \into Y_\alpha)$ is trivial. Then no group $\Delta$ through which the inclusion $\Gamma \into \Lambda$ factors is of type~$\FP_n$.
\end{theorem}

The application is similar in spirit to that of \cite{krstic97}, where a morphism $\Gamma \to \Lambda$ is constructed that cannot factor through a finitely presented group. The proof should be compared to \cite[Theorem~2.2]{brown87}.

\begin{proof}
 For $n = 1$ suppose that $\Gamma$ is contained in a finitely generated subgroup $\gen{S}$ of $\Lambda$. Let $K$ be a compact subspace such that $\Gamma.K = Z$. Since $Y$ is connected, we can add finitely many edges to $K$ and take $Z$ to be its $\Gamma$-orbit, so without loss of generality $K$ is connected. For every $s \in S$ we may pick an edge path $p_s$ that connects $K$ to $s.K$. Let $P \defeq \bigcup\{p_s \mid s \in S\}$. Now any two points in $Z$ can be connected in $\Gamma.(K \cup P)$. In other words, the map $\tilde{H}_0(Z \to \Gamma.(K \cup P))$ is trivial. But $K \cup P$ and thus $\Gamma.(K \cup P)$ is contained in some $Y_\alpha$, contradicting the assumption.

 From now on we assume that $n > 1$. Our goal is to find an index set $J$ such that the map $H_{n-1}(\Gamma,\prod_J \Z\Gamma) \to H_{n-1}(\Lambda,\prod_J \Z\Lambda)$ is non-trivial. The result then follows from the Bieri--Eckmann criterion \cite[Proposition~1.2]{bieri74}, because if this map factors through $H_{n-1}(\Delta,\prod_J \Z\Delta)$ then the latter module cannot be zero. Note that $Z$ is contained in a subfiltration of $(Y_\alpha)_\alpha$ so we may assume without loss of generality that $Z$ is contained in all $Y_\alpha, \alpha \in I$.

 Let $J$ be a cofinal set in $I$ (for instance all of $I$) and for $\alpha \in J$ let $c_\alpha \in H_{n-1}(Z)$ be such that the image in $H_{n-1}(Y_\alpha)$ is non-trivial. By the arguments in the proof of \cite[Theorem~2.2]{brown87} we have the isomorphisms in the following diagram (essentially the two vertical arrows at the top are isomorphisms because $Y$ is $(n-1)$-acyclic and the two vertical arrows at the bottom are isomorphisms by cocompactness of the actions and the assumptions on the finiteness properties of the stabilizers).
 \begin{diagram}
 H_{n-1}(\Gamma,\prod_J \Z\Gamma) & \rTo & H_{n-1}(\Lambda,\prod_J \Z\Lambda)\\
 \uTo^\cong && \uTo^\cong\\
 H_{n-1}^\Gamma(Y, \prod_J \Z\Gamma) & \rTo & H_{n-1}^{\Lambda}(Y,\prod_J \Z\Lambda)\\
 \uTo && \uTo^\cong\\
 H_{n-1}^\Gamma(Z, \prod_J \Z\Gamma) & \rTo & \varinjlim H_{n-1}^{\Lambda}(Y_\alpha,\prod_J \Z\Lambda)\\
 \dTo^\cong && \dTo^\cong\\
 \prod_J H_{n-1}(Z) &\rTo & \varinjlim \prod_J H_{n-1}(Y_\alpha)\text{.}
 \end{diagram}

 Assuming that the diagram commutes, the chain $(c_\alpha)_{\alpha \in J} \in \prod_J H_{n-1}(Z)$ has non-trivial image in $\varinjlim \prod_J H_{n-1}(Y_\alpha)$ and we are done.

 The rest of the proof will be concerned with the commutativity of the diagram. The only square whose commutativity is not clear is the bottom one. In what follows, all products are taken over $J$ which we suppress from notation.

 Let $C_*$, $C_*^\alpha$, and $D_*$ be the cellular chain complexes of $Y$, $Y_\alpha$, and $Z$ (respectively). Let $P_* \to \Z$ be a resolution by projective $\Z\Lambda$-modules (which are also projective $\Z\Gamma$-modules). The third horizontal map is induced by the maps $P_q \otimes_{\Gamma} (D_p \otimes \prod \Z\Gamma) \to P_q \otimes_{\Lambda} (C_p^\alpha \otimes \prod \Z\Lambda)$. (Or equivalently $(P_q \otimes D_p) \otimes_{\Gamma} \prod \Z\Gamma \to (P_q \otimes C_p^\alpha) \otimes_{\Lambda} \prod \Z\Lambda)$, which is the same since the tensor product is associative and, using the notation from \cite[p.~55]{brown82}, also commutative.) The bottom horizontal map is just induced by $D_* \to C_*^\alpha$. The lower vertical maps come from spectral sequences
 \begin{align*}
 E^1_{pq} &= \Tor^\Gamma_q(D_p,\prod \Z\Gamma) \Rightarrow H_{p+q}^\Gamma(Z,\prod \Z\Gamma)\text{ and }\\
 E^1_{pq} &= \Tor^\Lambda_q(C_p^\alpha,\prod \Z\Lambda) \Rightarrow H_{p+q}^\Lambda(Y_\alpha,\prod \Z\Lambda)\text{.}
 \end{align*}
 The finiteness and cocompactness assumptions guarantee that $D_p$ is of type~$\FP_{n-p}$ over $\Z\Gamma$ and $C_p^\alpha$ is of type~$\FP_{n-p}$ over $\Z\Lambda$ so that the natural maps
 \begin{align*}
 \Tor^\Gamma_q(D_p,\prod \Z\Gamma) &\to \prod \Tor^\Gamma_q(D_p,\Z\Gamma) \text{ and}\\
 \Tor^\Lambda_q(C_p^\alpha,\prod \Z\Lambda) &\to \prod \Tor^\Lambda_q(C_p^\alpha,\Z\Lambda)
 \end{align*}
 are isomorphisms and the spectral sequences collapse on the second page. We have the commutative diagram of chain complexes
 \begin{diagram}
 &\Tor^\Gamma_0(D_*,\prod \Z\Gamma) & \rTo & \Tor^\Lambda_0(C_*^\alpha,\prod \Z\Lambda)&\\
 &\dTo^\cong && \dTo^\cong&\\
 \prod D_* = &\prod \Tor^\Gamma_0(D_*,\Z\Gamma) & \rTo & \prod\Tor^\Lambda_0(C_*^\alpha,\Z\Lambda)& = \prod C_*^\alpha
 \end{diagram}
 and taking homology in degree $n-1$ gives the commutative diagram
 \begin{diagram}
 H_{n-1}^\Gamma(Z,\prod \Z\Gamma) & \rTo & H_{n-1}^\Lambda(Y_\alpha,\prod \Z\Lambda)\\
 \dTo^\cong && \dTo^\cong\\
 \prod H_{n-1}(Z) & \rTo & \prod H_{n-1}(Y_\alpha)
 \end{diagram}
 that we were looking for.
\end{proof}

\section{A Thompson group for direct products of a group}\label{sec:direct_prods}
The examples in this section were constructed independently by Slobodan Tanusevski in his PhD thesis \cite{tanusevski14}, using entirely different techniques, and in discussions with him we have determined that his groups are identical to those discussed here.

Fix a group $G$. Let $G_n$ be the direct power $G^n$. We declare that $\rho_n$ is trivial for all $n$, and define cloning maps via $(g_1,\dots,g_k,\dots,g_n)\clone^n_k \defeq (g_1,\dots,g_k,g_k,\dots,g_n)$. This makes rather literal the word ``cloning.'' To verify that this defines a cloning system, observe that since the $\rho_n$ are trivial, we need only check that the cloning maps are homomorphisms (which they are) and that $\clone_\ell^n \circ \clone_k^{n+1} = \clone_k^n \circ \clone_{\ell+1}^{n+1}$ for $1\le k<\ell \le n$ (which is visibly true). These respectively handle conditions~\eqref{item:fcs_cloning_a_product} and~\eqref{item:fcs_product_of_clonings} of Definition~\ref{def:filtered_cloning_system}, and condition~\eqref{item:fcs_compatibility} is trivial. Lastly, the cloning system is visibly properly graded.

It turns out that this cloning system is an example answering Question~\ref{ques:finiteness_conjecture} positively, that is, the finiteness length of $\Thomp{G^*}$ is exactly that of $G$ (notationally, the asterisk is a superscript now because we are considering the family of direct powers $(G^n)_{n \in \N}$). The proof is due to Tanusevski and we sketch a version of it here, using our setup and language. For the positive finiteness properties, we just need that the complexes $\dlkmodel{G^*}{n}$ become increasingly highly connected. This follows by noting that every simplex fiber of the projection $\dlkmodel{G^*}{n} \to \match(L_n)$ is the join of its vertex fibers, and applying \cite[Theorem~9.1]{quillen78}. For the negative finiteness properties, we claim that there is a sequence of homomorphisms $G\to \Thomp{G^*}\to G$ that composes to the identity. This is sufficient by the Bieri--Eckmann criterion \cite[Proposition~1.2]{bieri74}; see \cite[Proposition~4.1]{bux04}. The first map in the claim is $g\mapsto [1,g,1]$, and the second is $[T_-,(g_1,\dots,g_n),T_+] \mapsto g_1$. One must check that this second map is well defined on equivalence classes under reduction and expansion, and is a homomorphism, but this is not hard to see.

A variation of these groups was recently studied using cloning systems, by Berns-Zieve, Fry, Gillings, and Mathews \cite{berns-zieve14}. With the above setup, they consider cloning maps of the form $(g_1,\dots,g_k,\dots,g_n)\clone^n_k \defeq (g_1,\dots,g_k,\phi(g_k),\dots,g_n)$ where $\phi\in\Aut(G)$. They prove that for $G$ finite, the resulting Thompson group is co$\mathcal{CF}$. If these groups turn out to not embed into $V$, which seems believable when $\phi\ne\id$, then they would be counterexamples to the conjecture that $V$ is universal co$\mathcal{CF}$.

\section{Thompson groups for matrix groups}\label{sec:matrix_groups}

Let $R$ be a unital ring and consider the algebra of $n$-by-$n$ matrices $M_n(R)$. We will define a family of injective functions $M_n(R) \to M_{n+1}(R)$, which will become cloning maps after we restrict to the subgroups of upper triangular matrices $B_n(R)$. Consider the map $\clone_k$ defined by
\[
\left(\left(
\begin{array}{ccc}
A_{<,<} & A_{<,k} & A_{<,>}\\
A_{k,<} & A_{k,k} & A_{k,>}\\
A_{>,<} & A_{>,k} & A_{>,>}
\end{array}
\right)\right)\clone_k
=
\left(
\begin{array}{cccc}
A_{<,<} & A_{<,k} & A_{<,k} & A_{<,>}\\
A_{k,<} & A_{k,k} & 0 & 0\\
0 & 0 & A_{k,k} & A_{k,>}\\
A_{>,<} & A_{>,k} & A_{>,k}  & A_{>,>}
\end{array}
\right)
\]
where the matrix has a block structure under which the middle column and row are the $k$th column and row of the full matrix respectively. Given the block structure it is not hard to see that $\clone_k$ is a morphism of monoids, but it generally fails to map invertible elements to invertible elements. We therefore restrict to the groups $B_n(R)$ of invertible upper triangular matrices. Let $B_\infty(R) = \varinjlim B_n(R)$.

\begin{lemma}\label{lem:up_tri_clone}
 The trivial morphisms $\rho_n$ and the maps $\clone_k^n$ defined above describe a properly graded cloning system on $B_*(R)$.
\end{lemma}

It may be noted that the action of $\Fmonoid$ on $B_\infty(R)$ factors through $\Hmonoid$, that is $\clone_\ell \clone_k = \clone_k \clone_{\ell+1}$ even for $\ell = k$.

\begin{proof}
 Since $\rho_*$ is trivial, condition~\eqref{item:fcs_cloning_a_product} asks that the cloning maps be group homomorphisms. That $\clone_k$ is multiplicative and takes $1$ to $1$ is straightforward to check. Also, $A$ is invertible if and only if all the $A_{i,i}$ are units, in which case $(A)\clone_k$ is also invertible.
 
 To check condition~\eqref{item:fcs_product_of_clonings} it is helpful to note that for any $A\in M_n(R)$, $((A)\clone_k)_{i,j} = A_{\pi_k(i),\pi_k(j)}$ unless $i = k$ or $i>j$ (here $\pi_k$ is as in Example~\ref{ex:symm_gps}). One can now distinguish cases similar to Example~\ref{ex:symm_gps}. The compatibility condition \eqref{item:fcs_compatibility} is vacuous for trivial $\rho_*$.

 To see that the cloning system is properly graded note that $g \in \image \iota_{n,n+1}$ if and only if the last column of $g$ is the vector $e_{n+1}$. If at the same time $g = (h)\clone_k$ then by the definition of $\clone_k$ the last column of $h$ has to be $e_n$. Hence $h \in \image \iota_{n-1,n}$.
\end{proof}

Having equipped $B_*(R)$ with a cloning system, we get a generalized Thompson group $\Thomp{B_*(R)}$. Elements are represented by triples $(T_-,A,T_+)$ for trees $T_\pm$ with $n$ leaves and matrices $A\in B_n(R)$, up to reduction and expansion. Figure~\ref{fig:borel_thomp_expansion} gives an example of an element of $\Thomp{B_*(R)}$, represented as a triple and an expansion of that triple.

\begin{figure}[ht]
 \centering
\begin{tikzpicture}[line width=0.8pt, scale=0.3]
  \begin{scope}
  \node at (-3,-1){$\Bigg[$};
  \draw
   (-2,-2) -- (0,0) -- (2,-2)   (0,-2) -- (-1,-1);
  \filldraw
   (-2,-2) circle (1.5pt)   (0,0) circle (1.5pt)   (2,-2) circle (1.5pt)   (-1,-1) circle (1.5pt)   (0,-2) circle (1.5pt);
   \node at (2.5,-2){,};
   \node at (5,-1){$\begin{psmallmatrix}
                    1&2&3\\0&4&5\\0&0&6
                   \end{psmallmatrix}$};
                   
   \node at (7.5,-2){,};
   \begin{scope}[xshift=4in]    
    \draw
   (-2,-2) -- (0,0) -- (2,-2)   (0,-2) -- (1,-1);
  \filldraw
   (-2,-2) circle (1.5pt)   (0,0) circle (1.5pt)   (2,-2) circle (1.5pt)   (1,-1) circle (1.5pt)   (0,-2) circle (1.5pt);
   \node at (3,-1){$\Bigg]$};
   \end{scope}
   \end{scope}

    \node at (14,-1){$=$};
   
   \begin{scope}[xshift=18cm]
    \node at (-3,-1){$\Bigg[$};
    \draw
   (-2,-2) -- (0,0) -- (2,-2)   (0,-2) -- (-1,-1)   (-1,-2) -- (-.5,-1.5);
    \filldraw
   (-2,-2) circle (1.5pt)   (0,0) circle (1.5pt)   (2,-2) circle (1.5pt)   (-1,-1) circle (1.5pt)   (0,-2) circle (1.5pt)   (-1,-2) circle (1.5pt);
   \node at (3,-2){,};
   \node at (6.5,-1){$\begin{psmallmatrix}
                    1&2&2&3\\0&4&0&0\\0&0&4&5\\0&0&0&6
                   \end{psmallmatrix}$};
                   
   \node at (9.75,-2){,};
   \begin{scope}[xshift=5in]    
    \draw
   (-2,-2) -- (0,0) -- (2,-2)   (0,-2) -- (1,-1)   (.5,-1.5) -- (1,-2);
  \filldraw
   (-2,-2) circle (1.5pt)   (0,0) circle (1.5pt)   (2,-2) circle (1.5pt)   (1,-1) circle (1.5pt)   (0,-2) circle (1.5pt)   (1,-2) circle (1.5pt);
   \node at (3,-1){$\Bigg]$};
   \end{scope}
   \end{scope}
  
\end{tikzpicture}
\caption[Expansion for matrix groups]{An example of expansion in $\Thomp{B_*(\Q)}$.}
\label{fig:borel_thomp_expansion}
\end{figure}

We are interested in finiteness properties of $\Thomp{B_*(R)}$ because of the following examples where the groups $B_*(R)$ themselves have interesting finiteness properties, see \cite[Theorem~A, Remarks~3.6,~3.7]{bux04}.

\begin{theorem}\label{thm:bux}
 Let $\GlobalField$ be a global function field, let $S$ be a finite nonempty set of places and $\calO_S$ the ring of $S$-integers. Then $B_n(\calO_S)$ is of type~$\F_{\abs{S}-1}$ but not of type~$\F_{\abs{S}}$ for any $n \ge 2$.
\end{theorem}

For instance, when $R=\mathbb{F}_p[t,t^{-1}]$ then $B_n(\mathbb{F}_p[t,t^{-1}])$ is finitely generated but not finitely presented, for $n\ge2$. What is particularly interesting about Theorem~\ref{thm:bux} is that the finiteness properties of $B_n(\calO_S)$ depend on $\abs{S}$ but not on $n$.

A class of examples where the finiteness properties do depend on $n$ arises as subgroups of groups of the form $B_n(R)$. Let $\abels_n \le B_{n+1}$ be the group of invertible upper triangular $n+1$-by-$n+1$ matrices whose upper left and lower right entries are $1$. The groups $\abels_n(\Z[1/p])$ were studied by Abels and others and we call them the \emph{Abels groups}. Their finiteness length tends to infinity with $n$ \cite{abels87, brown87}:

\begin{theorem}\label{thm:abels}
For any prime $p$ the group $\abels_n(\Z[1/p])$ is of type~$\F_{n-1}$ but not of type~$\F_n$ for $n \ge 1$.
\end{theorem}

For any ring $R$, the cloning system described above for $B_n(R)$ preserves the Abels groups $\abels_{n-1}(R)$. By restriction we obtain a generalized Thompson group $\Thomp{\abels_{*-1}(R)}$ which we will just denote by $\Thomp{\abels_*(R)}$.

\section{Finiteness properties of Thompson groups for matrix groups}\label{sec:mtx_fin_props}

We will prove below that the finiteness length of $\Thomp{B_*(\calO_S)}$ is the same as that of all the $B_n(\calO_S)$. For consistency, we can state this as
\[
\phi(\Thomp{B_*(\calO_S)}) = \liminf_n \phi(B_n(\calO_S))\text{.}
\]
The inequality $\ge$, i.e., that $\Thomp{B_*(\calO_S)}$ is of type~$\F_{\abs{S}-1}$, is proved in Section~\ref{sec:mtx_pos_fin_props}, and follows the general strategy outlined in Sections~\ref{sec:spaces} and~\ref{sec:finiteness_props}. In fact, it applies to arbitrary rings. To show the inequality $\le$, i.e., that $\Thomp{B_*(\calO_S)}$ is not of type~$\FP_{\abs{S}}$, we develop some new tools in Section~\ref{sec:mtx_neg_fin_props}, and make use of the criterion established in Theorem~\ref{thrm:relative_brown}.

The proof showing the inequality $\ge$ above also applies to $\Thomp{\abels_*(R)}$. Since the right hand side is infinite this time, this directly gives the full equation
\[
\phi(\Thomp{\abels_*(\Z[1/p])}) = \liminf_n \phi(\abels_n(\Z[1/p]))\text{.}
\]

\subsection{Positive finiteness properties}\label{sec:mtx_pos_fin_props}

The first main result of this section is that the group $\Thomp{B_*(R)}$ has all the finiteness properties that the individual groups $B_*(R)$ eventually have:

\begin{theorem}\label{thrm:borel_thomp_pos_fin_props} 
 $\displaystyle\phi(\Thomp{B_*(R)}) \ge \liminf_n(\phi(B_n(R)))$.
\end{theorem}

In particular, together with Theorem~\ref{thm:bux} this implies:

\begin{corollary}
 $\Thomp{B_*(\calO_S)}$ is of type~$\F_{\abs{S}-1}$.
\end{corollary}

In view of Proposition~\ref{prop:generic_finiteness}, to prove Theorem~\ref{thrm:borel_thomp_pos_fin_props} it suffices to show that the connectivity of $\dlkmodel{B_*(R)}{n}$ goes to infinity with $n$. In fact, we will induct, so we need to consider a slightly larger class of complexes.

For a spanning subgraph $\Delta$ of the linear graph $L_n$, define $\dlkmodel{B_*(R);\Delta}{n}$ to be the subcomplex of $\dlkmodel{B_*(R)}{n}$ whose elements only use graphs that are subgraphs of $\Delta$. Define $e(\Delta)$ to be the number of edges of $\Delta$. Define $\eta(m)\defeq \lfloor\frac{m-1}{4}\rfloor$. Taking $\Delta = L_n$, Theorem~\ref{thrm:borel_thomp_pos_fin_props} will follow from:

\begin{proposition}\label{prop:borel_thomp_fin_props_strong}
 $\dlkmodel{B_*(R);\Delta}{n}$ is $(\eta(e(\Delta))-1)$-connected.
\end{proposition}

The base case is that $\dlkmodel{B_*(R);\Delta}{n}$ is non-empty provided $e(\Delta) \ge 1$, which is clearly true.

We need to do a bit of preparation before we can prove the proposition. To work with simplices of $\dlkmodel{B_*(R)}{n}$ it will be helpful to have simple representatives for dangling classes. To define them we have to recall some of the origins of $\dlkmodel{B_*(R)}{n}$: by Observation~\ref{obs:matchings_to_forests} matchings $\Gamma$ of $L_n$ correspond to elementary forests. Using this correspondence, it makes sense to denote the corresponding cloning map by $\clone_\Gamma$. In fact, since our cloning maps factor through the hedge monoid, we even get a cloning map $\clone_\Gamma$ for any spanning subgraph $\Gamma$ of $L_n$ using Observation~\ref{obs:graphs_to_hedges}. For the sake of readability, we describe this map explicitly. Let $D_k(\lambda)$ be the $k$-by-$k$ matrix with all diagonal entries $\lambda$ and all other entries $0$. Let $F_{k,\ell}(\lambda)$ be the $k$-by-$\ell$ matrix whose bottom row has all entries $\lambda$ and all other entries are $0$ and let $C_{k,\ell}(\lambda)$ be defined analogously for the top row. Assume that $\Gamma$ has $m$ connected components which we think of as numbered from left to right. Then
\[
\clone_\Gamma \colon M_m(R) \to M_n(R)
\]
can be described as follows. The image $\clone_\Gamma(A)$ has a block structure where columns and rows are grouped together if their indices lie in a common component of $\Gamma$. More precisely, the $(i,j)$-block has $k$ rows and $\ell$ columns if the $i$th (respectively $j$th) component of $\Gamma$ has $k$ (respectively $\ell$) vertices. The block is $D_k(A_{i,i})$, $F_{k,\ell}(A_{i,j})$ or $C_{k,\ell}(A_{i,j})$ depending on whether $i = j$, $i<j$, or $i>j$ (see Figure~\ref{fig:cloning}).

\begin{figure}[th]
\begin{multline*}
\left(
\begin{array}{ccc}
a_{1,1}&a_{1,2}&a_{1,3}\\
a_{2,1}&a_{2,2}&a_{2,3}\\
a_{3,1}&a_{3,2}&a_{3,3}
\end{array}
\right)
\stackrel{\clone_\Gamma}{\mapsto}
\left(
\begin{array}{ccc}
D_2(a_{1,1})&F_{2,4}(a_{1,2})&F_{2,3}(a_{1,3})\\
C_{4,2}(a_{2,1})&D_4(a_{2,2})&F_{4,3}(a_{2,3})\\
C_{3,2}(a_{3,1})&C_{3,4}(a_{3,2})&D_3(a_{3,3})
\end{array}
\right)\\[.5cm]
=\hspace{.5cm}
\left(
\begin{array}{c@{}cc|cccc|ccc}
& \tikzmark{t1} & \tikzmark{t2} & \tikzmark{t3} & \tikzmark{t4} & \tikzmark{t5} & \tikzmark{t6} &
\tikzmark{t7} & \tikzmark{t8}& \tikzmark{t9}\\[-.3cm]
\tikzmark{l1}&a_{1,1} & &                    &&&&                     &&\\
\tikzmark{l2}& &a_{1,1} &                    a_{1,2}&a_{1,2}&a_{1,2}&a_{1,2}&                     a_{1,3}&a_{1,3}&a_{1,3}\\
\hline
\tikzmark{l3}&a_{2,1} &a_{2,1} &                    a_{2,2}&&&&                     &&\\
\tikzmark{l4}&& &                    &a_{2,2}&&&                     &&\\
\tikzmark{l5}&& &                    &&a_{2,2}&&                     &&\\
\tikzmark{l6}&& &                    &&&a_{2,2}&                     a_{2,3}&a_{2,3}&a_{2,3}\\
\hline
\tikzmark{l7}&a_{3,1} &a_{3,1} &                    a_{3,2}&a_{3,2}&a_{3,2}&a_{3,2}&                     a_{3,3}&&\\
\tikzmark{l8}&& &                    &&&&                     &a_{3,3}&\\
\tikzmark{l9}&& &                    &&&&                     &&a_{3,3}
\end{array}
\right)
\end{multline*}
\tikz[remember picture, overlay,yshift=.5cm]
   \filldraw
     (pic cs:t1) circle (1.5pt)
     (pic cs:t2) circle (1.5pt)
     (pic cs:t3) circle (1.5pt)
     (pic cs:t4) circle (1.5pt)
     (pic cs:t5) circle (1.5pt)
     (pic cs:t6) circle (1.5pt)
     (pic cs:t7) circle (1.5pt)
     (pic cs:t8) circle (1.5pt)
     (pic cs:t9) circle (1.5pt);
\tikz[remember picture, overlay,yshift=.5cm]
   \draw
     (pic cs:t1) -- (pic cs:t2)  (pic cs:t3)  -- (pic cs:t4)  -- (pic cs:t5)  -- (pic cs:t6)  (pic cs:t7)  -- (pic cs:t8)  -- (pic cs:t9) ;
\tikz[remember picture, overlay,xshift=-.7cm]
   \filldraw
     (pic cs:l1) circle (1.5pt)
     (pic cs:l2) circle (1.5pt)
     (pic cs:l3) circle (1.5pt)
     (pic cs:l4) circle (1.5pt)
     (pic cs:l5) circle (1.5pt)
     (pic cs:l6) circle (1.5pt)
     (pic cs:l7) circle (1.5pt)
     (pic cs:l8) circle (1.5pt)
     (pic cs:l9) circle (1.5pt);
\tikz[remember picture, overlay,xshift=-.7cm]
   \draw
     (pic cs:l1) -- (pic cs:l2)   (pic cs:l3)  -- (pic cs:l4)  -- (pic cs:l5)  -- (pic cs:l6)  (pic cs:l7)  -- (pic cs:l8)  -- (pic cs:l9) ;

\caption[Cloning map of a graph]{Visualization of the cloning map of a graph. The graph $\Gamma$ is drawn on top and to the left of the last matrix.}
\label{fig:cloning}
\end{figure}

Recall that we denote by $e_k$ the $k$th edge of $L_n$. We denote by $\OneEdge_k$ the matching of $L_n$ whose only edge is $e_k$ (as we did in Lemma~\ref{lem:fin_gen_case}). For a spanning subgraph $\Gamma$ of $L_n$ we say that an index $i$ is \emph{fragile} if $e_i \in \Gamma$ and we say that $i$ is \emph{stable} otherwise. In other words, $i$ is stable if it is the rightmost vertex of its component in $\Gamma$. A matrix $A \in M_n(R)$ is said to be \emph{modeled on} $\Gamma$ if $A_{i,j} = 0$ whenever both $i$ and $j$ are stable in $\Gamma$ (see Figure~\ref{fig:modeled}).

\begin{figure}[th]
\[
\left(
\begin{array}{c@{}cc|cccc|ccc}
& \tikzmark{tt1} & \tikzmark{tt2} & \tikzmark{tt3} & \tikzmark{tt4} & \tikzmark{tt5} & \tikzmark{tt6} &
\tikzmark{tt7} & \tikzmark{tt8}& \tikzmark{tt9}\\[-.3cm]
\tikzmark{ll1}&* &* &    *                &*&*&*&  *                  &*&*\\
\tikzmark{ll2}&* & 0&         *           &*&*&0&             *       &*&0\\
\hline
\tikzmark{ll3}&* &* & *                   &*&*&*& *                   &*&*\\
\tikzmark{ll4}&*& *&  *                  &*&*&*& *                    &*&*\\
\tikzmark{ll5}&*& *& *                   &*&*&*&   *                  &*&*\\
\tikzmark{ll6}&*&0 & *                   &*&*&0& *                    &*&0\\
\hline
\tikzmark{ll7}&* &* &*                    &*&*&*&*                    &*&*\\
\tikzmark{ll8}&*&* &   *                 &*&*&*&  *                   &*&*\\
\tikzmark{ll9}&*& 0&       *             &*&*&0&       *              &*& 0
\end{array}
\right)
\quad\quad\quad
\left(
\begin{array}{c@{}cc|cccc|ccc}
& \tikzmark{ttt1} & \tikzmark{ttt2} & \tikzmark{ttt3} & \tikzmark{ttt4} & \tikzmark{ttt5} & \tikzmark{ttt6} &
\tikzmark{ttt7} & \tikzmark{ttt8}& \tikzmark{ttt9}\\[-.3cm]
\tikzmark{lll1}&* &* &    *                &*&*&*&  *                  &*&*\\
\tikzmark{lll2}& & 1&         *           &*&*&0&             *       &*&0\\
\hline
\tikzmark{lll3}& & &                  * &*&*&*& *                   &*&*\\
\tikzmark{lll4}&& &                    &*&*&*& *                    &*&*\\
\tikzmark{lll5}&& &                    &&*&*&   *                  &*&*\\
\tikzmark{lll6}&& &                    &&&1& *                    &*&0\\
\hline
\tikzmark{lll7}& & &                    &&&&*                    &*&*\\
\tikzmark{lll8}&& &                    &&&&                     &*&*\\
\tikzmark{lll9}&& &                    &&&&                     && 1
\end{array}
\right)
\]
\tikz[remember picture, overlay,yshift=.5cm]
   \filldraw
     (pic cs:tt1) circle (1.5pt)
     (pic cs:tt2) circle (1.5pt)
     (pic cs:tt3) circle (1.5pt)
     (pic cs:tt4) circle (1.5pt)
     (pic cs:tt5) circle (1.5pt)
     (pic cs:tt6) circle (1.5pt)
     (pic cs:tt7) circle (1.5pt)
     (pic cs:tt8) circle (1.5pt)
     (pic cs:tt9) circle (1.5pt);
\tikz[remember picture, overlay,yshift=.5cm]
   \draw
     (pic cs:tt1) -- (pic cs:tt2)  (pic cs:tt3)  -- (pic cs:tt4)  -- (pic cs:tt5)  -- (pic cs:tt6)  (pic cs:tt7)  -- (pic cs:tt8)  -- (pic cs:tt9) ;
\tikz[remember picture, overlay,xshift=-.7cm]
   \filldraw
     (pic cs:ll1) circle (1.5pt)
     (pic cs:ll2) circle (1.5pt)
     (pic cs:ll3) circle (1.5pt)
     (pic cs:ll4) circle (1.5pt)
     (pic cs:ll5) circle (1.5pt)
     (pic cs:ll6) circle (1.5pt)
     (pic cs:ll7) circle (1.5pt)
     (pic cs:ll8) circle (1.5pt)
     (pic cs:ll9) circle (1.5pt);
\tikz[remember picture, overlay,xshift=-.7cm]
   \draw
     (pic cs:ll1) -- (pic cs:ll2)   (pic cs:ll3)  -- (pic cs:ll4)  -- (pic cs:ll5)  -- (pic cs:ll6)  (pic cs:ll7)  -- (pic cs:ll8)  -- (pic cs:ll9) ;
\tikz[remember picture, overlay,yshift=.5cm]
   \filldraw
     (pic cs:ttt1) circle (1.5pt)
     (pic cs:ttt2) circle (1.5pt)
     (pic cs:ttt3) circle (1.5pt)
     (pic cs:ttt4) circle (1.5pt)
     (pic cs:ttt5) circle (1.5pt)
     (pic cs:ttt6) circle (1.5pt)
     (pic cs:ttt7) circle (1.5pt)
     (pic cs:ttt8) circle (1.5pt)
     (pic cs:ttt9) circle (1.5pt);
\tikz[remember picture, overlay,yshift=.5cm]
   \draw
     (pic cs:ttt1) -- (pic cs:ttt2)  (pic cs:ttt3)  -- (pic cs:ttt4)  -- (pic cs:ttt5)  -- (pic cs:ttt6)  (pic cs:ttt7)  -- (pic cs:ttt8)  -- (pic cs:ttt9) ;
\tikz[remember picture, overlay,xshift=-.7cm]
   \filldraw
     (pic cs:lll1) circle (1.5pt)
     (pic cs:lll2) circle (1.5pt)
     (pic cs:lll3) circle (1.5pt)
     (pic cs:lll4) circle (1.5pt)
     (pic cs:lll5) circle (1.5pt)
     (pic cs:lll6) circle (1.5pt)
     (pic cs:lll7) circle (1.5pt)
     (pic cs:lll8) circle (1.5pt)
     (pic cs:lll9) circle (1.5pt);
\tikz[remember picture, overlay,xshift=-.7cm]
   \draw
     (pic cs:lll1) -- (pic cs:lll2)   (pic cs:lll3)  -- (pic cs:lll4)  -- (pic cs:lll5)  -- (pic cs:lll6)  (pic cs:lll7)  -- (pic cs:lll8)  -- (pic cs:lll9) ;
\caption[Matrix reduced relative to a graph]{A matrix that is modeled on a graph (left) and an upper triangular matrix that is reduced relative to a graph (right).}
\label{fig:modeled}
\end{figure}

\begin{lemma}\label{lem:borel_reduction}
 Let $\Gamma$ be a spanning subgraph of $L_n$ with $m$ components and let $A \in B_n(R)$. There is a representative $B$ in the coset $A(B_m(R))\clone_\Gamma$ such that $B-I_n$ is modeled on $\Gamma$. Moreover, rows of zeroes in $A$ (off the diagonal) can be preserved in $B$.
\end{lemma}

\begin{proof}
 We inductively multiply $A$ on the right by matrices in $(B_m(R))\clone_\Gamma$ to eventually obtain $B$. Let $E_{i,j}(\lambda)$ denote the matrix that coincides with the identity matrix in all entries but $(i,j)$ and is $\lambda$ there.

 We begin by clearing the diagonal. Let $i$ be the (stable) rightmost vertex of the $k$th component of $\Gamma$ and let $\lambda = A_{i,i}^{-1}$. Then $A(E_{k,k}(\lambda))\clone_\Gamma$ has $(i,i)$-entry one and no other diagonal entry with stable indices was affected.

 Now we clear the region above the diagonal. We proceed inductively by rows and columns. Let $(i,j)$ be the (lexicographically) minimal pair of stable indices of $\Gamma$ such that $0 \ne A_{i,j} \eqdef - \lambda$. Let $i$ and $j$ lie in the $k$th respectively $\ell$th component of $\Gamma$. Then $A(E_{k,\ell}(\lambda))\clone_{m}$ has $(i,j)$-entry zero and no other entry with stable indices was affected.

 For the last statement assume that the $i$th row of $A$ was zero off the diagonal. Then none of the matrices by which we multiplied had a nonzero off-diagonal entry in the $i$th row. If $i$ is fragile no such matrix even lies in $(B_m(R))\clone_\Gamma$. If $i$ is stable then the only matrices we might have used of this form were meant to clear the $i$th row, but since the entries there were zero, nothing happened in these steps.
\end{proof}

\begin{corollary}[Reduced form]\label{cor:borel_reduced_form}
 Every simplex in $\dlkmodel{B_*(R)}{n}$ has a representative $(A,\Gamma)$ such that the matrix $A - I_n$ is modeled on $\Gamma$. \qed
\end{corollary}

We will refer to a matrix $A\in B_n(R)$ as being \emph{reduced relative $\Gamma$} if it satisfies the conclusion of Corollary~\ref{cor:borel_reduced_form}.

The next sequence of lemmas is a gradual checking of the hypotheses of Proposition~\ref{prop:putman_flow}, still in the context of an induction proof, ultimately leading to a proof of Proposition~\ref{prop:borel_thomp_fin_props_strong}.

\begin{lemma}[Flag complex]\label{lem:mtx_cpx_flag}
 $\dlkmodel{B_*(R);\Delta}{n}$ is a flag complex.
\end{lemma}

\begin{proof}
 We need to show that any collection of vertices $\{v_1,\dots,v_r\}$ that are pairwise connected by edges spans a simplex. We induct on $r$ (with the trivial base case of $r\le 2$). Each vertex $v_i$ in our collection is of the form $[A_i,\OneEdge_{k_i}]$ for $\OneEdge_{k_i}$ some single-edge subgraphs of $\Delta$. Assume without loss of generality that $k_1<k_i$ for all $1<i\le r$, so $v_1$ is the vertex whose lone merge occurs farthest to the left among all the $v_i$. By induction, $v_2,\dots,v_r$ span a simplex, $\sigma$. Thanks to the action of $B_n(R)$, without loss of generality $v_1$ is the vertex $[I_n,\OneEdge_k]$, where we have set $k\defeq k_1$ for brevity.
 
 Represent $\sigma=[A,\Gamma]$ with $A$ reduced relative to $\Gamma$. Since $k$ is less than the index of any edge of $\Gamma$, we know that the $k$th column of $A-I_n$ is all zeros. Since $v_1$ shares an edge with every vertex of $\sigma$, we know that in fact $k$ is even less than the index of any edge of $\Gamma$, minus one. Hence the $(k+1)$st column of $A-I_n$ is similarly all zeros. Our goal is to show that $A\in \image \clone_k$, since then $\sigma$ and $v_1$ will share a simplex. Thanks to the setup, it suffices to show that the $k$th row of $A-I_n$ is all zeros. Since $A$ is reduced relative $\Gamma$, non-zero entries of $A-I_n$ may only possibly occur in columns indexed by $k_2,\dots,k_r$.
 
 For each vertex $[A,\OneEdge_{k_i}]$, $2\le i\le r$, of $\sigma$, let $A_i$ be such that $[A_i,\OneEdge_{k_i}]=[A,\OneEdge_{k_i}]$ and $A_i$ is reduced relative $\OneEdge_{k_i}$. Let $\ell\in \{k_2,\dots,k_r\}$. Observe that $A_\ell$ is obtained from $A$ by right multiplication by an element $D$ of $\image(\clone_\ell)$. For $1\le i\le n$ denote by $M_{(i,*)}$ the $i$th row of an $n$-by-$n$ matrix $M$, and by $M_{(*,i)}$ the $i$th column. When we multiply by $D$ to get $AD=A_\ell$, the $(k,\ell)$-entry of $A_\ell$ is $A_{(k,*)}D_{(*,\ell)}$ and the $(k,\ell+1)$-entry is $A_{(k,*)}D_{(*,\ell+1)}$. Since $A_\ell$ is reduced relative $\OneEdge_\ell$, we know that its $(k,\ell+1)$-entry must be $0$. Also since $D\in\image(\clone_{\OneEdge_\ell})$, we have $D_{(*,\ell)}=D_{(*,\ell+1)}+d(e_\ell-e_{\ell+1})$ for some $d\in R^\times$. Let $a$ denote the $(k,\ell)$-entry of $A$, and note that the $(k,\ell+1)$-entry of $A$ is $0$. We calculate that the $(k,\ell)$-entry of $A_\ell$ is
 \begin{align*}
  A_{(k,*)}D_{(*,\ell)} &= A_{(k,*)}(D_{(*,\ell+1)}+d(e_\ell-e_{\ell+1})) \\
  &= A_{(k,*)}(d(e_\ell-e_{\ell+1})) \\
  &= da \text{.}
 \end{align*}
 Since $d$ is a unit, this shows that the $(k,\ell)$-entry of $A_\ell$ is zero if and only if the $(k,\ell)$-entry of $A$ is zero. By the same argument just given, this statement remains true with $A_\ell$ replaced by $A_\ell D$ for any $D\in\image(\clone_\ell)$. But by assumption $v_1$ shares an edge with $v_\ell$, and so some such $A_\ell D$ must have $(k,\ell)$-entry zero. We conclude that $A$ has $(k,\ell)$-entry zero. Since $\ell$ was arbitrary, the $k$th row of $A-I_n$ is all zeros and so $v_1$ and $\sigma$ share a simplex.
\end{proof}

Let $\Delta_0\defeq \Delta\setminus\{e_1\cup e_2\}$, and consider $\dlkmodel{B_*(R);\Delta_0}{n}$ as a subcomplex of the complex $\dlkmodel{B_*(R);\Delta}{n}$. For a vertex $[A,\OneEdge_k] \in  \dlkmodel{B_*(R);\Delta_0}{n}$ we write $\lk_0([A,\OneEdge_k])$ for the link in $\dlkmodel{B_*(R);\Delta_0}{n}$, to differentiate from the link in $\dlkmodel{B_*(R);\Delta}{n}$ which is just denoted $\lk([A,\OneEdge_k])$. To prove Proposition~\ref{prop:borel_thomp_fin_props_strong} we follow the strategy outlined by Proposition~\ref{prop:putman_flow}:  we want to show that $\dlkmodel{B_*(R);\Delta_0}{n}$ is $(\eta(e(\Delta))-2)$-connected, that $(\dlkmodel{B_*(R);\Delta}{n},\dlkmodel{B_*(R);\Delta_0}{n})$ is $(\eta(e(\Delta))-1)$-connected and that there is a vertex satisfying condition \eqref{item:exchange1}. That vertex is $w \defeq [I_n,\OneEdge_1]$ in our case. The following statements (up to the proof of Proposition~\ref{prop:borel_thomp_fin_props_strong}) are part of an induction, so we assume that Proposition~\ref{prop:borel_thomp_fin_props_strong} has been proven for graphs $\Delta'$ with $e(\Delta') < e(\Delta)$ and intend to prove it for $\Delta$.

\begin{lemma}[Links are lower rank complexes]\label{lem:links_are_cpxes}
 Let $\sigma$ be a simplex of dimension $d \ge 0$ in $\dlkmodel{B_*(R);\Delta}{n}$. Then $\lk(\sigma)$ is isomorphic to a complex of the form $\dlkmodel{B_{*}(R);\Delta'}{n-(d+1)}$ where $\Delta'$ is a spanning subgraph of $L_{n-(d+1)}$ with at least $e(\Delta)-3d-3$ edges. In particular, it is $(\eta(e(\Delta)-3d-3)-1)$-connected by induction.
\end{lemma}

\begin{proof}
 The simplex $\sigma$ is of the form $[g,\Gamma]$ with $g \in B_n(R)$ and $\Gamma \subseteq \Delta$. If it has dimension $d$ then $\Gamma$ has $d+1$ edges, say $e_{i_1},\ldots,e_{i_{d+1}}$. Using the left action of $B_n(R)$ we may assume that $g = 1$. Then $\lk(\sigma)$ is $\dlkmodel{(B_{*}(R))\clone_\Gamma;\Delta^\sharp}{n}$, where $\Delta^\sharp$ is $\Delta$ with the edges $e_{i_j-1}, e_{i_j}, e_{i_j+1}$ removed for each $1 \le j \le d+1$. In particular $\Delta^\sharp$ has at least $e(\Delta)-3d-3$ edges. Now consider the map $b_\Gamma \colon L_n\to L_{n-(d+1)}$ given by blowing down the edges of $\Gamma$. The image of $\Delta^\sharp$ under $b_\Gamma$ is what we will call $\Delta'$. Note that $\Delta'$ still has at least $e(\Delta)-3d-3$ edges. Since $\clone_\Gamma$ is injective, we may now apply $\clone_\Gamma^{-1}$ paired with $b_\Gamma$ to $\dlkmodel{(B_{*}(R))\clone_\Gamma;\Delta^\sharp}{n}$ and get an isomorphism to $\dlkmodel{B_{*}(R);\Delta'}{n-(d+1)}$.
\end{proof}

\begin{lemma}\label{lem:borel_relative_connectivity}
 The pair $(\dlkmodel{B_*(R);\Delta}{n},\dlkmodel{B_*(R);\Delta_0}{n})$ is $(\eta(e(\Delta))-1)$-connected.
\end{lemma}

\begin{proof}
 Note that for any vertex of $\dlkmodel{B_*(R);\Delta}{n} \setminus \dlkmodel{B_*(R);\Delta_0}{n}$, the entire link of the vertex lies in $\dlkmodel{B_*(R);\Delta_0}{n}$. Hence the function sending vertices of the former to $1$ and vertices of the latter to $0$ yields a Morse function in the sense of Section~\ref{sec:finiteness_props}, and to prove the statement we need only show that links of vertices in $\dlkmodel{B_*(R);\Delta}{n} \setminus \dlkmodel{B_*(R);\Delta_0}{n}$ are $(\eta(e(\Delta))-2)$-connected. By Lemma~\ref{lem:links_are_cpxes}, each descending link is isomorphic to a complex of the form $\dlkmodel{B_{*}(R);\Delta'}{n-1}$ for $\Delta'$ a graph with at least $e(\Delta)-3$ edges. By induction, these are $(\eta(e(\Delta))-2)$-connected as desired.
\end{proof}

In addition to the subcomplex $\dlkmodel{B_*(R);\Delta_0}{n}$ we will now need to consider $\dlkmodel{B_*(R);\Delta_1}{n}$ where $\Delta_1\defeq \Delta\setminus\{e_1\}$. We will write links in this complex using the symbol $\lk_1$.

\begin{lemma}[Shared links]\label{lem:mut_lks_are_lks}
 Let $k>2$ and let $A$ be reduced relative $\OneEdge_k$. Let $A'$ be obtained from $A$ by setting the $(1,k)$-entry to $0$. Then $\lk_1([A,\OneEdge_k]) \subseteq \lk_1([A',\OneEdge_k])$ and $[A',\OneEdge_k] \in \lk w$.
\end{lemma}

\begin{proof}
 As a first observation, note that since $A$ is reduced relative $\OneEdge_k$ and $k>2$, the $(1,1)$-entry and $(2,2)$-entry of $A$ are both $1$, and the entries of the top row of $A$ past the first entry is all $0$'s except possibly in the $k$th column. Let $-\lambda$ be the $(1,k)$-entry of $A$, and note that $A' = A E_{1k}(\lambda)$. The first row of $A'$ is now $(1,0,\ldots,0)$ and the $(2,2)$-entry is $1$, which tells us that $A' \in (B_{n-1}(R))\clone_1$. Hence $[A',\OneEdge_k] \in \lk_0 w$.

 To see that $\lk_1 ([A,\OneEdge_k]) \subseteq \lk_1([A',\OneEdge_k])$ we first multiply by $A^{-1}$ from the left and are reduced to showing that $\lk_1 ([I_n,\OneEdge_k]) \subseteq \lk_1([E_{1k}(\lambda),\OneEdge_k])$. An arbitrary simplex of $\lk_1([I_n,\OneEdge_k])$ is of the form $[B,\Gamma]$, with $B\in\image(\clone_k)$ and $\Gamma$ not containing any of $e_1$, $e_{k-1}$, $e_k$, or $e_{k+1}$. Note that the $k$th row of $B$ is zero off the diagonal. By Lemma~\ref{lem:borel_reduction} there is a $B' \in B \image(\clone_\Gamma)$ that is reduced relative $\Gamma$ and has $k$th row zero off the diagonal. We have $[B',\Gamma]=[B,\Gamma]$. Since $e_1 \not\in \Gamma$ and $B'$ is reduced relative $\Gamma$, the first column of $B'$ is $e_1$.
 
 We now claim that $B'$ commutes with $E_{1k}(\lambda)$. Indeed, left multiplication by $E_{1k}(\lambda)$ is the row operation $r_1\mapsto r_1+\lambda r_k$, and right multiplication by $E_{1k}(\lambda)$ is the column operation $c_k\mapsto c_k+\lambda c_1$. For our $B'$, both of these operations change the $(1,k)$-entry by adding $\lambda$ to it, and change no other entries. This proves the claim.
 
 Now we have
 \[
 [B,\Gamma] =[B',\Gamma]\\
 =[E_{1k}(\lambda)B'E_{1k}(-\lambda),\Gamma]\\
 =[E_{1k}(\lambda)B',\Gamma]\\
 =[E_{1k}(\lambda)B,\Gamma]\text{.}
 \]
 The second to last step works since $E_{1k}(-\lambda)\in\image(\clone_\Gamma)$ by virtue of $e_{k-1},e_k \not \in \Gamma$. This shows that our arbitrary simplex of $\lk_1([I_n,\OneEdge_k])$ is also in $\lk_1([E_{1k}(\lambda),\OneEdge_k])$.
\end{proof}

\begin{proof}[Proof of Proposition~\ref{prop:borel_thomp_fin_props_strong}]
 We want to apply Proposition~\ref{prop:putman_flow}. The complexes are $X = \dlkmodel{B_*(R);\Delta}{n}$, $X_1 = \dlkmodel{B_*(R);\Delta_1}{n}$ and $X_0 = \dlkmodel{B_*(R);\Delta_0}{n}$, and $k = \eta(e(\Delta))-1$. We check the assumptions. The pair $(\dlkmodel{B_*(R);\Delta}{n},\dlkmodel{B_*(R);\Delta_0}{n})$ is $k$-connected by Lemma~\ref{lem:borel_relative_connectivity}. Since $X$ is a flag complex (Lemma~\ref{lem:mtx_cpx_flag}), it is conical at every vertex, in particular at our vertex $w=[I_n,\OneEdge_1]$. The complex $\dlkmodel{B_*(R);\Delta_0}{n}$ is $(\eta(e(\Delta_0))-1)$-connected by induction. This is sufficient because $\eta(e(\Delta_0)) - 1 \ge \eta(e(\Delta) - 2) - 1 \ge \eta(e(\Delta)) - 2 = k -1$. The link of a $d$-simplex is $(\eta(e(\Delta)-3d-3)-1)$-connected by Lemma~\ref{lem:links_are_cpxes}. This is sufficient because $\eta(e(\Delta)-3d-3) - 1\ge \eta(e(\Delta)) - d - 2 = k - d - 1$. Finally condition \eqref{item:exchange1} is satisfied by Lemma~\ref{lem:mut_lks_are_lks} where $\lk_1([A,\OneEdge_k])$ is at least $(\eta(e(\Delta)-4)-1)$-connected and $\eta(e(\Delta)-4) - 1 = \eta(e(\Delta)) - 2 = k - 1$ as desired.
\end{proof}

Shifting focus to the Abels groups, thanks to the flexibility of Lemma~\ref{lem:borel_reduction}, the above arguments also show high connectivity of $\dlkmodel{\abels_*(\Z[1/p])}{n}$, and using Proposition~\ref{prop:generic_finiteness} and Theorem~\ref{thm:abels} we conclude:

\begin{theorem}\label{thrm:thomp_abels_Finfty}
 $\Thomp{\abels_*(\Z[1/p])}$ is of type~$\F_\infty$.
\end{theorem}

This, despite none of the $\abels_n(\Z[1/p])$ individually being $\F_\infty$.

\medskip

The remaining question is whether $\phi(\Thomp{B_*(R)})=\liminf_n(\phi(B_n(R)))$, that is whether negative finiteness properties of the $B_n(R)$ can impose negative finiteness properties on $\Thomp{B_*(R)}$. For $R$ the ring of $S$-integers of a global function field, we will answer this question affirmatively in the next section.

Before we do that, we need to treat one more relative of the family $B_n(R)$: Let $B_n^2$ be the normal subgroup of $B_n$ consisting of matrices that differ from the identity only from the second off-diagonal on (the second term of the lower central series), and let $\barB_n \defeq B_n/B_n^2$ be the quotient group. Set $\nu(n) = \lfloor \frac{n-2}{3} \rfloor$. One could check that the above proof for $B_*$ goes through for the family $\barB_*$ as well, but instead we will prove directly:

\begin{proposition}
\label{prop:barb_thomp_pos_fin_props}
The descending link $\dlkmodel{\barB_*(R)}{n}$ is $(\nu(n)-1)$-connected. Thus $\displaystyle \phi(\Thomp{\barB_*(R)}) \ge \liminf_n \phi(\barB_*(R))$.
\end{proposition}

\begin{proof}
Using reductions as in Lemma~\ref{lem:borel_reduction} one can see the following: every simplex in $\dlkmodel{\barB_*(R)}{n}$ has a representative $[A,\Gamma]$ where the matrix $A$ has a diagonal block of the form
\[
\left(
\begin{array}{cc}
* & *\\
& 1
\end{array}
\right)
\]
above every edge of $\Gamma$ and otherwise equals the identity matrix (here the representative is modulo dangling as well as modulo $B_n^2(R)$). What makes this case particularly easy is that this representative is unique. That is, we may think of $\dlkmodel{\barB_*(R)}{n}$ as consisting of pairs $(A,\Gamma)$ where $A$ is as above and the face relation is given by removing an edge of $\Gamma$ and turning the diagonal block above it into an identity block.

Let $sL_n$ denote the linear graph with vertices $\{1,\ldots,n\}$ and with every pair of adjacent vertices $i$ and $i+1$ connected by $s$ distinct edges. By what we just said, $\dlkmodel{\barB_*(R)}{n}$ is isomorphic to the matching complex $\mathcal{M}(sL_n)$ where $s = \abs{R^* \times R}$. There is an obvious map $\mathcal{M}(sL_n) \to \mathcal{M}(L_n)$. The fiber of this map over a $k$-simplex is a $(k+1)$-fold join of $s$-element sets, thus $k$-spherical. Moreover $\mathcal{M}(L_n)$ is $(\nu(n)-1)$-connected by \cite[Proposition~11.16]{kozlov08} (and links in $\mathcal{M}(L_n)$ are highly connected as well, being joins of lower-rank copies of the complex). Thus we can apply \cite[Theorem~9.1]{quillen78} to conclude that $\mathcal{M}(sL_n)$ is $(\nu(n)-1)$-connected.
\end{proof}

As a remark, this simple approach for $\barB_*(R)$ would not have worked for $B_*(R)$, since the analogous fibers are not joins of vertex fibers.

\subsection{Negative finiteness properties}\label{sec:mtx_neg_fin_props}

In the last section we saw that for any $R$, the generalized Thompson group $\Thomp{B_*(R)}$ is of type~$\F_n$ if all but finitely many $B_k(R)$ are. In this section we prove the converse in the case we are most interested in (cf.~\cite{bux04}): Let $\GlobalField$ be a global function field and let $S$ be a non-empty set of places. Denote by $\calO_S$ the ring of $S$-integers in $\GlobalField$.

\begin{theorem}\label{thrm:borel_thomp_neg_fin_props}
 The group $\Thomp{B_*(\calO_S)}$ is not of type~$\FP_{\abs{S}}$.
\end{theorem}

\begin{remark}
 Unlike the positive statement from the previous section, for the proof of Theorem~\ref{thrm:borel_thomp_full_fin_props} we cannot just use the \emph{results} from \cite{bux04} but have to use parts of the \emph{proof}. By using the more substantial parts of the proof, it is quite possible that the setup of this section could be used to prove the positive finiteness properties as well, but we will not do so.
\end{remark}

We will actually prove first that $\Thomp{\barB_*(\calO_S)}$ is not of type~$\FP_{\abs{S}}$. We then use the result from Section~\ref{sec:relative_brown} to deduce Theorem~\ref{thrm:borel_thomp_full_fin_props}. Instead of the Stein--Farley complex on which $\Thomp{\barB_*(\calO_S)}$ acts with stabilizers isomorphic to the $\barB_*(\calO_S)$ we will construct a new space $Y$ for which the stabilizers are themselves generalized Thompson groups of smaller cloning systems. In particular the stabilizers on $Y$ will have good finiteness properties and the negative finiteness properties of the $\barB_*(\calO_S)$ are reflected in bad connectivity properties.

\medskip

For any place $s \in S$ denote by $\GlobalField_s$ the completion of $\GlobalField$ at $s$, and by $\calO_s$ the ring of integers of $\GlobalField_s$. As before we let $B_n$ be the linear algebraic group of invertible upper triangular $n$-by-$n$ matrices, let $B_n^2$ be the normal subgroup of matrices that differ from the identity only from the second off-diagonal on, and let $\barB_n \defeq B_n/B_n^2$ be the quotient group. Let $Z_n \le B_n$ be the group of homotheties, i.e., scalar multiples of the identity matrix, and let $\PB_2 = B_2/Z_2$.

All of this is relevant to us for the following reason: For any of the local fields $\GlobalField_s$ the group $\PGL_2(\GlobalField_s)$ admits a Bruhat--Tits tree $V_s$ on which it acts properly. Since $\calO_S$ is discrete as a subset of $\prod_{s \in S} \GlobalField_s$ when embedded diagonally, we get a properly discontinuous action of $\PGL_2(\calO_S)$ on
$$V \defeq \prod_{s \in S} V_s \text{.}$$
Our goal is to use this action to understand finiteness properties of $\Thomp{\barB_*(\calO_S)}$. Note that the group $\PGL_2(\calO_s)$ is the stabilizer of a vertex in $V_s$, call it $z_s$. Define $z \defeq (z_s)_{s \in S}$, so $z$ is a vertex in $V$.

Denote the quotient morphism from $B_n$ to $\barB_n$ by $\bar{~} \colon B_n \to \barB_n \text{, } g \mapsto \bar{g}$. For $1 \le i \le n-1$ let $\pi_i$ denote the homomorphism $\barB_n \to \PB_2, [A] \mapsto [A_i]$ where $A_i$ is the $i$th diagonal $2$-by-$2$ block of $A$. For brevity we denote the composition $\pi_i \circ \bar{~}$ by $\bar{\pi}_i$. Now for any $i$, $1 \le i \le n-1$ consider the composition
\[
\alpha_i \colon \barB_n(\calO_S) \to \prod_{s \in S} \bar{B}_n(\GlobalField_s) \stackrel{\prod \pi_i(\GlobalField_s)}{\to} \prod_{s \in S} \PB_2(\GlobalField_s)
\]
where the first morphism is induced by the diagonal inclusion $\calO_S \to \prod_{s \in S} \GlobalField_s$.
Define
$$K_n \defeq \bigcap_{1 \le i \le n-1} \alpha_i^{-1}\left(\prod_{s \in S} \PB_2(\calO_s)\right)\text{.}$$

\medskip

\begin{lemma}\label{lem:good_stab}
 The group $K_n$ is of type~$\F_\infty$.
\end{lemma}

\begin{remark}
 The importance of the Lemma lies in the fact that the groups $K_n$ will appear in stabilizers of an action of $\Thomp{\barB_*(\calO_S)}$. It is worth noting that the statement does not remain true if $\barB_n(\calO_S)$ is replaced by $B_n(\calO_S)$ so that the strategy does not immediately carry over to $\Thomp{B_*(\calO_S)}$. Instead we will have to apply Theorem~\ref{thrm:relative_brown} in the end to conclude that $\Thomp{B_*(\calO_S)}$ is not of type~$\FP_{\abs{S}}$.
\end{remark}

\begin{proof}[Proof of Lemma~\ref{lem:good_stab}]
 We first study the map $\pi_i(\GlobalField_s) \colon \barB_n(\GlobalField_s) \to \PB_2(\GlobalField_s)$. The kernel $N_i(\GlobalField_s)$ is determined by the conditions that the $(i,i+1)$-entry of a matrix is $0$ and that the $(i,i)$ and the $(i+1,i+1)$-entry coincide. The inverse image of $\PB_2(\calO_s)$ under $\pi_i(\GlobalField_s)$ is thus generated by $N_i(\GlobalField_s)$ and a copy of $B_2(\calO_s)$. Intersecting over all $i$, we find that $\bigcap_{i} \pi_i(\GlobalField_s)^{-1}(\PB_2(\calO_s)) = Z_n(\GlobalField_s) \barB_n(\calO_s)$.

 The intersection of this group with $\barB_n(\calO_S)$ is $Z_n(\calO_S)\barB_n(\calO_{S \setminus \{s\}})$. Intersecting over all $s\in S$ we find that $K_n = Z_n(\calO_S)\barB_n(\CoefficientField)$ where $\CoefficientField \defeq \calO_\emptyset$ is the coefficient field of $k$, which is finite. In particular $\barB_n(\CoefficientField)$ is finite and of type~$\F_\infty$. By the Dirichlet Unit Theorem, as extended to $S$-units by Hasse and Chevalley, $Z_n(\calO_S)$ is finitely generated abelian and so of type~$\F_\infty$. Since $K_n$ is a central product of these groups, this finishes the proof.
\end{proof}

Now consider the action of $\barB_n(\calO_S)$ on $V^{n-1}$ via the maps $\alpha_i$.

\begin{corollary}\label{cor:all_stabs_F_infty}
 The stabilizers for the action of $\barB_n(\calO_S)$ on $V^{n-1}$ are all of type~$\F_\infty$.
\end{corollary}

\begin{proof}
 The group $K_n$ is precisely the stabilizer of $(z,\dots,z)\in V^{n-1}$. Since the product of trees $V^{n-1}$ is locally finite and the action is proper, every stabilizer is commensurable to $K_n$ and therefore of type~$\F_\infty$ as well.
\end{proof}

We are about to define a space $Y$ for $\Thomp{\barB_*(\calO_S)}$ to act on. The advantage over the Stein--Farley complex will be that the stabilizers have better finiteness properties.  Let $D = \Z[1/2] \cap (0,1)$ be the set of dyadic points in $(0,1)$. Let $V^D$ be the set of all maps $D \to V$. We will usually regard these elements as tuples; that is, we write $x_q$ for the value of $x \in V^D$ at $q \in D$ and sometimes we write $x$ as $(x_q)_{q \in D}$. Let
\[
Y \defeq V^{(D)}
\]
be the subset consisting of those maps that evaluate to $z$ at all but finitely many points. An alternative description is as a direct limit $\varinjlim_{I \subseteq D \text{ finite}} V^I$. Note that this set is naturally equipped with a (unique) topology: the topology induced from the product topology and the CW topology coincide.

Note that Thompson's group $F$ acts on $D$ from the right, via $q.f = f^{-1}(q)$ for $f \in F$ and $q \in D$. To describe this action in terms of paired tree diagrams, note that every point in $D$ corresponds to a caret in the leafless rooted binary tree. Thus every finite rooted binary tree $\tree$ determines a finite subset $D(\tree)$ of $D$, namely that consisting of points that correspond to its carets. An element $[\tree,\alttree]$ of $F$ takes $D(\tree)$ to $D(\alttree)$ (preserving the order) and is linear between these break points.

As a consequence, $F$ acts from the left on the set $V^D$ via $(f.x)_q = x_{q.f}$ where $x \in V^D$, $q \in D$ and $f \in F$. Clearly this induces an action of $F$ on $Y$. Explicitly, the action of $F$ on $Y$ satisfies
\[
([\tree,\alttree].x)_{t_i} = x_{u_i}
\]
where $D(\tree) = \{t_1 < \ldots < t_{n-1}\}$ and $D(\alttree) = \{u_1 < \ldots < u_{n-1}\}$. Away from the break points, the values are interpolated linearly: $[\tree,\alttree].x_{s t_i + (1-s) t_{i+1}} = x_{s u_i + (1-s) u_{i+1}}$.

There is also an action of $\Thkern{\barB_*(\calO_S)}$ on $Y$ which is given as follows: if $\tree$ is a finite rooted binary tree and $D(\tree) = \{q_1 < \ldots < q_{n-1}\}$ then
\[
([\tree,g,\tree].x)_q = \left\{
\begin{array}{ll}
\alpha_i(g).x_{q_i} & \text{if }q = q_i\\
x_q & \text{else.}
\end{array}
\right.
\]
This is just the action obtained by taking the direct limit over the actions of $\barB_T(\calO_S)$ on $V^{D(\tree)}$.

These actions are compatible and so give an action of $\Thomp{\barB_*(\calO_S)}$ on $Y$, which is given by
\[
([\tree,g,\alttree].x)_{t_i} = \alpha_i(g).x_{u_i}
\]
and $([\tree,g,\alttree].x)_t = ([\tree,\alttree].x)_t$ for $t \not\in D(\tree)$; see Figure~\ref{fig:action_on_Y}.

\newsavebox{\barb}
\newsavebox{\barc}
\savebox{\barb}{$\bar{b} = \left(\begin{smallmatrix}1 & 2\\&4\end{smallmatrix}\right)b$}
\savebox{\barc}{$\bar{c} = \left(\begin{smallmatrix}4 & 5\\&6\end{smallmatrix}\right)c$}

\begin{figure}[htb]
\begin{tikzpicture}[scale=0.3,line width=0.8pt]
  \begin{scope}
  \draw
   (-2,-2) -- (0,0) -- (2,-2)   (0,-2) -- (-1,-1);
  \filldraw
   (-2,-2) circle (1.5pt)   (2,-2) circle (1.5pt)  (0,-2) circle (1.5pt);
  \end{scope}
  \begin{scope}[yshift=-3.5cm]
   \node at (0,-1){$\begin{pmatrix}
                    1&2& \\ &4&5\\ & &6
                   \end{pmatrix}$};
  \end{scope}
  \begin{scope}[yshift=-9cm,yscale=-1]
  \draw
   (-2,-2) -- (0,0) -- (2,-2)   (0,-2) -- (1,-1);
  \filldraw
   (-2,-2) circle (1.5pt)   (2,-2) circle (1.5pt)   (0,-2) circle (1.5pt);
  \end{scope}

  \begin{scope}[yshift=-9cm]
  \draw
   (-3,-3) -- (0,0) -- (3,-3)   (1,-3) -- (2,-2)  (-1,-3) -- (-2,-2);
  \filldraw
   (-3,-3) circle (1.5pt)  (-1,-3) circle (1.5pt)  (1,-3) circle (1.5pt)  (3,-3) circle (1.5pt);
  \end{scope}
  \begin{scope}[yshift=-12.5cm]
   \node at (0,-.5){$a \quad b \quad c$};
  \end{scope}
  \node at (4,-6.5) {$=$};
  \begin{scope}[xshift=9cm,yshift=-3cm]
  \draw
   (-3.5,-3.5) -- (0,0) -- (3.5,-3.5)   (-.5,-3.5) -- (-2,-2)  (-2.5,-3.5) -- (-3,-3);
   \node at (0,-4.5){$a~\, \bar{b}\quad\quad\bar{c}\quad\quad \phantom{ }$};
  \end{scope}
\end{tikzpicture}
\hspace{1cm}
\begin{tikzpicture}[scale=2,line width=0.8pt]
 \draw
  (0,0) -- (1,0) -- (1,-1) -- (0,-1) -- cycle
  (.5,0) -- (.75,-1)
  (.25,0) -- (.5,-1);
  \fill[color=white,xshift=.5cm,yshift=-.5cm] {(-.25,-.2) -- (.25,-.2) -- (.25,.2) -- (-.25,.2) -- cycle};
  \node at (.5,-.5) {$\left(\begin{smallmatrix}1&2&\\&4&5\\&&6\end{smallmatrix}\right)$};
  \draw[yshift=-1.1cm]
  (0,0) -- (1,0) (0,.05) -- (0,-.05) (.25,.05) -- (.25,-.05)  (.5,.05) -- (.5,-.05)  (.75,.05) -- (.75,-.05)  (1,.05) -- (1,-.05);
  \node at (.5,-1.3) {$a \quad b \quad c$};
  \node at (1.2,-.5) {$=$};
  \draw[xshift=1.5cm, yshift= -.5cm]
  (0,0) -- (1,0) (0,.05) -- (0,-.05)  (.125,.05) -- (.125,-.05)  (.25,.05) -- (.25,-.05)  (.5,.05) -- (.5,-.05)  (1,.05) -- (1,-.05);
  \node at (2,-.7) {$a~\! \bar{b}\quad\bar{c}\quad\quad \phantom{ }$};
\end{tikzpicture}
\caption{Two points of view on the action of $\mathscr{T}(\barB_*(\calO_S))$ on $Y$. On the left the action is described in terms of tree diagrams, on the right in terms of piecewise linear homeomorphisms. In both pictures \usebox{\barb} and \usebox{\barc}. All unspecified values are $z$.}
\label{fig:action_on_Y}
\end{figure}
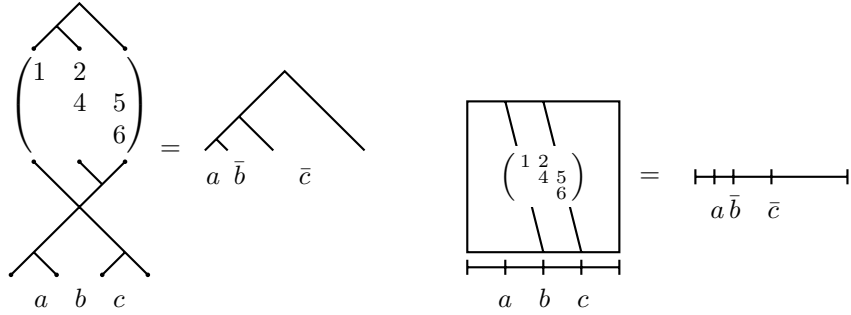

Next we want to understand stabilizers of this action. First observe that the action has a nontrivial kernel, namely the center of $\Thomp{\barB_*(\calO_S)}$, which is isomorphic to $\calO_S^\times$.

\begin{observation}\label{obs:center}
 Let $(G_*,(\clone_k)_k)$ be a cloning system and let $H$ be a group. Define a new cloning system $(H \times G_*,(\hat{\clone}_k)_k)$ by taking $\hat{\clone}_k \defeq \id \times \clone_k$. Then $\Thomp{H \times G_*} = H \times \Thomp{G_*}$.
\end{observation}

\begin{proof}
 The isomorphism is given by $(h,[\tree,g,\alttree]) \mapsto [\tree,hg,\alttree]$.
\end{proof}

We now turn to one particular stabilizer.

\begin{observation}\label{obs:stabilizer}
 The cloning system on $\barB_*(\calO_S)$ induces a cloning system on $K_*$. The stabilizer in $\Thomp{\barB_*(\calO_S)}$ of the point $(z)_q$ is $\Thomp{K_*}$.
\end{observation}

\begin{proof}
 For the first statement it suffices to show that $(Z_n(\calO_S))\clone_k \subseteq Z_{n+1}(\calO_S)$ and that $(\barB_n(\CoefficientField))\clone_k \subseteq \barB_{n+1}(\CoefficientField)$ which is easy to see. The second statement is clear.
\end{proof}

\begin{corollary}\label{cor:Thomp(K_*)F_infty}
 The group $\Thomp{K_*}$ is of type~$\F_\infty$.
\end{corollary}

\begin{proof}
 By Observation~\ref{obs:center} $\Thomp{K_*}$ is isomorphic to a central product $\calO_S^\times\Thomp{\barB_*(\CoefficientField)}$. The second factor is of type~$\F_\infty$ by Proposition~\ref{prop:barb_thomp_pos_fin_props}.
\end{proof}

We now turn to general stabilizers. For a point $x \in V$ write $[x]$ for its $\PB_2(\calO_S)$-orbit. We call a point $(x_q)_q \in Y$ \emph{reduced} if $x_q = z$ whenever $[x_q] = [z]$.

\begin{lemma}\label{lem:reduced_points}
 Every point of $Y$ has a reduced point in its $\Thomp{\barB_*(\calO_S)}$-orbit.
\end{lemma}

\begin{proof}
 If $(x_q)_q \in Y$ is arbitrary, let $\tree$ be a tree such that $D(\tree)$ contains all of the finitely many indices $q \in D$ for which $x_q \ne z$. Write $D(\tree) = \{q_1,\ldots,q_{n-1}\}$ where the indices are in increasing order. For each $i$ pick $g_i \in \PB_2(\calO_S)$ such that $g_i.x_{q_i} = z$ whenever possible (i.e., when $[x_{q_i}] = [z]$) and arbitrarily otherwise. Take $g \in \barB_n(\calO_S)$ such that $\alpha_i(g) = g_i$ for all $i$. Then $[\tree,g,\tree].(x_q)_q$ is reduced.
\end{proof}

\begin{lemma}\label{lem:stabilizer_of_reduced_point}
 The stabilizer in $\Thomp{\barB_*(\calO_S)}$ of any reduced point is of type~$\F_\infty$.
\end{lemma}

\begin{proof}
 Let $(x_q)_{q \in D}$ be a reduced point and let $I = \{q \in D \mid x_q \ne z\} = \{q_1,\ldots,q_{n-1}\}$. Let $H$ be the stabilizer of $(x_q)_{q}$ in $\Thomp{\barB_*(\calO_S)}$ and let $K$ be the kernel of the action of $H$. Since $[x_q] \ne [z]$ for $q \in I$, we see that the stabilizer has to fix $I$ (when acting on $D$ via the canonical homomorphism to $F$). Thus the action of the stabilizer $H$ on $Y=V^{(D)}$ decomposes into an action on $V^I$ and on $V^{(D \setminus I)}$.

 Modulo $K$ we find that $H$ is a direct product of the pointwise stabilizer (in $H$) of $V^I$ and the pointwise stabilizer of $V^{(D \setminus I)}$. The action of the pointwise stabilizer of $V^{(D \setminus I)}$ in $\Thomp{\barB_*(\calO_S)}$ is isomorphic to $\barB_n(\calO_S)$ acting on $V^I$. Thus its intersection with $H$ is isomorphic to a point stabilizer in $\barB_n(\calO_S)$, and hence is of type~$\F_\infty$ by Corollary~\ref{cor:all_stabs_F_infty}.

 The pointwise stabilizer of $V^I$ decomposes further. Let $D_1 \defeq D \cap (0,q_1)$, $D_2 \defeq D \cap (q_1,q_2)$, \ldots, $D_n \defeq D \cap (q_{n-1},1)$. The pointwise stabilizer of $V^{(D \setminus D_j)}$ is itself isomorphic to a copy of $\Thomp{\barB_*(\calO_S)}$ and therefore the stabilizer of $(z)_{q \in D_J}$ in this stabilizer is isomorphic to a copy of $\Thomp{K_*}$, which is of type~$\F_\infty$ by Corollary~\ref{cor:Thomp(K_*)F_infty}.

 Putting everything together we find that $H/K$ is a product of groups of type~$\F_\infty$, and $K$ is of type~$\F_\infty$ as well, so $H$ is of type~$\F_\infty$.
\end{proof}

In summary we have:

\begin{proposition}\label{prop:Y_stabilizers}
 The group $\Thomp{\barB_*(\calO_S)}$ acts on $Y$ with stabilizers of type~$\F_\infty$.
\end{proposition}

\begin{proof}
 Every point is in the orbit of a reduced point by Lemma~\ref{lem:reduced_points} so every stabilizer is isomorphic to that of a reduced point. Those are of type~$\F_\infty$ by Lemma~\ref{lem:stabilizer_of_reduced_point}.
\end{proof}

It remains to provide a cocompact filtration and determine its essential connectivity. For this purpose we will use the key result from \cite{bux04} used to show that $\PB_2(\calO_S)$ is not of type~$\FP_{\abs{S}}$:

\begin{theorem}[\cite{bux04}]
 There is a filtration $(V_r)_{r \in \N}$ of $V$ that is $\PB_2(\calO_S)$-invariant and -cocompact and is essentially $(\abs{S}-2)$-connected but not essentially $(\abs{S}-1)$-acyclic.
\end{theorem}

In fact, by Brown's criterion \emph{any} cocompact filtration of $V$ has that property just because $B_2(\calO_S)$ is of type~$\F_{\abs{S}-1}$ but not of type~$\FP_{\abs{S}}$. We use this filtration to construct a cocompact filtration of $Y$ as follows. For $r \in \N$ let $Y^{(r)}$ be the set of all points $(x_q)_{q \in D}$ for which $\{q \in D \mid [x_q] \ne [z]\}$ has at most $r$ elements. Note that $Y^{(r)}$ is $\Thomp{\barB_*(\calO_S)}$-invariant. The filtration we want to consider is
\[
Y_r \defeq Y^{(r)} \cap V_r^{(D)}\text{.}
\]
The last piece that is missing to conclude that $\Thomp{\barB_*(\calO_S)}$ is not of type~$\FP_{\abs{S}}$ is the following:

\begin{proposition}\label{prop:Y_essential_connectivity}
 The filtration $(Y_r)_{r \in \N}$ is $\Thomp{\barB_*(\calO_S)}$-invariant and -cocompact. It is not essentially $(\abs{S}-1)$-acyclic.
\end{proposition}

Before we can prove the second part, we have to state a technical lemma which says that taking products does not help to kill cycles:

\begin{lemma}\label{lem:non-trivial_maps}
 Let $(X_1,A_1)$ and $(X_2,A_2)$ be pairs of CW complexes and assume that the map $\tilde{H}_n(A_1 \to X_1)$ is non-trivial and that $A_2$ is non-empty. Then the map $\tilde{H}_n(A_1 \times A_2 \to X_1 \times X_2)$ is non-trivial as well.
\end{lemma}

\begin{proof}
 The case $n = 0$ is clear so assume $n > 0$ from now on.

 Let $c$ be an $n$-cycle in $A_1$ that is mapped non-trivially into $X_1$ and let $d$ be a non-trivial $0$-cycle in $A_2$. Consider the diagram
 \begin{diagram}
 H_n(A_1) \otimes H_0(A_2) & \rInto & H_n(A_1 \times A_2)\\
 \dTo && \dTo\\
 H_n(X_1) \otimes H_0(X_2) & \rInto & H_n(X_1 \times X_2) \text{.}
 \end{diagram}
 where the rows are parts of the K\"unneth formula (see \cite[Theorem~3B.6]{hatcher01}) and the columns are the maps induced from the inclusions. The diagram commutes by naturality of the K\"unneth formula. The cycle $c \otimes d$ in the upper left maps non-trivially into the lower left which injects into the lower right. Hence it has non-trivial image in the lower right. Since the diagram commutes, it follows that its image in the upper right also has non-trivial image in the lower right, which is what we want.
\end{proof}

\begin{proof}[Proof of Proposition~\ref{prop:Y_essential_connectivity}]
 For cocompactness let $C_r \subseteq V$ be compact such that its $\PB_2(\calO_S)$-translates cover $V_r$. Let $\hat{C}_r \subseteq Y$ be the product of $r$ copies of $C_r$ (say at positions $q_1,\ldots,q_r$) and $\{z\}$ otherwise. We claim that the translates of $\hat{C}_r$ cover $Y_r$.

 Indeed, let $(x_q)_q \in Y_r$ be arbitrary. Since it lies in $Y^{(r)}$ there are at most $r$ positions where $[x_q] \ne [z]$. Using the action of $F$ we can achieve that these positions are (some of) $q_1$ to $q_r$. Now, since each $x_{q_i}$ lies in $V_r$, we can move it into $C_r$, using an element of the form $[\tree,g,\tree]$, without moving any of the other $x_q$. At all other coordinates $q$, i.e., where $[x_q]=[z]$, we can move $x_q$ to $z$ using the same method. Since all but finitely many $x_q$ were $z$ to begin with, we have moved $(x_q)_q$ into $\hat{C}_r$ in finitely many steps.

 For the second statement let $N = \abs{S}-1$, so we want to show that $(\tilde{H}_N(Y_r))_r$ is not essentially trivial. Let $k$ be such that the map $H_{\abs{S}-1}(V_k \to V_m)$ is non-trivial for every $m \ge k$. For arbitrary $m \ge k$ take $A_1 = V_k$, $A_2 = \prod_{\substack{q \in D\\q \ne 1/2}} \{z\}$, $X_1 = V_m$, and $X_2 = \prod_{\substack{q \in D\\q \ne 1/2}} V_m$. Then $A_1 \times A_2 \subseteq Y_k$ and $Y_m \subseteq X_1 \times X_2$ (on the infinite products we take the CW topology, not the product topology). By Lemma~\ref{lem:non-trivial_maps} the map $H_N(A_1 \times A_2 \to X_1 \times X_2)$ is non-trivial. But this factors through the map $H_N(Y_k \to Y_m)$ which is therefore non-trivial as well. This shows that $(H_N(Y_r))_r$ is not essentially trivial.
\end{proof}

\begin{theorem}
 The group $\Thomp{\barB_*(\calO_S)}$ is not of type~$\FP_{\abs{S}}$.
\end{theorem}

\begin{proof}
 The group acts on $Y$, which is contractible, with stabilizers of type~$\F_\infty$ (Proposition~\ref{prop:Y_stabilizers}). There is an invariant, cocompact filtration $(Y_r)_r$ which is not essentially $(\abs{S}-1)$-acyclic (Proposition~\ref{prop:Y_essential_connectivity}). We conclude using Brown's criterion.
\end{proof}

\begin{remark}
 As far as we can tell, none of the established methods in the literature can be used now to show that $\Thomp{B_*(\calO_S)}$ is not of type~$\FP_{\abs{S}}$. The kernel of the morphism $\Thomp{B_*(\calO_S)} \to \Thomp{\barB_*(\calO_S)}$ is very unlikely to be even finitely generated (or else one could apply \cite[Proposition~2.7]{bieri76} or the following exercise, see also \cite[Theorem~7.2.21]{geoghegan08}). Also the projection does not split (or else one could apply the retraction argument \cite[Proposition~4.1]{bux04}). For this reason we will now use the new methods established in Section~\ref{sec:relative_brown}, which can be regarded as a generalization of the retraction argument. We should mention that one \emph{can} deduce that $\Thomp{B_*(\calO_S)}$ is not finitely generated if $\abs{S}=1$, without using this new machinery.
\end{remark}

\begin{proof}[Proof of Theorem~\ref{thrm:borel_thomp_neg_fin_props}]
 We apply Theorem~\ref{thrm:relative_brown} to the inclusion homomorphism $B_2(\calO_S) \into \Thomp{\barB_*(\calO_S)}$ that takes $g$ to $[\lambda_1,g,\lambda_1]$ where $\lambda_1$ is a single caret. We take $Z$ to be the subspace of $Y$ consisting of points $(x_q)_q$ with $x_q = z$ for $q \ne 1/2$, so $Z$ is $B_2(\calO_S)$-cocompact. Our $\Thomp{\barB_*(\calO_S)}$-cocompact filtration of $Y$ is $(Y_r)_{r\in\N}$. Observe that $H_{\abs{S}-1}(Z \to Y_r)$ is not eventually trivial by the proof of Proposition~\ref{prop:Y_essential_connectivity}. Thus we can apply Theorem~\ref{thrm:relative_brown}.

 Since the inclusion $B_2(\calO_S) \into \Thomp{\barB_*(\calO_S)}$ clearly factors through $\Thomp{B_*(\calO_S)}$ we conclude that this group is not of type~$\FP_{\abs{S}-1}$.
\end{proof}

Combining Theorem~\ref{thrm:borel_thomp_pos_fin_props} and Theorem~\ref{thrm:borel_thomp_neg_fin_props}, we obtain:

\begin{theorem}\label{thrm:borel_thomp_full_fin_props}
 The group $\Thomp{B_*(\calO_S)}$ is of type~$\F_{\abs{S}-1}$ but not of type~$\FP_{\abs{S}}$. \qed
\end{theorem}

\section{Thompson groups for mock-symmetric groups}\label{sec:mock}

The groups discussed in this section are instances of what Davis, Januszkiewicz and Scott call ``mock reflection groups'' \cite{davis03}. These are groups generated by involutions, and act on associated cell complexes very much like Coxeter groups, with the only difference being that some of the generators may be ``mock reflections'' that do not fix their reflection mirror pointwise. Here we will only be concerned with one family of groups consisting of the minimal blow up of Coxeter groups of type $A_n$. These Coxeter groups are symmetric groups and so we call their blow ups \emph{mock symmetric groups}. For $n \in \N$ the mock symmetric group $\mock\symm_n$ is given by the presentation
\begin{align}
 \mock\symm_n = \gen{s_{i,j}, 1 \le i < j \le n \mid{}& s_{i,j}^2 = 1 \text{ for all } i, j\nonumber\\
 &s_{i,j}s_{k,\ell} = s_{k,\ell}s_{i,j} \text{ for } i < j < k < \ell\label{eq:mock_presentation}\\
 &s_{k,\ell}s_{i,j} = s_{k+\ell-j,k+\ell-i}s_{k,\ell} \text{ for }k \le i < j \le \ell}\text{.}\nonumber
\end{align}
We also set $\mock\symm_\infty = \varinjlim \mock\symm_n$. See Figure~\ref{fig:mock_relation} for a visualization of elements of $\mock\symm_n$, and a visualization of the last relation.

\begin{figure}[htb]
\centering
\begin{tikzpicture}[line width=0.8pt, scale=0.5]
   \node at (10,-2.5) {$=$};

   \draw
   (4,0) -- (6,-2)
   (6,0) -- (4,-2);
   \filldraw[color=white]
   (5,-1) circle (5pt);
   \draw
   (5,-1) circle (5pt);
   \filldraw
   (0,0) circle (1.5pt)   (2,0) circle (1.5pt)   (4,0) circle (1.5pt)   (6,0) circle (1.5pt)   (8,0) circle (1.5pt)   
   (0,-2) circle (1.5pt)   (2,-2) circle (1.5pt)   (4,-2) circle (1.5pt)   (6,-2) circle (1.5pt)   (8,-2) circle (1.5pt);

  \begin{scope}[yshift=-2cm]
   \draw
   (0,0) -- (8,-2)
   (2,0) -- (6,-2)
   (4,0) -- (4,-2)
   (6,0) -- (2,-2)
   (8,0) -- (0,-2);
   \filldraw[color=white]
   (4,-1) circle (5pt);
   \draw
   (4,-1) circle (5pt);
   \filldraw
   (0,-2) circle (1.5pt)   (2,-2) circle (1.5pt)   (4,-2) circle (1.5pt)   (6,-2) circle (1.5pt)   (8,-2) circle (1.5pt);
  \end{scope}

  \begin{scope}[xshift=12cm]
   \draw
   (0,0) -- (8,-2)
   (2,0) -- (6,-2)
   (4,0) -- (4,-2)
   (6,0) -- (2,-2)
   (8,0) -- (0,-2);
   \filldraw[color=white]
   (4,-1) circle (5pt);
   \draw
   (4,-1) circle (5pt);
   \filldraw
   (0,0) circle (1.5pt)   (2,0) circle (1.5pt)   (4,0) circle (1.5pt)   (6,0) circle (1.5pt)   (8,0) circle (1.5pt)   
   (0,-2) circle (1.5pt)   (2,-2) circle (1.5pt)   (4,-2) circle (1.5pt)   (6,-2) circle (1.5pt)   (8,-2) circle (1.5pt);
  \end{scope}

  \begin{scope}[yshift=-2cm, xshift=12cm]
   \draw
   (2,0) -- (4,-2)
   (4,0) -- (2,-2);
   \filldraw[color=white]
   (3,-1) circle (5pt);
   \draw
   (3,-1) circle (5pt);
   \filldraw
   (0,-2) circle (1.5pt)   (2,-2) circle (1.5pt)   (4,-2) circle (1.5pt)   (6,-2) circle (1.5pt)   (8,-2) circle (1.5pt);
  \end{scope}
\end{tikzpicture}
\caption[Mock relation]{The relation $s_{i,j}s_{k,\ell} = s_{k,\ell} s_{k+\ell-j,k+\ell-i}$ of $\mock\symm_n$ in the case $i = 3$, $j=4$, $k=1$, $\ell=5$, $n = 5$.}
\label{fig:mock_relation}
\end{figure}
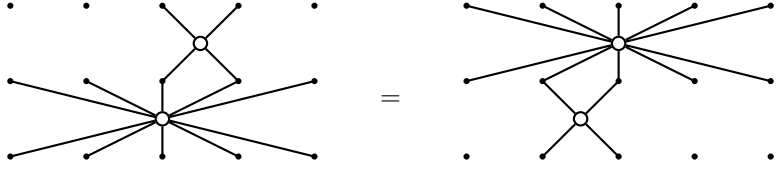

Let $\bar{s}_{i,j} \in \symm_n$ be the involution $(i\ j) ((i+1)\ (j-1)) \cdots (\floor{\frac{i+j}{2}}\ \ceil{\frac{i+j}{2}})$ (this is the longest element in the Coxeter group generated by $(i\ i+1), \ldots, (j-1\ j)$). Taking $s_{i,j}$ to $\bar{s}_{i,j}$ defines a surjective homomorphism $\rho_n \colon \mock\symm_n \to \symm_n$. We define cloning maps $\clone^n_k \colon \mock\symm_n \to \mock\symm_{n+1}$ by first defining them on the generators:
\begin{equation}
\label{eq:cloning_mock}
(s_{i,j}) \clone^n_k = \left\{
\begin{array}{ll}
s_{i,j} & \text{for }j < k\\
s_{i,j+1}s_{k,k+1} & \text{for }i \le k \le j\\
s_{i+1,j+1} & \text{for }k < i\text{.}
\end{array}
\right.
\end{equation}
Now we extend $\clone^n_k$ to a map $\mock\symm_n \to \mock\symm_{n+1}$ as in the paragraph leading up to Lemma~\ref{lem:extend_zs-products}. See Figure~\ref{fig:mock_cloning} for an example of cloning.

\begin{figure}[htb]
\centering
\begin{tikzpicture}[line width=0.8pt, scale=0.5,yscale=-1]
   \node at (9,-2.5) {$=$};

  \begin{scope}[yshift=-1cm]
   \draw
   (3,0) -- (4,-1)
   (5,0) -- (4,-1);
   \filldraw
   (3,0) circle (1.5pt) (5,0) circle (1.5pt)
   (0,-1) circle (1.5pt)   (2,-1) circle (1.5pt)   (4,-1) circle (1.5pt)   (6,-1) circle (1.5pt);
   (0,-1) circle (1.5pt)   (2,-1) circle (1.5pt)   (4,-1) circle (1.5pt)   (6,-1) circle (1.5pt);
   \end{scope}

  \begin{scope}[yshift=-2cm]
   \draw
   (0,0) -- (6,-2)
   (2,0) -- (4,-2)
   (4,0) -- (2,-2)
   (6,0) -- (0,-2);
   \filldraw[color=white]
   (3,-1) circle (5pt);
   \draw
   (3,-1) circle (5pt);
   \filldraw
   (0,-2) circle (1.5pt)   (2,-2) circle (1.5pt)   (4,-2) circle (1.5pt)   (6,-2) circle (1.5pt);
  \end{scope}

  \begin{scope}[xshift=12cm,yshift=-2cm]
   \draw
   (0,0) -- (8,-2)
   (2,0) -- (6,-2)
   (4,0) -- (4,-2)
   (6,0) -- (2,-2)
   (8,0) -- (0,-2);
   \filldraw[color=white]
   (4,-1) circle (5pt);
   \draw
   (4,-1) circle (5pt);
   \filldraw
   (0,0) circle (1.5pt)   (2,0) circle (1.5pt)   (4,0) circle (1.5pt)   (6,0) circle (1.5pt)   (8,0) circle (1.5pt)   
   (0,-2) circle (1.5pt)   (2,-2) circle (1.5pt)   (4,-2) circle (1.5pt)   (6,-2) circle (1.5pt)   (8,-2) circle (1.5pt);
  \end{scope}

  \begin{scope}[yshift=0cm, xshift=12cm]
   \draw
   (4,0) -- (6,-2)
   (6,0) -- (4,-2);
   \filldraw[color=white]
   (5,-1) circle (5pt);
   \draw
   (5,-1) circle (5pt);
   \filldraw
   (4,0) circle (1.5pt)   (6,0) circle (1.5pt);
  \end{scope}

  \begin{scope}[yshift=-4cm, xshift=12cm]
   \draw
   (2,0) -- (3,-1)
   (4,0) -- (3,-1);
   \filldraw
   (2,0) circle (1.5pt)   (4,0) circle (1.5pt)
   (3,-1) circle (1.5pt);
  \end{scope}
\end{tikzpicture}
\caption[Mock cloning]{The relation $s_{1,4}\lambda_3 = \lambda_2s_{1,5} s_{3,4}$ of $\Fmonoid \bowtie \mock\symm_\infty$.}
\label{fig:mock_cloning}
\end{figure}

\begin{proposition}\label{prop:mock_clone}
 The above data define a cloning system on $\mock\symm_*$.
\end{proposition}

\begin{proof}
 Note first that \eqref{eq:mock_presentation} is a presentation for $\mock\symm_n$ as a monoid because all the generators are involutions by the first relation. Following the advice from Remark~\ref{rmk:axioms_via_pres}, we will apply Lemma~\ref{lem:extend_zs-products} with this presentation rather than the trivial presentation used in Proposition~\ref{prop:BZS_existence}.

 We have to verify conditions coming from relations of $\Fmonoid$ and conditions coming from relations of $\mock\symm_n$, after which the proof proceeds as that of Proposition~\ref{prop:BZS_existence}. For the relations of $\Fmonoid$ we must verify the conditions~\eqref{item:fcs_product_of_clonings} (product of clonings) and~\eqref{item:fcs_compatibility} (compatibility)
 \begin{align}
 (s_{i,j})\clone_\ell \clone_k &= (s_{i,j})\clone_k \clone_{\ell + 1}&&\text{for }k < \ell\text{ and }i<j\label{eq:mock_product_of_clonings}\\
 \rho((s_{i,j})\clone_k) & = (\rho(s_{i,j}))\symmclone_k &&\text{for }i < j\text{.}\label{eq:mock_compatibility}
 \end{align}
(Note that we verified \eqref{item:fcs_compatibility} for all $i$, which is not technically necessary; see the remark after Observation~\ref{obs:sync}).
 For the relations of $\mock\symm_n$ we have to check that $\rho$ is a well defined homomorphism, and check that the following equations, standing in for~\eqref{item:cs_cloning_a_product} (cloning a product), are satisfied:
 \begin{align}
 (s_{i,j})\clone_{\rho(s_{k,\ell})p}(s_{k,\ell})\clone_p &=  (s_{k,\ell})\clone_{\rho(s_{i,j})p}(s_{i,j})\clone_p &&\text{for }i < j < k < \ell\label{eq:mock_cloning_commutator}\\
 (s_{k+\ell-j,k+\ell-i})\clone_{\rho(s_{k,\ell})p}(s_{k,\ell})\clone_p &= (s_{k,\ell})\clone_{\rho(s_{i,j})p}(s_{i,j})\clone_p&&\text{for } k \le i < j \le \ell\text{.}\label{eq:mock_cloning_mock_relation}
 \end{align}
 Note that the conditions coming from the relations $s_{i,j}^2 = 1$ are vacuous.

 Condition \eqref{eq:mock_product_of_clonings} is easy to check if $k< i$ or $\ell > j$ so we consider the situation where $i \le k < \ell \le j$. In this case we have
 \begin{multline*}
 (s_{i,j})\clone_\ell\clone_k = (s_{i,j+1}s_{\ell,\ell+1})\clone_k  = (s_{i,j+1})\clone_k(s_{\ell,\ell+1})\clone_k =\\
  s_{i,j+2} s_{k,k+1} s_{\ell+1,\ell+2} = s_{i,j+2} s_{\ell+1,\ell+2} s_{k,k+1}=\\
    (s_{i,j+1})\clone_{\ell+1}(s_{k,k+1})\clone_{\ell+1} = (s_{i,j+1}s_{k,k+1})\clone_{\ell+1} = (s_{i,j})\clone_k\clone_{\ell+1}
 \end{multline*}
 since $\rho(s_{k,k+1})(\ell+1) = (\ell+1)$, $\rho(s_{\ell,\ell+1})k = k$ and $s_{k,k+1}$ and $s_{\ell+1,\ell+2}$ commute.

 Condition \eqref{eq:mock_compatibility} amounts to showing that
 \[
 (\bar{s}_{i,j})\symmclone_k = 
 \left\{
 \begin{array}{ll}
 \bar{s}_{i+1,j+1}&k<i\\
 \bar{s}_{i,j+1}\bar{s}_{k,k+1}&i \le k \le j\\
 \bar{s}_{i,j}&k>j\text{.}
 \end{array}
 \right.
 \]
 The cases $k<i$ and $k > j$ are clear. For the remaining case we first note that
 \[
 \bar{s}_{i,j+1}\bar{s}_{k,k+1}(m) = \tau_{i+j-k} \bar{s}_{i,j} \pi_{k}(m) = ((\bar{s}_{i,j})\symmclone_k)(m)
 \]
 for $m \ne k,k+1$ (which is also the same as $\bar{s}_{i,j+1}(m)$). Here $\tau_k$ and $\pi_k$ are as in Example~\ref{ex:symm_gps}. Finally one checks that
 \[
 \bar{s}_{i,j+1}\bar{s}_{k,k+1}(k) = i+j-k = (\bar{s}_{i,j})\symmclone_k(k)
 \]
 and that
 \[ 
 \bar{s}_{i,j+1}\bar{s}_{k,k+1}(k+1) = i+j-k+1 = (\bar{s}_{i,j})\symmclone_k(k+1)\text{.}
 \]

 That $\rho$ is a well defined homomorphism amounts to saying that the defining relations of $\mock\symm_n$ hold in $\symm_n$ with $s_{i,j}$ replaced by $\bar{s}_{i,j}$, which they do.

 Condition \eqref{eq:mock_cloning_commutator} is also easy to check unless $i \le p \le j$ or $k \le p \le \ell$. We treat the case $i \le p \le j$, the other remaining case being similar. We have
 \begin{multline*}
 (s_{i,j})\clone_{\rho(s_{k,\ell})p}(s_{k,\ell})\clone_p = s_{i,j+1}s_{p,p+1}s_{k+1,\ell+1} =\\
  s_{k+1,\ell+1} s_{i,j+1} s_{p,p+1} = (s_{k,\ell})\clone_{\rho(s_{i,j})p}(s_{i,j}) \clone_p\text{.}
 \end{multline*}

 Finally, the interesting case of condition \eqref{eq:mock_cloning_mock_relation} is when $i \le p \le j$. We have
 \begin{multline*}
 (s_{k,\ell})\clone_{\rho(s_{i,j})p}(s_{i,j})\clone_p = s_{k,\ell+1}s_{i+j-p,i+j-p+1}s_{i,j+1}s_{p,p+1} = s_{k,\ell+1}s_{i,j+1}\\
 = s_{k+\ell-j,k+\ell-i+1} s_{k+\ell-p,k+\ell-p+1} s_{k,\ell+1}s_{p,p+1} = (s_{k+\ell-j,k+\ell-i})\clone_{\rho(s_{k,\ell})p}(s_{k,\ell})\clone_p
 \end{multline*}
 using the defining relations of $\mock\symm_n$ several times.

\end{proof}

As a consequence we get

\begin{theorem}\label{thm:existence_thompson_mock}
 There is a generalized Thompson group $\Thomp{\mock\symm_*}$ which contains all the $\mock\symm_n$ and canonically surjects onto $V$. We denote it $\Vmock$.
\end{theorem}

\begin{conjecture}\label{conj:Vmock_conj}
 $\Vmock$ is of type~$\F_\infty$.
\end{conjecture}

Since each $\mock\symm_n$ is of type~$\F_\infty$ \cite[Section~4.7, Corollary~3.5.4]{davis03}, to prove the conjecture it suffices to show that the the cloning system is properly graded and that the connectivity of the complexes $\dlkmodel{\mock\symm_*}{n}$ goes to infinity as $n$ goes to infinity.

\section{Thompson groups for loop braid groups}\label{sec:loop}

Our next example of a cloning system comes from the family of \emph{loop braid groups} $\LB_n$, also known as groups $\SAut_n$ of \emph{symmetric automorphisms} of free groups, or as \emph{braid-permutation groups} (see \cite{damiani16} for an overview). This will produce a generalized Thompson group $\Vloop$ that contains both $\Vbr$ and $V$ as subgroups. There is also a pure version of this cloning system, using the \emph{pure loop braid groups}, which we will discuss as well, yielding a group $\Floop$.

We first describe the family of groups in terms of free group automorphisms. Fix a set of generators $\{x_1,\dots,x_n\}$ for $F_n$, and call an automorphism $\phi\in \Aut(F_n)$ \emph{symmetric} if for every $1\le i\le n$ there exists $1\le j\le n$ such that $\phi(x_i)$ is conjugate to $x_j$. If every $\phi(x_i)$ is even conjugate to $x_i$, call $\phi$ \emph{pure symmetric}. The group of symmetric automorphisms of $F_n$ is denoted $\SAut_n$, and the group of pure symmetric automorphisms is denoted $\PSAut_n$. The latter is also denoted by $PLB_n$, for \emph{pure loop braid group}. The reader is cautioned that in the literature ``symmetric'' sometimes allows for generators to map to conjugates of \emph{inverses of} generators, but we do not allow this.

The $\LB_n$ fit into a directed system. The map $\iota_{n,n+1} \colon \LB_n \into \LB_{n+1}$ is given by sending the automorphism $\phi$ of $F_n$ to the automorphism of $F_{n+1}$ that does nothing to the new generator and otherwise acts like $\phi$. This restricts to $\PLB_n$ as well, and so we have directed systems $\LB_*$ and $\PLB_*$.

Our presentation for $\LB_n=\SAut_n$ will be taken from~\cite{fenn97}. The generators are as follows, for $1\le i\le n$.
\begin{align*}
 &\beta_i \colon \left\{\begin{array}{lll} x_i & \mapsto x_{i+1} \\ x_{i+1} & \mapsto x_{i+1}^{-1} x_i x_{i+1} \\ x_j & \mapsto x_j & (j\neq i,i+1) \end{array}\right. \\
 &\sigma_i \colon \left\{\begin{array}{lll} x_i & \mapsto x_{i+1} \\ x_{i+1} & \mapsto x_i \\ x_j & \mapsto x_j & (j\neq i,i+1) \end{array}\right.
\end{align*}

\medskip

The $\beta_i$ together with the $\sigma_i$ generate $\SAut_n$. The $\beta_i$ by themselves generate a copy of $B_n$ in $\SAut_n$, and the $\sigma_i$ generate a copy of $\symm_n$. As seen in \cite{fenn97}, defining relations for $\SAut_n$ are as follows (with $1\le i\le n-1$):
\allowdisplaybreaks
\begin{align*}
 \beta_i \beta_j &= \beta_j \beta_i \hfill (|i-j|>1)\\
 \beta_i \beta_{i+1} \beta_i &= \beta_{i+1} \beta_i \beta_{i+1}\\
 \sigma_i^2 &= 1\\
 \sigma_i \sigma_j &= \sigma_j \sigma_i \hfill (|i-j|>1)\\
 \sigma_i \sigma_{i+1} \sigma_i &= \sigma_{i+1} \sigma_i \sigma_{i+1}\\
 \beta_i \sigma_j &= \sigma_j \beta_i \hfill (|i-j|>1)\\
 \sigma_i \sigma_{i+1} \beta_i &= \beta_{i+1} \sigma_i \sigma_{i+1}\\
 \beta_i \beta_{i+1} \sigma_i &= \sigma_{i+1} \beta_i \beta_{i+1} \text{.}
\end{align*}
\allowdisplaybreaks[0]

\medskip

This is a group presentation, and it becomes a monoid presentation after adding generators $\beta_i^{-1}$ with relations $\beta_i \beta_i^{-1} = \beta_i^{-1} \beta_i = 1$.

Since we already have cloning systems on $\symm_*$ (from Example~\ref{ex:symm_gps}) as well as on $B_*$ (from \cite{brin07}), we already know how the cloning system on $\LB_*=\SAut_*$ should be defined. The only thing to check is that it is actually well defined.

The homomorphism $\rho_n \colon \LB_n \to S_n$ just takes $\beta_i$ as well as $\sigma_i$ to $\sigma_i \in \symm_n$. This is easily seen to be well defined.

The cloning maps are defined as they are defined for the symmetric groups and braid groups respectively: for $\varepsilon \in \{\pm 1\}$ this means that
\begin{align}
 (\beta_i^\varepsilon)\clone_k \defeq \left\{\begin{array}{ll} \beta_{i+1}^\varepsilon &\text{ if } k<i \\
 \beta_i^\varepsilon \beta_{i+1}^\varepsilon &\text{ if } k=i \\
 \beta_{i+1}^\varepsilon \beta_i^\varepsilon &\text{ if } k=i+1 \\
 \beta_i^\varepsilon &\text{ if } k>i+1
 \end{array}
 \right.\label{eq:loop_braid_braid_clone}\\
 (\sigma_i)\clone_k \defeq \left\{\begin{array}{ll} \sigma_{i+1} &\text{ if } k<i \\
 \sigma_i \sigma_{i+1} &\text{ if } k=i \\
 \sigma_{i+1} \sigma_i &\text{ if } k=i+1 \\
 \sigma_i &\text{ if } k>i+1
 \end{array}
 \right.\label{eq:loop_braid_transposition_clone}
\end{align}

\begin{lemma}\label{lem:loop_bd_cloning}
The above data $\rho_*$ and $\clone^*_k$ define cloning systems on $\LB_*$ and on $\PLB_*$.
\end{lemma}

\begin{proof}
 We already noted that $\rho$ is a well defined group homomorphism. We have to check~\eqref{item:cs_product_of_clonings} (product of clonings) and~\eqref{item:cs_compatibility} (compatibility) on generators of $\LB_n$. But since every generator is a generator of either $\symm_n$ or of $B_n$, each verification needed has been performed in establishing the cloning systems on either $\symm_*$ or $B_*$.

 It remains to check that cloning a relation is well defined, standing in for~\eqref{item:cs_cloning_a_product} (cloning a product). Again, the relations involving only elements of $\symm_n$ or $B_n$ are already verified. This leaves the last three kinds of relations.

 For the first relation we have to check that
 \[
 (\beta_i)\clone_{\rho(\sigma_j)k} (\sigma_j)\clone_k = (\beta_i)\clone_{\sigma_j k} (\sigma_j)\clone_k = (\sigma_j)\clone_{\sigma_i k} (\beta_i)\clone_k = (\sigma_j)\clone_{\rho(\beta_i)k} (\beta_i)\clone_k
 \]
 which is easy to do case by case. For the other two relations we must show that
 \begin{align*}
 (\sigma_i)\clone_{(i~i+2~i+1)k} (\sigma_{i+1})\clone_{(i~i+1)k} (\beta_i)\clone_k &=  (\beta_{i+1})\clone_{(i~i+1~i+2)k} (\sigma_i)\clone_{(i+1~i+2)k} (\sigma_{i+1})\clone_k\\
 (\beta_i)\clone_{(i~i+2~i+1)k} (\beta_{i+1})\clone_{(i~i+1)k} (\sigma_i)\clone_k &= (\sigma_{i+1})\clone_{(i~i+1~i+2)k} (\beta_i)\clone_{(i+1~i+2)k} (\beta_{i+1})\clone_k
 \end{align*}
 which can be treated formally equivalently as long as we do not use either of the relations $\sigma_i^2 = 1$ or $\beta_i\beta_i^{-1} = \beta_i^{-1}\beta_i = 1$. The cases $k < i$ and $k > i+2$ are easy. For $k=i$ we apply only mixed relations to find
 \begin{multline*}
 (\sigma_i)\clone_{i+2} (\sigma_{i+1})\clone_{i+1} (\beta_i)\clone_i = \sigma_i \sigma_{i+1} \sigma_{i+2} \beta_i \beta_{i+1}\\
 = \beta_{i+1} \beta_{i+2} \sigma_i \sigma_{i+1} \sigma_{i+2} = (\beta_{i+1})\clone_{i+1} (\sigma_i)\clone_i (\sigma_{i+1})\clone_i\text{.}
 \end{multline*}
 Similarly for $k = i+1$ we get
 \begin{multline*}
 (\sigma_i)\clone_i (\sigma_{i+1})\clone_i (\beta_i)\clone_{i+1} = \sigma_i \sigma_{i+1} \sigma_{i+2} \beta_{i+1} \beta_i\\
 = \beta_{i+2} \beta_{i+1} \sigma_i \sigma_{i+1} \sigma_{i+2} =(\beta_{i+1})\clone_{i+2} (\sigma_i)\clone_{i+2} (\sigma_{i+1})\clone_{i+1} \text{.}
 \end{multline*}
 Lastly for $k = i+2$ we first apply a braid relation and then use the mixed relations to get
 \begin{multline*}
 (\sigma_i)\clone_{i+1} (\sigma_{i+1})\clone_{i+2} (\beta_i)\clone_{i+2} = \sigma_{i+1} \sigma_i \sigma_{i+2} \sigma_{i+1} \beta_i = \sigma_{i+1} \sigma_{i+2} \sigma_{i} \sigma_{i+1} \beta_i\\
 =  \beta_{i+2} \sigma_{i+1} \sigma_i \sigma_{i+2} \sigma_{i+1} = (\beta_{i+1})\clone_i (\sigma_i)\clone_{i+1} (\sigma_{i+1})\clone_{i+2} \text{.}
 \end{multline*}
 
 Finally, that the cloning system on $\LB_*$ restricts to one on $\PLB_*$ is straightforward.
\end{proof}

\begin{theorem}\label{thm:existence_thompson_loop}
 There are generalized Thompson groups
 \[
 \Vloop\defeq\Thomp{\LB_*} \quad \text{and}\quad \Floop\defeq\Thomp{\PLB_*}
 \]
 containing the loop braid groups and the pure loop braid groups, respectively. The group $\Vloop$ canonically surjects onto $V$, and the group $\Floop$ canonically surjects onto $F$.
\end{theorem}

The group $\LB_n$ is known to be of type~$\F_\infty$, for instance it acts properly cocompactly on the contractible space of marked cactus graphs \cite{collins89}. For this reason understanding the finiteness properties of $\Vloop$ and $\Floop$ amounts to showing that the cloning systems are properly graded, and understanding the connectivity of $\dlkmodel{\LB_*}{n}$ and $\dlkmodel{\PLB_*}{n}$. We expect that these should be increasingly highly connected and thus:

\begin{conjecture}\label{conj:Vloop_Floop_conj}
 $\Vloop$ and $\Floop$ are of type~$\F_\infty$.
\end{conjecture}

We do not attempt to prove this conjecture here. However, we end by sketching a more geometric viewpoint of these cloning systems, which could be useful in the future. To do so, we will view $\LB_n$ as a group of motions of loops (which is where the name comes from); see \cite{baez07}, \cite{brendle13} and \cite{wilson12}. Let $\R^3$ be Euclidean~$3$-space, and define a \emph{loop} $\gamma$ to be a smooth, unknotted, oriented embedded copy of the circle $S^1$ in $\R^3$. Now fix a set $L$ of $n$ pairwise disjoint, unlinked loops in $\R^3$, and let $C_n \defeq \coprod\limits_{\gamma \in L}\gamma$.  A \emph{motion} of $C_n$ is a path of diffeomorphisms $f_t\in\Diff(\R^3)$ for $t\in[0,1]$ such that $f_0$ is the identity and $f_1$ stabilizes $C_n$ set-wise, preserving orientations of the loops. Two motions $f_{t,0}$ and $f_{t,1}$ are considered equivalent if they are smoothly isotopic via an isotopy $f_{t,s}$ with $f_{0,s}$ and $f_{1,s}$ setwise stabilizing $C_n$. If $f_1$ also stabilizes each $\gamma \in L$ then the motion $f_t$ is a \emph{pure motion}. These constructions and the above ones yield isomorphic groups, that is to say $\LB_n$ is the group of motions, and $\PLB_n$ is the group of pure motions. This is explained, e.g., in \cite{goldsmith81} and \cite[Section~3]{wilson12}. One should picture $\sigma_i$ as the motion in which the $i$th and $(i+1)$st loops move around each other and take each other's old spots. Then $\beta_i$ is similar, except that during the motion the $(i+1)$st loop passes through the $i$th instead of around. See Figure~\ref{fig:loop_braid_gens} for an idea.

\begin{figure}[htb]\centering
\begin{tikzpicture}[line width=0.8pt]
  \draw (0,0) circle (0.5cm)   (3,0) circle (0.5cm);
  \draw[->] (0.5,-0.5) to [out=-30, in=210] (3.5,-0.5);
  \draw[->] (2.5,0.5) to [out=150, in=30] (-0.5,0.5);
  \node at (1.5,-1.5) {$\sigma_i$}; \node at (-0.75,0) {$i$}; \node at (4,0) {$i+1$};

  \begin{scope}[xshift=6.5cm]
   \draw[->] (2.25,0) -- (0,0);
   \draw[white,line width=4pt] (0,0) circle (0.5cm)   (3,0) circle (0.5cm);
   \draw (0,0) circle (0.5cm)   (3,0) circle (0.5cm);
   \draw[->] (0.5,-0.5) to [out=-30, in=210] (3.5,-0.5);
   \draw[->] (0.5,0.5) to [out=30, in=-210] (3.5,0.5);
  \node at (1.5,-1.5) {$\beta_i$}; \node at (-0.75,0) {$i$}; \node at (4,0) {$i+1$};
  \end{scope}
   
\end{tikzpicture}
\caption{Generators of $\LB_n$.}
\label{fig:loop_braid_gens}
\end{figure}
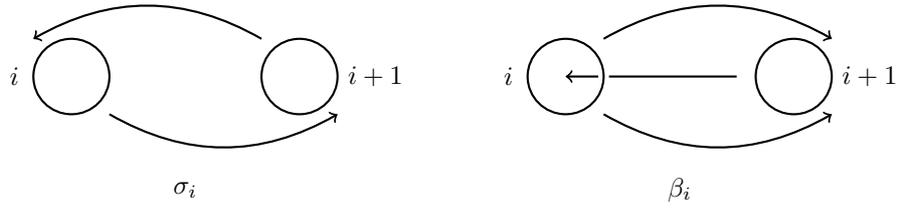

There is a bit of inconsistency in the literature: all that we have described here is as in, e.g., \cite{fenn97}, but in, e.g., \cite{brendle13}, instead of the generators $\beta_i$ their inverses are used (called $\rho_i$ there), and then the relevant relations look slightly different.

In \cite{baez07} there are some helpful diagrams, analogous to strand diagrams for braids, illustrating elements of $\LB_n$. The pictures are four-dimensional, and show one loop passing through another in a sort of movie. Using a bit of artistic license, we can draw similar diagrams to demonstrate cloning; see Figure~\ref{fig:tube_cloning}.

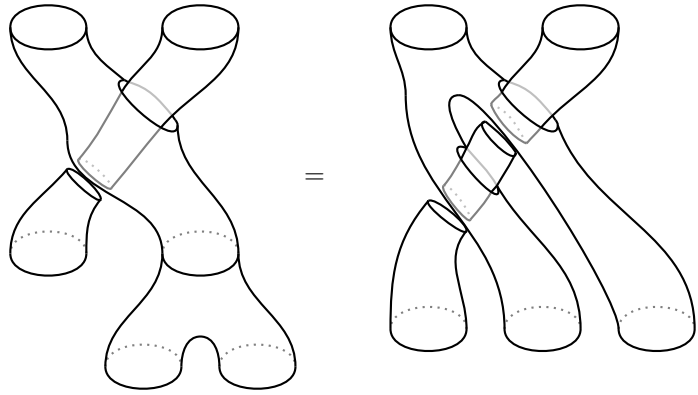
\begin{figure}[htb]\centering
\begin{tikzpicture}[line width=0.8pt]
  \coordinate (a) at (0,0); \coordinate (b) at (1,0); \coordinate (c) at (2,0); \coordinate (d) at (3,0); \coordinate (e) at (0,-3); \coordinate (f) at (1,-3); \coordinate (g) at (2,-3); \coordinate (h) at (3,-3); \coordinate (i) at (1.55,-0.725); \coordinate (j) at (2.13,-1.17); \coordinate (k) at (1.47,-0.85); \coordinate (l) at (0.9,-1.75); \coordinate (m) at (1.98,-1.33); \coordinate (n) at (1.33,-2.13); \coordinate (o) at (0.75,-1.5); \coordinate (p) at (1.5,-0.7); \coordinate (q) at (2.2,-1.3); \coordinate (r) at (2.17,-1.23); \coordinate (s) at (0.75,-1.88); \coordinate (t) at (1.19,-2.28); \coordinate (u) at (1.25,-4.5); \coordinate (v) at (2.25,-4.5); \coordinate (w) at (2.75,-4.5); \coordinate (x) at (3.75,-4.5);  \coordinate (y) at (2.5,-4.1);

  \draw (a) to[out=-90, in=-90] (b) to[out=90, in=90] (a); \draw (c) to[out=-90, in=-90] (d) to[out=90, in=90] (c); \draw (e) to[out=-90, in=-90] (f); \draw[gray,dotted] (f) to[out=90, in=90] (e); \draw (g) to[out=-90, in=-90] (h); \draw[gray,dotted] (h) to[out=90, in=90] (g);
  \draw[lightgray] (i) to[out=-15, in=130] (j);
  \draw[lightgray,dotted] (l) to[out=30, in=80, looseness=.4] (n); \draw[gray] (l) to[out=-150, in=-100, looseness=.4] (n);
  \draw (c) to[out=-90, in=60] (k); \draw[gray] (k) to[out=-120, in=50] (l); \draw (s) to[out=-130, in=90] (e);
  \draw (d) to[out=-90, in=50] (m); \draw[gray] (m) to[out=-130, in=50] (n); \draw (t) to[out=-130, in=90] (f);
  \draw (a) to[out=-90, in=90] (o);
  \draw (b) to[out=-90, in=140] (p); \draw (p) to[out=180, in=-90] (q) to[out=90, in=-50] (r); \draw (p) to[out=-15, in=160] (i); \draw (j) to[out=-50, in=130] (r);
  \draw (o) to[out=-90, in=90] (g); \draw (q) to[out=-90, in=90] (h);
  \draw (s) to[out=30, in=80, looseness=.4] (t); \draw (s) to[out=-150, in=-100, looseness=.4] (t);
  \draw[gray,dotted] (u) to[out=90, in=90] (v); \draw[gray,dotted] (w) to[out=90, in=90] (x); \draw (g) to[out=-90, in=90] (u); \draw (h) to[out=-90, in=90] (x); \draw (u) to[out=-90, in=-90] (v); \draw (w) to[out=-90, in=-90] (x); \draw (v) to[out=90, in=180] (y) to[out=0, in=90] (w);
  
  \node at (4,-2) {$=$};
  
  \begin{scope}[xshift=5cm]
   \coordinate (a) at (0,0); \coordinate (b) at (1,0); \coordinate (c) at (2,0); \coordinate (d) at (3,0); \coordinate (e) at (0,-4); \coordinate (f) at (1,-4); \coordinate (i) at (1.55,-0.725); \coordinate (j) at (2.13,-1.17); \coordinate (k) at (1.47,-0.85); \coordinate (l) at (1.32,-1.02); \coordinate (m) at (1.98,-1.33); \coordinate (n) at (1.79,-1.54); \coordinate (o) at (0.2,-0.9); \coordinate (p) at (1.5,-0.7); \coordinate (q) at (2.2,-1.3); \coordinate (r) at (2.17,-1.23); \coordinate (s) at (0.5,-2.3); \coordinate (t) at (1,-2.7); \coordinate (u) at (1.5,-4); \coordinate (v) at (2.5,-4); \coordinate (w) at (3,-4); \coordinate (x) at (4,-4);  \coordinate (y) at (0.9,-0.9);
   \coordinate (A) at (1,-1.6); \coordinate (B) at (1.4,-2.1); \coordinate (C) at (1.25,-1.25); \coordinate (D) at (1.65,-1.66); \coordinate (E) at (0.95,-1.7); \coordinate (F) at (1.35,-2.15); \coordinate (G) at (0.7,-2.1); \coordinate (H) at (1.05,-2.55);
  \path[name path=AB] (A) to[out=160, in=-70, looseness=2] (B);
  \path[name path=CG] (C) -- (G);
  \path[name path=DH] (D) to[out=-110, in=58] (H);
  \path[name intersections={of=AB and CG}];
  \coordinate (E)  at (intersection-1);
  \path[name intersections={of=AB and DH}];
  \coordinate (F)  at (intersection-1);

  \draw[lightgray,dotted] (G) to[out=30, in=80, looseness=0.4] (H);\draw (a) to[out=-90, in=-90] (b) to[out=90, in=90] (a); \draw (c) to[out=-90, in=-90] (d) to[out=90, in=90] (c); \draw (e) to[out=-90, in=-90] (f); \draw[gray,dotted] (f) to[out=90, in=90] (e);
  \draw[lightgray] (i) to[out=-15, in=130] (j);
  \draw[lightgray,dotted] (l) to[out=-20, in=120] (n); \draw[gray] (l) to[out=-70, in=160, looseness=0.7] (n);
  \draw (c) to[out=-90, in=60] (k); \draw[gray] (k) to[out=-120, in=60] (l); \draw (s) to[out=-130, in=90] (e);
  \draw (d) to[out=-90, in=50] (m); \draw[gray] (m) to[out=-130, in=50] (n); \draw (t) to[out=-130, in=90] (f);
  \draw[gray] (E) to[out=-120, in=50] (G); \draw[gray] (F) to[out=-120, in=50] (H);
  \draw (a) to[out=-90, in=90] (o);
  \draw (b) to[out=-90, in=140] (p); \draw (p) to[out=180, in=-90] (q) to[out=90, in=-50] (r); \draw (p) to[out=-15, in=160] (i); \draw (j) to[out=-50, in=130] (r);
  \draw (s) to[out=30, in=80, looseness=.4] (t); \draw (s) to[out=-150, in=-100, looseness=.4] (t);
  \draw[gray,dotted] (u) to[out=90, in=90] (v); \draw[gray,dotted] (w) to[out=90, in=90] (x); \draw (o) to[out=-90, in=90] (u); \draw (q) to[out=-70, in=90] (x); \draw (u) to[out=-90, in=-90] (v); \draw (w) to[out=-90, in=-90] (x); \draw (v) to[out=90, in=-60] (B); \draw (A) to[out=120, in=180] (y); \draw (y) to[out=0, in=90,looseness=0.3] (w);
  \draw[lightgray] (A) to[out=-50, in=120] (B);
  \draw (A) to[out=160, in=-70, looseness=2] (B); \draw (C) to[out=180, in=-90] (D); \draw (C) to[out=-40, in=125] (D);
  \draw (C) to[out=-155, in=60] (A) -- (E); \draw (D) to[out=-115, in=60] (B) -- (F);
  \draw[gray] (G) to[out=-150, in=-100, looseness=0.4] (H);
  
  \end{scope}

\end{tikzpicture}
\caption{An example of cloning, namely $(\beta_1)\clone_2^2 = \beta_2 \beta_1$. The picture shows $\beta_1 \lambda_2 = \lambda_1 \beta_2 \beta_1$. The vertical direction is time, while the missing spatial direction is indicated by breaking the surfaces; see \cite[p.~717]{baez07} for a detailed explanation.}
\label{fig:tube_cloning}
\end{figure}

Alternatively we can draw cloning using the welded braid diagrams from \cite{fenn97}. See Figure~\ref{fig:weld_cloning}.

\begin{figure}[htb]\centering
\begin{tikzpicture}[line width=0.8pt]
  \draw (2,0) to[out=-90, in=90] (1,-3);
  \draw[white, line width=4pt] (0,0) to[out=-90, in=90] (2,-3);
  \draw (0,0) to[out=-90, in=90] (2,-3)   (1,0) to[out=-90, in=90] (0,-3);
  \filldraw (0.65,-1.2) circle (3pt);
  \draw (1.5,-3.5) -- (2,-3) -- (2.5,-3.5);
  
  \node at (3,-1.5) {$=$};
  
  \begin{scope}[xshift=4cm]
   \draw (3,0) to[out=-90, in=90] (1,-3);
   \draw[white, line width=4pt] (0,0) to[out=-90, in=90] (2,-3)   (1,0) to[out=-90, in=90] (3,-3);
   \draw (2,0) to[out=-90, in=90] (0,-3)   (0,0) to[out=-90, in=90] (2,-3)   (1,0) to[out=-90, in=90] (3,-3);
   \filldraw (1,-1.5) circle (3pt) (1.5,-1.06) circle (3pt);
   \draw (0,0) -- (0.5,0.5) -- (1,0);
  \end{scope}

\end{tikzpicture}
\caption{Another example of cloning, now using welded braid diagrams. We see that $\sigma_1 \beta_2 \lambda_3 = \lambda_1 \sigma_2 \sigma_1 \beta_3 \beta_2$.}
\label{fig:weld_cloning}
\end{figure}
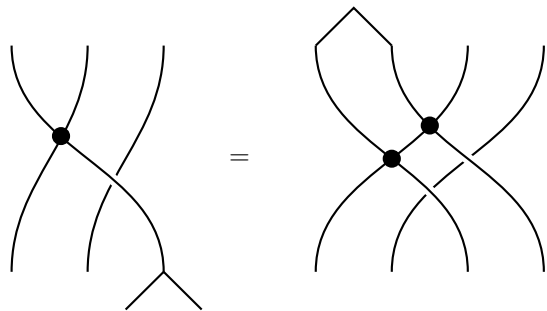

One might expect the descending links to be modeled on disjoint ``tubes'' in $3$-space with prescribed boundaries, or ``welded arcs'' of some sort. This is in analogy to the disjoint arcs in $2$-space with prescribed boundaries for descending links in the braid group case.

\newcommand{\etalchar}[1]{$^{#1}$}

\end{document}